\documentclass[11pt, oneside]{article}

\usepackage{amsmath,amsfonts,amssymb,amsthm,enumerate,mathrsfs,amscd}

\usepackage[top=15mm, bottom=15mm, left=44mm, right=44mm]{geometry}

\usepackage[unicode,breaklinks=true,colorlinks=true]{hyperref}

\usepackage[dvipsnames]{xcolor}

\numberwithin{equation}{section}
\theoremstyle{plain}
\newtheorem{theorem}{Theorem}[section]

\newtheorem{lemma}[theorem]{Lemma}
\newtheorem{proposition}[theorem]{Proposition}

\newtheorem{example}[theorem]{Example}

\theoremstyle{remark}
\newtheorem{remark}{Remark}[section]

\theoremstyle{definition}
\newtheorem{definition}{Definition}[section]

\newcommand{\bke}[1]{\left ( #1 \right )}

\newcommand{\norm}[1]{\left \| #1 \right \|}
\newcommand{\bka}[1]{{\langle #1 \rangle}}

\newcommand\al{\alpha}

\newcommand\e {\varepsilon}

\newcommand\De{\Delta}

\newcommand\Om{\Omega}
\newcommand{\R}{\mathbb{R}}

\newcommand{\N}{\mathbb{N}}

\renewcommand{\div}{\mathop{\mathrm{div}}\nolimits}

\newcommand{\pd}{\partial}
\newcommand{\nb}{\nabla}

\newcommand{\I}{\infty}

\newcommand{\EQN}[1]{\begin{equation*}\begin{split} #1 \end{split}\end{equation*}}
\newcommand{\bb}{\mathbf{b}}

\newcommand{\eps}{\varepsilon}
\newcommand{\loc}{\mathrm{loc}}

\title{Existence, uniqueness, and regularity results for elliptic equations
with drift terms in critical weak spaces
}
\author{Hyunseok Kim\thanks{%
Department of Mathematics, Sogang University, Seoul, 121-742, Korea (kimh@sogang.ac.kr).
The research of Kim was supported by Basic Science Research Program through the National Research Foundation of Korea (NRF) funded by the Ministry of Education (No. NRF-2016R1D1A1B02015245).}
\and Tai-Peng Tsai\thanks{%
Department of Mathematics, University of British Columbia, Vancouver, V6T 1Z2, Canada (ttsai@math.ubc.ca). The research of Tsai was supported in part by Natural Sciences and Engineering Research Council of Canada (NSERC) Grant 261356-13.}
}
\date{}

\begin{document}
\baselineskip 16pt

\maketitle

\begin{abstract}
We consider Dirichlet problems for linear elliptic equations of second order in divergence form on a bounded or exterior smooth domain $\Om$ in $\R^n$, $n \ge 3$, with drifts $\bb$ in the critical weak $L^n$-space $L^{n,\infty}(\Om ; \R^n )$. First, assuming that the drift $\bb$ has nonnegative weak divergence in $L^{n/2, \infty }(\Om )$, we establish existence and uniqueness of weak solutions in $W^{1,p}(\Om )$ or $D^{1,p}(\Om )$ for any $p$ with $n' = n/(n-1)< p < n$.
By duality, a similar result also holds for the dual problem.
Next, we  prove $W^{1,n+\eps}$ or $W^{2, n/2+\delta}$-regularity of weak solutions of the dual problem  for some $\eps, \delta >0$ when the domain $\Om$ is bounded.
By duality, these results enable us to obtain a quite general uniqueness result as well as an existence result for weak solutions belonging to  $\bigcap_{p< n' }W^{1,p}(\Om )$.
Finally, we prove a uniqueness result for exterior problems, which implies in particular that (very weak) solutions are unique in both $L^{n/(n-2),\infty}(\Om )$ and $L^{n,\infty}(\Om )$.

{\it Keywords}: elliptic equations, rough drifts, regularity, uniqueness

{\it Mathematics Subject Classification (2010)}: 35J15, 35J25

\end{abstract}

\section{Introduction}
\label{Sec1}

Let $\Om$ be a bounded or exterior domain in $\R^n$, where $n \ge 3$. In this paper, we consider %
the following Dirichlet problem for linear elliptic equations of second order in divergence form:
\begin{equation}
\label{bvp}
\left\{\begin{array}{rr}
 -\De u + \div( u \mathbf{b} ) = f \quad \text{in }\Om ,\,\,\,\, \\[4pt]
 u =0 \quad\text{on } \pd \Om ,\\
 u(x) \to 0 \quad\text{as } |x| \to \infty.
 \end{array}
\right.
\end{equation}
Here the drift $\mathbf{b}= (b_1 , ..., b_n )$ is a given vector field on $\Omega$ and the data $f$ is a suitable scalar distribution on $\Om$.
We also consider the dual problem of \eqref{bvp}:
\begin{equation}
\label{bvp-dual}
\left\{\begin{array}{rr}
 -\De v - \bb \cdot \nb v = g \quad \text{in }\Om ,\,\,\,\, \\[4pt]
 v =0 \quad\text{on } \pd \Om ,
 \\
 v(x) \to 0 \quad\text{as } |x| \to \infty .
 \end{array}
\right.
\end{equation}
Of course, the decaying condition at infinity in \eqref{bvp} or \eqref{bvp-dual} should be neglected if $\Omega$ is bounded. Similarly, the boundary condition should be neglected  if $\Om$ is the whole space $\R^n$, which will be regarded as a special exterior domain.

\medskip

Suppose that $\Omega$ is a bounded domain with smooth boundary $\partial \Omega$.
Then it follows from the classical $L^p$-theory of elliptic equations (see \cite[Theorems 8.3, 8.6]{GilTru} and \cite[Theorem 4]{DK} e.g.) that if $\bb \in L^\infty (\Omega ; \R^n )$ and $1<p< \infty$, then for each $f \in W^{-1,p}(\Omega )$ there exists a unique weak solution $u$ in $W_0^{1,p}(\Omega )$ of the problem \eqref{bvp}. A similar result also holds for the dual problem \eqref{bvp-dual}.
These $W^{1,p}$-results have been extended to general elliptic equations with more singular drift terms $\bb$; for instance, see Droniou \cite{Dr}, Moscariello \cite{Mo}, Kim-Kim \cite{KK}, and Kang-Kim \cite{KangK}. In particular, it was shown in \cite[Theorem 1.1]{KK} that if $\bb \in L^n (\Omega ; \R^n )$ and $1<p< n$, then for each $f \in W^{-1,p}(\Omega )$ there exists a unique weak solution in $W_0^{1,p}(\Omega )$ of the problem \eqref{bvp}. Such a weak solution in $W_0^{1,p}(\Omega )$ of \eqref{bvp} will be called a $p$-weak solution or simply a weak solution if $p=2$ (see Definition \ref{def2.1}).

The class $L^n (\Omega ; \R^n )$ for the drift $\bb$ is {\it optimal} among the Lebesgue $L^r$-spaces for existence of $p$-weak solutions of \eqref{bvp}, as shown by the following simple example from \cite{Mo}.

\begin{example}\label{eg:noncoercive-LnI} Consider the problem \eqref{bvp}, where
$$
\Omega= B_1 (0) = \{ x \in \R^n: |x| < 1\}, \quad \bb (x)= - \frac{M x}{|x|^2}, \quad\mbox{and}
\quad f =- \, {\rm div}\, \bb .
$$
Note that $\bb \in L^r (\Omega ; \R^n )$ if and only if $r<n$.
Assume that $2 < p<n$ and $(n-p)/p \le M < (n-2)/2$.
Then $u (x) =|x|^{-M} - 1$ is a weak solution in $W_0^{1,2}(\Omega )$ of \eqref{bvp} but does not belong to $ W^{1,p}(\Omega)$. On the other hand, it was shown in \cite[Theorem 1.1]{Mo} that there exists at most one weak solution in $W_0^{1,2}(\Omega )$ of \eqref{bvp}. Hence there can be no $p$-weak solutions of \eqref{bvp} even though $f \in W^{-1,p}(\Omega )$.
\end{example}

For $1 \le p< \infty$, we denote by $L^{p,\infty} (\Omega )$ the weak $L^p$-space over $\Omega$, which is one of the standard Lorentz spaces $L^{p,q}(\Omega )$. Then $b(x) = 1/|x|$ is a typical example of functions in $L^{n,\infty} (B_1 (0) )$ but not in $L^n (B_1 (0) )$ (see Section 3.1 for more details). Hence Example \ref{eg:noncoercive-LnI} also shows that $p$-weak solutions of \eqref{bvp} may fail to exist for general drifts $\bb$ in the critical weak space $ L^{n,\infty} (\Omega ; \R^n )$.
This suggests us to impose an additional condition on the drift $\bb$ for better regularity of weak solutions of \eqref{bvp} and \eqref{bvp-dual}  when $\bb \notin L^n  (\Omega ; \R^n )$. For instance, motivated partially by the fluid mechanics of incompressible flows, we may assume that $\bb\in L^{n,\infty}  (\Omega ; \R^n ) $ and  $\div \bb =0$ weakly in $\Om$. In this case, the interior regularity   and Liouville property of  weak solutions  have been extensively studied  by
Zhikov \cite{Zhi}, Kontovourkis \cite{Ko}, Nazarov-Uraltseva \cite{NU},
 Zhang \cite{ZhangQi}, Chen-Strain-Tsai-Yau \cite{CSTY},
Seregin-Silvestre-\v{S}ver\'{a}k-Zlato\v{s} \cite{SSSZ},
Filonov \cite{Fi}, Ignatova-Kukavica-Ryzhik \cite{IKR}, and Filonov-Shilkin \cite{FSh1,FSh2}.

\medskip

The main purpose of the paper is to study existence, uniqueness, and regularity of weak solutions or $p$-weak solutions of the problem \eqref{bvp} and its dual \eqref{bvp-dual}, when the drift $\bb$ in $L^{n,\infty} (\Omega ; \R^n )$ satisfies the additional condition
\begin{equation}\label{coercivity cond-b}
 \div \mathbf{b} \ge 0 \quad\mbox{(weakly) in} \,\, \Om;
\end{equation}
that is,
\[
-\int_\Omega \bb \cdot \nb \phi \, dx \ge 0 \quad\mbox{for all nonnegative}\,\, \phi \in C_c^1 (\Om ).
\]
Note that the nonnegativity condition (\ref{coercivity cond-b})  implies the coercivity of the (non-symmetric) bilinear form associated with \eqref{bvp} and \eqref{bvp-dual}:
\begin{equation}\label{bilinear form}
B(u,v)= \int_\Om \bke{\nb u - u \mathbf{b}} \cdot \nb
v \, dx .
\end{equation}

Let us now  summarize    the main results that are obtained in the paper.

First of all, assuming that $\Omega$ is   bounded, $\bb \in L^{n,\infty} (\Omega ; \R^n )$, and $\div \mathbf{b} \ge 0$ in $\Om$, we shall prove existence and uniqueness of $p$-weak solutions of \eqref{bvp} for the case when $2 \le p< n $ (see Theorem \ref{th4-q version}). Existence of a unique weak solution of \eqref{bvp} immediately follows from the Lax-Milgram theorem because  the bilinear form $B$ in  \eqref{bilinear form} is bounded and coercive.
To prove existence of $p$-weak solutions of \eqref{bvp} for the case $2 < p< n $, we apply the classical iteration technique due to J. Moser and then utilize several results from the theory of real interpolation. First, by Moser's iteration method, we show that if $2 < p< n $ and $f \in W^{-1,p}(\Omega )$, then the weak solution $u \in W_0^{1,2}(\Om ) $ of \eqref{bvp} satisfies the higher integrability $u \in L^{p^*}(\Om )$, where $p^* = np/(n-p)$ is the Sobolev conjugate of $p$. Next, by H\"{o}lder's inequality in weak spaces, we deduce that $u \bb \in L^{p,\infty} (\Omega ; \R^n )$. Then since $-\De u = f- \div( u \mathbf{b} )$ in $\Omega$, it follows from the Calderon-Zygmund result in weak spaces (Proposition \ref{CZ-estimates}) that $\nabla u \in L^{p,\infty} (\Omega ; \R^n )$.
Finally, by the Marcinkiewicz interpolation theorem (Lemma \ref{marcinkiewicz-inter}), we can conclude that $\nabla u \in L^{p} (\Omega ; \R^n )$. This outlines our proof that if $2 \le p< n $, then for each $f \in W^{-1,p}(\Omega )$ there exists a unique $p$-weak solution $u$ of \eqref{bvp}. By duality, we also deduce that if $n' = n/(n-1)< p \le 2$, then for each $g \in W^{-1,p}(\Omega )$ there exists a unique $p$-weak solution $v$ of \eqref{bvp-dual}.

Again, by Moser's method, we can show (Lemma \ref{th4-higher integrability}) that if $2 < p< n $ and $g \in W^{-1,p}(\Omega )$, then there exists a unique weak solution $v \in W_0^{1,2}(\Om ) \cap L^{p^*}(\Om )$ of \eqref{bvp-dual}. Since $v \in W_0^{1,2}(\Om ) $, it follows from a standard bilinear estimate (see Lemma \ref{th1} e.g.) that the convection term $\bb \cdot \nabla v$ in \eqref{bvp-dual} belongs to $ W^{-1,2}(\Om )$. But no further regularity of $\bb \cdot \nabla v$ follows from the higher $L^{p^*}$-integrability of $v$, which is contrary to the convection term $\div( u \mathbf{b} )$ in \eqref{bvp}. This is why we need an additional assumption (see \eqref{coercivity cond-b-stronger} below) on the drift $\bb$.
Note that if $\bb (x) = - Mx /|x|^2$ as in Example \ref{eg:noncoercive-LnI}, then $\div \bb (x) = -M(n-2)/|x|^2 $ in $\R^n \setminus \{0\}$ and so $\div \bb \in L^{n/2,\infty} (\R^n )$.
Hence it is quite natural to assume that the drift $\bb$ in $ L^{n,\infty} (\Omega ; \R^n )$ satisfies
\begin{equation}\label{coercivity cond-b-stronger}
\div \mathbf{b}\in L^{n/2,\infty} (\Omega ) \quad\mbox{and}\quad \div \bb \ge 0 \quad\mbox{in} \,\, \Om;
\end{equation}
that is, there exists a nonnegative function $c \in L^{n/2,\infty} (\Omega )$ such that
\[
-\int_\Omega \bb \cdot \nb \phi \, dx = \int_\Omega c \, \phi \, dx \quad\mbox{for all}\,\, \phi \in C_c^1 (\Om ).
\]

Assume now that the drift $\bb$ in $ L^{n,\infty} (\Omega ; \R^n )$ satisfies (\ref{coercivity cond-b-stronger}).
Then since
$$
-\De v = g+ \mathbf{b} \cdot \nabla v =g+ {\rm \div} \, (v \mathbf{b} ) - v \, {\rm div}\, \mathbf{b}\quad\mbox{in}\,\, \Om,
$$
we can deduce that $\nabla v \in L^{p} (\Omega ; \R^n )$. By duality, it also follows that if $n' < p < 2$, then for each $f \in W^{-1,p}(\Omega )$ there exists a unique $p$-weak solution $u$ of \eqref{bvp}. Moreover, it will be shown that if $f\in L^q (\Om )$ and $1<q < n/2$, then $u \in W^{2,q}(\Om )$. A similar result also holds for weak solutions of the dual problem \eqref{bvp-dual}.
Furthermore, all these results hold even when the domain $\Om$ is an exterior domain in $\R^n$. In particular, if $n' < p< n$, then there exists unique $p$-weak solutions of both problems \eqref{bvp} and \eqref{bvp-dual} on exterior domains $\Om$ in $\R^n$.
See Theorems \ref{th4-q version} and \ref{th4-q version-strong} for more details.

\medskip

Next, assuming that $\Omega$ is bounded, $\bb \in L^{n,\infty} (\Omega ; \R^n )$, $\div \mathbf{b}\in L^{n/2,\infty} (\Omega )$, and $\div \bb \ge 0$ in $\Om$, we shall prove $W^{1,p}$-regularity of weak solutions of the dual problem \eqref{bvp-dual} for some $p>n$. Suppose that $g \in W^{-1,p}(\Omega )$ for some
$n < p< \infty $. Then it follows from Theorem \ref{th4-q version} that the problem \eqref{bvp-dual} has a unique weak solution $v$ which belongs to $W_0^{1,q}(\Om )$ for any $q<n$.
Since $\bb \in L^{n,\infty} (\Omega ; \R^n )$, it follows from H\"{o}lder and Sobolev inequalities in Lorentz spaces that
\begin{equation}\label{int-bilinear est}
\|\bb \cdot \nb w \|_{ W^{-1,p}(\Omega )} \le C (n,p,\Om ) \| \bb \|_{L^{n,\infty} (\Omega  )} \|w \|_{W_0^{1,p}(\Om )}
\end{equation}
for all $w \in W_0^{1,p}(\Om )$.
If $\bb$ were in $ L^{n} (\Omega ; \R^n )$ or more generally, if $\bb$ could be approximated in $ L^{n,\infty} (\Omega ; \R^n )$ by smooth vector fields, then using the estimate (\ref{int-bilinear est}) (see the proofs of
\cite[Lemmas 3.3, 3.4]{KK}), we could derive the following $\eps$-inequality:
$$
\|\bb \cdot \nb w \|_{ W^{-1,p}(\Omega )} \le \eps \|w \|_{W_0^{1,p}(\Om )} + C_\eps (n,p,\Om , \bb ) \|w\|_{L^p (\Om )},
$$
from which we deduce that $v \in W_0^{1,p}(\Om )$, by applying a standard method such as the Leray-Schauder principle \cite{KK} or the method of continuity \cite{KangK}.
However the lack of a density property of $ L^{n,\infty} (\Omega )$ prevents us from deriving such an $\eps$-inequality for general $\bb$ in $ L^{n,\infty} (\Omega ; \R^n )$.
To overcome this difficulty, we first show, once again by Moser's iteration technique, that $v$ is globally H\"{o}lder continuous on $\Om$, that is, $v \in C^\al(\overline{\Om})$ for some $0< \alpha \le 1-n/p$.
We then deduce that $v \in W^{1,n +\eps }(\Om )$ for some $0 < \eps \le p-n$, by making crucial use of the Calderon-Zygumend estimates as well as the following interpolation inequality due to Miranda \cite{Miranda} and Nirenberg \cite{Nirenberg}:
$$
 \norm{\nb w}_{L^r(\Om)} \le C(n, \alpha , q , \Omega ) \left( \norm{ w}_{W^{2,q}(\Om)}+\norm{ w}_{C^\al(\bar \Om)} \right),
$$
where $1 \le q<n$, $0< \alpha <1$, and $ r = (2-\al)q /(1-\al )$. It will be also shown that if $g \in L^q (\Om )$ for some $n/2<q<\infty$, then $v \in W^{2,n/2+\delta }(\Om )$ for some $0 < \delta \le q-n/2$. See Theorem \ref{th4} and its proof for more details.

As an important application of these regularity results for \eqref{bvp-dual}, we obtain a quite general uniqueness result for very weak solutions of \eqref{bvp}. By a \emph{very weak solution} of \eqref{bvp}, we mean a function $u $ on $\Om$ such that
$$
u \in L^r (\Om )\quad\mbox{for some} \,\, \frac{n}{n-1}< r< \infty
$$
and
$$%
- \int_\Om u \bke{ \De \phi + \mathbf{b} \cdot \nb
\phi} \, dx = \bka{f,\phi}
$$%
for all $\phi \in C^2 ( \overline{\Om}) $ with $\phi|_{\partial\Omega}=0 $.
Then by duality, it follows from Theorem \ref{th4} that very weak solutions of \eqref{bvp} are unique in $L^r (\Om )$ for some $r $ less than but close to $(n/2)'= n/(n-2)$. Note that if $1<p<n$, then every $p$-weak solution of \eqref{bvp} is a very weak solution of \eqref{bvp} belonging to $L^{p^*}(\Om )$. Hence it also follows that $p$-weak solutions of \eqref{bvp} are unique if $p $ is less than but sufficienlty close to $n'= n/(n-1)$. Existence of such a $p$-weak solution seems to be open for general $f \in W^{-1,p}(\Om )$, when $p<n'$. Nevertheless, by a duality argument, we shall show that if $f \in W^{-1,p}(\Om )$ for all $p<n'$, then the problem \eqref{bvp} has a unique weak solution  $u \in \cap_{p<n'}W_0^{1,p}(\Om )$.
See Theorem \ref{th5} for precise statements.

\medskip

Last but not least,
assuming still that $\bb \in L^{n,\infty} (\Omega ; \R^n )$, $\div \mathbf{b}\in L^{n/2,\infty} (\Omega )$, and $\div \bb \ge 0$ in $\Om$, we shall prove that very weak solutions of \eqref{bvp} are unique in $L^{n,\infty}(\Om )$ even when $\Om$ is an exterior domain. Indeed, this   is an immediate consequence of  a more general  uniqueness result, Theorem \ref{th6}.  Our uniqueness result is motivated by an
open question about uniqueness of stationary solutions of the incompressible Navier-Stokes equations on exterior domains. Let $X^{n,\infty} (\Om )$ be the space of all vector fields $\mathbf{u}$ on $\Om$ such that $\mathbf{u} \in L^{n,\infty} (\Om ; \R^n )$ and $\nabla \mathbf{u} \in L^{n/2,\infty} (\Om ; \R^{n^2} )$. Then the class $X^{n,\infty} (\Om )$ turns out to be critical for solvability of the Navier-Stokes equations because $X^{n,\infty} (\R^n )$ is invariant under the natural scaling $\mathbf{u}_\lambda (x) = \lambda \mathbf{u}(\lambda x)$ with respect to the Navier-Stokes equations on $\R^n$.
In fact, Kozono-Yamazaki \cite{KoYa0} proved existence of a weak solution in $X^{n,\infty} (\Om )$ of the stationary Navier-Stokes equations under a smallness assumption on the exterior force.
A proof of uniqueness of such a solution may be reduced to showing uniqueness of a weak solution $\mathbf{v}$ in  $X^{n,\infty} (\Om )$ of the following problem:
\begin{equation}
\label{stokes-bvp}
\left\{\begin{array}{rr}
 -\De \mathbf{v} + \div( \mathbf{b} \otimes \mathbf{v}) +\nabla p = \mathbf{f}, \quad \div \mathbf{v}=0 \quad \text{in }\Om ,\,\,\,\, \\[4pt]
 \mathbf{v} =0 \quad\text{on } \pd \Om ,\\
 \mathbf{v} (x) \to 0 \quad\text{as } |x| \to \infty,
 \end{array}
\right.
\end{equation}
where $\mathbf{b}$ is a given vector field in $X^{n,\infty} (\Om ; \R^n )$ with $\div \mathbf{b} =0$.
Unfortunately, it remains still open to prove uniqueness of weak solutions in $X^{n,\infty} (\Om )$ of the linearized Navier-Stokes problem (\ref{stokes-bvp}). However by Theorem \ref{th6}, a stronger uniqueness result holds for the scalar problem \eqref{bvp}; consequently, very weak solutions of \eqref{bvp} are unique in $L^{n,\infty}(\Om )$, provided that the drift $\bb$ in $ L^{n,\infty} (\Omega ; \R^n )$ satisfies (\ref{coercivity cond-b-stronger}). Uniqueness of very weak solutions in  $L^{n/(n-2),\infty}(\Om )$, which is the natural space for the Laplace equation on exterior domains,  also follows from the general uniqueness result, Theorem \ref{th6}.

\medskip

The rest of the paper is organized as follows. All the main results in the paper are stated in the following Section \ref{Sec2}.
Section \ref{Sec3} is devoted to stating and proving preliminary results including some classical results on Lorentz spaces, estimates involving weak $L^n$-functions, the Miranda-Nirenberg interpolation inequality, and the Calderon-Zygmund estimates in Lebesgue and Lorentz spaces.
In Section \ref{Sec4}, we prove several existence and regularity results for weak solutions of \eqref{bvp} and \eqref{bvp-dual}. Global H\"older estimates for weak solutions of \eqref{bvp-dual} are derived in Section \ref{Sec5}. Finally, in Sections \ref{Sec6} and \ref{Sec7}, we provide complete proofs of all the main results.

\section{Main results}
\label{Sec2}

Before stating our main results, let us introduce some standard function spaces.

Let $\Om$ be any domain in $\R^n$, where $n \ge 3$.
For $m \in \N \cup \{\infty\}$, let
$C_c^m (\Om )$ be the space of all functions in $C^m (\R^n )$ with compact supports in $\Om$
and let $C^m (\overline{\Om}) $ be the space of the restrictions to $\overline\Om$ of all functions in $ C^m (\R^n)$. The space of all functions in $C^m (\overline{\Omega})$ with compact supports is then denoted by $C_c^m (\overline{\Omega})$. Let $L^{q}_{\loc}(\overline \Om )$ be the space of all $u \in L^{q}_{\loc}(\Om )$ such that $u \in L^{q}(\Omega_R )$ for all $R>0$, where $\Omega_R =\Omega \cap B_R (0)= \{ x \in \Omega : |x|< R\}$.

For $m \in \N$ and $1 \le p \le \infty$, let $D^{m,p}(\Om )$ be the space of all functions $u$ in $L_{\loc}^p (\Om )$ such that $D^\alpha u$ exists and belongs to $L^p (\Om )$ for all indices $\alpha$ with $|\alpha |=m$.
If $1 \le p < n$, we denote by $p^*$ the Sobolev conjugate of $p$: $p^* = np/(n-p)$.
Then the intersection space $D^{1,p}(\Omega )\cap L^{p^*}(\Omega )$, equipped with the norm $\|\nabla \cdot \|_{L^p (\Omega )}+ \|\cdot \|_{L^{p^*}(\Omega )}$, is a Banach space.
For $1\le p<n $, let $\hat{D}^{1,p}_0 (\Om )$ be the closure of $C_c^\infty (\Om )$ in the space $D^{1,p}(\Omega )\cap L^{p^*}(\Omega )$.
It immediately follows from the Sobolev embedding theorem that $\| u \|_{L^{p^* }(\Om )} \le C(n,p) \|\nb u \|_{L^p (\Om ; \R^n )}$ for all $u \in \hat{D}^{1,p}_0 (\Om )$.
Hence $\hat{D}^{1,p}_0 (\Om )$ may be equipped with an equivalent norm $\|\nabla \cdot \|_{L^p (\Omega )}$.
If $\Omega$ is bounded, then the norm on $\hat{D}^{1,p}_0 (\Om )$ is also equivalent to the norm on the inhomogeneous Sobolev space $W^{1,p}(\Om )$. Hence it follows that if $\Omega$ is bounded, then $\hat{D}^{1,p}_0 (\Om )$ coincides with the space $W_0^{1,p}(\Om )$, which is the closure of $C_c^\infty (\Om )$ in $W^{1,p}(\Omega )$.
On the other hand, it is well-known that if $\Om$ is an exterior domain with Lipschitz boundary $\partial\Om$, then each $u \in D^{1,p}(\Om )$ has a well-defined trace on $\partial\Om$.
For  an exterior Lipschitz domain $\Om$  in $\R^n$,   we denote by $ D^{1,p}_0 (\Om )$
the space of all $u \in D^{1, p }(\Om )$ with $u =0$ on $\partial \Om$. Then it  can be shown (see \cite[Section II.6]{Ga} e.g.) that if $\Om$ is an exterior Lipschitz domain in $\R^n$, then
$\hat{D}^{1,p}_0 (\Om ) = D^{1,p}_0 (\Om )   \cap L^{p^*}(\Om )$ for $1 \le p<n$.

\begin{definition}\label{def2.1} %
Assume that $\mathbf{b} \in L_{\loc}^{1}(\Om;\R^n)$ and $1<p<n$. Then a function $u \in \hat{D}_0^{1,p}(\Om )$ is called a \emph{weak solution} in $\hat{D}_0^{1,p}(\Om )$ or a $p$-\emph{weak solution} of \eqref{bvp} if it satisfies
$$
u \mathbf{b} \in L_{\loc}^{1}(\Om;\R^n)
$$
and
\begin{equation}\label{eq1.3}
\int_\Om \bke{\nb u - u \mathbf{b}} \cdot \nb
\phi \, dx = \bka{f,\phi} \quad \text{for all } \phi \in C^1_{c}(\Om).
\end{equation}
Weak solutions in $\hat{D}_0^{1,2}(\Om )$ of \eqref{bvp} are simply called \emph{weak solution}s.
In addition, a $p$-weak solution $u$ of \eqref{bvp} will be called a \emph{strong solution} if it satisfies $\nabla^2 u \in L_{loc}^{1}(\Om; \R^{n^2} )$. Weak and $p$-weak solutions of the dual problem \eqref{bvp-dual} can be similarly defined;  that is, $v \in \hat{D}_0^{1,p}(\Om )$ is   a $p$-\emph{weak solution} (simply a \emph{weak solution} if $p=2$) of \eqref{bvp-dual} if it satisfies
\[
\mathbf{b}\cdot\nabla v   \in L_{\loc}^{1}(\Om )
\]
and
\begin{equation}\label{bvp-dual-wk}
\int_\Om \nb v\cdot
\bke{\nb \psi - \psi \mathbf{b}}   \, dx = \bka{g,\psi} \quad \text{for all } \psi \in C^1_{c}(\Om).
\end{equation}

\end{definition}

\medskip
The first purpose of the paper is to establish the following two results for existence and uniqueness of $p$-weak solutions and strong solutions of the problem \eqref{bvp} and its dual \eqref{bvp-dual}.

\begin{theorem}
\label{th4-q version}
Let $\Om $ be a bounded or exterior $C^1$-domain in $\R^n$, $n \ge 3$. Assume that
$\mathbf{b} \in L^{n,\I}(\Om;\R^n)$, $\div \mathbf{b} \ge 0$ in $\Om$, and $n /(n-1) < p<n$.
Assume in addition that ${\rm div}\, {\mathbf b} \in L^{n/2,\I} (\Om )$ if $n/(n-1)< p<2$.
\begin{enumerate}[(i)]
\item For each ${\mathbf F} \in L^{p}(\Om; \R^n )$, there exists a unique $p$-weak solution $u $ of \eqref{bvp} with $f= \div {\mathbf F}$. Moreover, we have
$$
 \|\nb u \|_{L^p (\Om )} \le C (n,p ,\Om ) M ( p,{\mathbf b}) \|{\mathbf F}\|_{L^{p}(\Om )},
$$
where
$$
 M ( p, {\mathbf b}) =\begin{cases}
 \ 1 & \text{if } p =2 \\
 \ 1+ \|{\mathbf b}\|_{L^{n,\I} (\Om )} & \text{if } 2 < p< n \\
 \ 1+ \|{\mathbf b}\|_{L^{n,\I} (\Om )}+ \|{\rm div}\, {\mathbf b}\|_{L^{n/2,\I} (\Om )} & \text{if } \frac{n}{n-1}< p<2 .
\end{cases}
$$

\item For each ${\mathbf G} \in L^{p'}(\Om; \R^n )$, there exists a unique $p'$-weak solution $v $ of \eqref{bvp-dual} with $g= \div {\mathbf G}$. Moreover, we have
$$
 \|\nb v \|_{L^{p'} (\Om )} \le C (n,p ,\Om ) M ( p,{\mathbf b}) \|{\mathbf G}\|_{L^{p'}(\Om )}.
$$
\end{enumerate}
\end{theorem}

\begin{theorem}\label{th4-q version-strong}
Let $\Om $ be a bounded or exterior $C^{1,1}$-domain in $\R^n$, $n \ge 3$. Assume that $\mathbf{b} \in L^{n,\I}(\Om;\R^n)$, ${\rm div}\, {\mathbf b} \in L^{n/2,\I} (\Om )$, $\div \mathbf{b} \ge 0$ in $\Om$, and $1< q< n/2 $.
\begin{enumerate}[(i)]
\item For each $f \in L^{q}(\Om )$, there exists a unique $q^*$-weak solution $u $ of \eqref{bvp}. Moreover, we have
$$
 u \in D^{2,q} (\Om )
 \quad\mbox{and}\quad
 \|\nb u\|_{ L^{q^*} (\Om ) }+ \|\nb^2 u\|_{ L^{q} (\Om ) } \le C (n,q ,\Om ) M_{{\mathbf b}}^2 \|f\|_{L^q (\Omega )} ,
$$
where $M_{{\mathbf b}} = 1+ \|{\mathbf b}\|_{L^{n,\I} (\Om )}+ \|{\rm div}\, {\mathbf b}\|_{L^{n/2,\I} (\Om )}$.
\item For each $g\in L^{q}(\Om)$, there exists a unique $q^*$-weak solution $v $ of \eqref{bvp-dual}. Moreover, we have
$$
 v \in D^{2,q} (\Om )
 \quad\mbox{and}\quad
 \|\nb v\|_{ L^{q^*} (\Om ) }+ \|\nb^2 v\|_{ L^{q} (\Om ) } \le C (n,q ,\Om ) M_{{\mathbf b}}^2 \|g\|_{L^q (\Omega )} .
$$
\end{enumerate}
\end{theorem}

\begin{remark}
Results similar to Theorems \ref{th4-q version} and \ref{th4-q version-strong}  can be found in the literature under   stronger assumptions on the drift $\bb$. For instance, the simplest case  $\bb ={\mathbf 0}$ was  studied in  \cite{AmGiGi} for both bounded and exterior domains $\Om$ in $\R^n$. See also    \cite{GilTru}, \cite{KK}, and  \cite{KangK} for the case when $\bb \in L^\infty (\Om ; \R^n )$ or more generally $\bb \in L^n (\Om ; \R^n )$, where  $\Om$ is a bounded domain.
\end{remark}

\begin{remark}
The $C^1$-regularity of the domain  $\Om$ is needed in our proof of Theorem \ref{th4-q version} that relies crucially on  the Calderon-Zygmund estimate for  weak solutions in $W^{1,p}$, $1<p<\infty$, of the Poisson equations on bounded  $C^1$-domains (see Subsection \ref{sec:CZ-est}). However it was shown by Jerison-Kenig  \cite[Theorem 0.5]{JK} that if $\Omega $ is   any bounded Lipschitz domain in $\R^3$, then such an estimate still holds for  $p_1 ' <p< p_1$, where   $p_1 $ is some number with $3< p_1 < \infty$.  Thus   Theorem \ref{th4-q version}   remains true if $\Omega$ is a general bounded \emph{Lipschitz} domain in $\R^3$.
\end{remark}

\medskip

Next, we prove $W^{1,n+\eps}$- or $W^{2,n/2+\delta }$-regularity of weak solutions of \eqref{bvp-dual} for some $\eps , \delta>0$.

\begin{theorem}\label{th4} Let $\Om $ be a bounded $C^{1,1}$-domain in $\R^n$, $n \ge 3$. Assume that $\mathbf{b} \in L^{n,\I}(\Om;\R^n)$, ${\rm div}\, {\mathbf b} \in L^{n/2,\I} (\Om )$, and $\div \mathbf{b} \ge 0$ in $\Om$. Then any weak solution $v $ of \eqref{bvp-dual} has the following regularity properties:
\begin{enumerate}[(i)]
\item If $g \in W^{-1,p}(\Om)$ and $n<p<\infty$, then
$$
v \in W^{1,n+\e}_0(\Om) \quad\mbox{and}\quad \|v\|_{W^{1,n+\e}_0(\Om)} \le C \|g\|_{W^{-1,p} (\Omega )}
$$
for some $ 0< \e \le p-n$ and $C>1$, depending only on $n, p, \Omega $, $ \|{\mathbf b}\|_{L^{n,\I}(\Om )} $, and $\|{\rm div}\, {\mathbf b}\|_{L^{n/2,\I} (\Om )}$.
\item If $g \in L^{q}(\Om)$ and $n/2 <q<\infty$, then
$$
v \in W^{2,n/2+\delta}(\Om) \quad\mbox{and}\quad \|v\|_{W^{2,n/2+\delta}(\Om)} \le C \|g\|_{L^q (\Omega )}
$$
for some $ 0 < \delta \le q-n/2 $ and $C>1$, depending only on $n, q, \Omega $, $ \|{\mathbf b}\|_{L^{n,\I}(\Om )} $, and $\|{\rm div}\, {\mathbf b}\|_{L^{n/2,\I} (\Om )}$.
\end{enumerate}
\end{theorem}

\begin{remark}
By Theorem \ref{th4}, there exists a weak solution $v$ of \eqref{bvp-dual} such that $v \in W^{1,n+\e}_0(\Om)$ for some $\eps>0$. Then by the Morrey embedding theorem, we deduce that $v \in C^\alpha (\overline{\Om})$ for some $0< \alpha < 1$. But, the H\"older regularity of $v$ could have been proved without assuming that ${\rm div}\, {\mathbf b} \in L^{n/2,\I} (\Om ).$ See Proposition \ref{th4-0}. In fact, the global H\"older regularity of $v$ is an essential ingredient of our proof of Theorem \ref{th4} by means of the Miranda-Nirenberg interpolation theorem (Lemma \ref{MNineq}).
\end{remark}

\begin{remark}
On one hand, $C^{1,1}$-regularity of the domain $\Om$ is reasonable for Part (ii) of Theorem \ref{th4} because it involves $W^{2,s}$-regularity of solutions of elliptic equations of second order. On the other hand, our proof of Part (i) is based crucially on the Calderon-Zygmund result for solutions in $W^{2,s}$ of the Poisson equation; that is, if $v \in W_0^{1,2}(\Om )$ and $ \De v \in L^s (\Om )$ for $s$ close to $n/2$, then $v \in W^{2,s}(\Om )$. This requires us to assume $C^{1,1}$-regularity of the domain $\Om$ for Part (i) too.
\end{remark}

\medskip

As an important consequence of Theorem \ref{th4}, we can prove existence and uniqueness results for $p$-weak solutions or very weak solutions in $L^q (\Omega )$ of \eqref{bvp}, where $ p< n/(n-1)$ and $q< n/(n-2)$. Note that
$$
n' =\frac{n}{n-1} \quad\mbox{and}\quad (n')^* = \frac{n}{n-2} = \left( \frac n2 \right)' .
$$
For the simplicity of presentation, let us define
\[
W_0^{1,p-}(\Om )=\bigcap_{q<p}W_0^{1,q}(\Om )\quad\mbox{and}\quad W^{-1,p-}(\Om )=\bigcap_{q<p}W^{-1,q}(\Om ).
\]

\begin{theorem}\label{th5} Let $\Om $ be a bounded $C^{1,1}$-domain in $\R^n$, $n \ge 3$. Assume that $\mathbf{b} \in L^{n,\I}(\Om;\R^n)$, ${\rm div}\, {\mathbf b} \in L^{n/2,\I} (\Om )$, and $\div \mathbf{b} \ge 0$ in $\Om$.

\begin{enumerate}[(i)]
 \item There exists a number $ n ' <r< (n/2)' $, close to $(n/2)'$ and
depending only on $n, \Omega $, $ \|{\mathbf b}\|_{L^{n,\I}(\Om )} $, and $\|{\rm div}\, {\mathbf b}\|_{L^{n/2,\I} (\Om )}$, such that
if $u \in L^r (\Omega )$ satisfies
\begin{equation}\label{eq1.3-0}
\int_\Om u \bke{ \De \phi + \mathbf{b} \cdot \nb
\phi} \, dx = 0 \quad \text{for all } \phi \in C^2 (\overline{\Om}) \,\,\mbox{with}\,\, \phi|_{\partial\Omega}=0 ,
\end{equation}
then $u=0$ identically on $\Omega$.
 \item For each $f \in W^{-1, n'-}(\Om)$, there exists a unique weak solution $u$ in $W_0^{1,n'-}(\Om )$ of \eqref{bvp}.\end{enumerate}
\end{theorem}

\begin{remark} On one hand,
Theorem \ref{th5} (ii) is a partial extension of Theorem \ref{th4-q version} (i) to the case  $p=n'$ when $\Om$ is bounded. On the other hand, suppose that $f= \div {\mathbf F}$ and ${\mathbf F} \in L^{n' , \I}(\Om; \R^n )$. Then since ${\mathbf F} \in L^{p }(\Om; \R^n )$ for any $p< n'$, it follows from Theorem \ref{th5} that there exists a unique weak solution $u$ in $W_0^{1,n'-}(\Om )$ of \eqref{bvp}. But our proof of Theorem \ref{th5} can not be adapted to prove the following very reasonable regularity of $u$:
$$
u \in L^{(n')^* , \I}(\Om ) \quad\mbox{and}\quad \nabla u \in L^{n' ,\infty}(\Om ;\R^n ),
$$
which seems to be an   open problem.
\end{remark}

\medskip

Finally, we prove the following uniqueness result for very weak solutions of \eqref{bvp} when $\Omega$ is an exterior domain in $\R^n$.

\begin{theorem}\label{th6} Let $\Om $ be an exterior $C^{1,1}$-domain in $\R^n$, $n \ge 3$. Assume that $\mathbf{b} \in L^{n,\I}(\Om;\R^n)$, ${\rm div}\, {\mathbf b} \in L^{n/2,\I} (\Om )$, and $\div \mathbf{b} \ge 0$ in $\Om$. Assume in addition that
\begin{enumerate}[(i)]
 \item $u\in L^{n/(n-2),\I}_\loc (\overline{\Om})$;
 \item $u \in L^{p_1 } (\Om ) + L^{p_2 } (\Om )$ for some $p_1 , p_2$ satisfying
 $ n' <  p_1 \le p_2  <\infty$; and
 \item $u$ satisfies
\begin{equation}\label{eq1.3-01}
\int_\Om u \bke{ \De \phi + \mathbf{b} \cdot \nb
\phi} \, dx = 0 \quad \text{for all } \phi \in C_c^2 (\overline{\Om}) \,\,\mbox{with}\,\, \phi|_{\partial\Omega}=0 .
\end{equation}
\end{enumerate}

\noindent Then $u=0$ identically on $\Omega$.
\end{theorem}

\begin{remark} It has been well-known that if $u$ is a smooth solution of the Laplace equation on an exterior domain $\Om$ in $\R^n, n \ge 3$, and $\lim_{x \to \infty} u(x)=0$, then $|u(x)| =O (|x|^{-(n-2)})$ as $|x| \to \infty$ in general.
This result suggests that $L^{n/(n-2),\infty}(\Omega)$ should be the natural class of (very weak) solutions of \eqref{bvp}. Uniqueness of solutions in  $L^{n/(n-2),\infty}(\Omega)$ of \eqref{bvp} immediately follows from Theorem \ref{th6}, since $L^{n/(n-2),\infty}(\Omega) \subset L^{p_1} (\Omega) +L^{p_2} (\Omega )$ for any $p_1 , p_2$ with $n/(n-1)  < p_1 <n/(n-2)<p_2 <\infty$.
\end{remark}

\begin{remark} Since $L^{n,\infty}(\Omega) \subset L^{p_1} (\Omega) +L^{p_2} (\Omega )$ for any $p_1 , p_2$ with $n' < p_1 <n<p_2 <\infty$,
it immediately follows from Theorem \ref{th6} that very weak solutions in $L^{n,\infty}(\Omega)$ of \eqref{bvp} are unique. However it is still open to prove an analogous uniqueness result for the linearized Navier-Stokes problem (\ref{stokes-bvp}).
\end{remark}

\begin{remark}
For the validity of all of our main theorems, Theorems \ref{th4-q version}--\ref{th6}, the drift $\bb$ in $ L^{n,\I}(\Om;\R^n)$ is  assumed to satisfy  $\div \bb \in L^{n/2,\infty}(\Omega)$ and $\div \bb \ge 0$ in $\Omega$. The stronger assumption $\div \bb = 0$ is not needed for any of the theorems. Moreover, for the case $2\le p < n$ of Theorem \ref{th4-q version}, we only assume that $\div \bb \ge 0$ in $\Omega$.
\end{remark}

\begin{remark}
 Several of our theorems are available for the problem \eqref{bvp} but not for its dual \eqref{bvp-dual}, and vice versa. See Remark \ref{holder regularity for u} for one example of the relevant difficulties.
\end{remark}

\section{Preliminaries}
\label{Sec3}

In this section we collect preliminary results, including some standard results for Lorentz spaces,
estimates involving weak $L^n$ functions, Miranda-Nirenberg interpolation inequalities, and Calderon-Zygmund estimates in Lebesgue and more generally in Lorentz spaces.

\subsection{Lorentz spaces}

Let $\Omega$ be any domain in $\R^n$. For $1 \le p<\infty$ and $1 \le q\le \infty$, let $L^{p,q}(\Om)$ denote the standard Lorentz space on $\Om$. Recall (see \cite{AdamsFournier,BL} e.g.) that
\begin{equation}\label{Lebesgue-Lorentz inclusion}
L^{p,1}(\Omega )\subset L^{p,p}(\Omega )=L^p (\Omega )\subset L^{p,\infty}(\Omega ) \subset L_{\loc}^r (\overline{\Om})
\end{equation}
 if $1 \le r <  p < \infty$. When $q=\infty$, $L^{p,\I}(\Om)$ coincides with the weak $L^p$-space over $\Omega$ and is equipped with the quasi-norm
$$
\| f\|_{L^{p,\I}(\Om)} = \sup_{t>0} \left[ t \, \mu_f (t)^{1/p} \right] ,
$$
where $\mu_f$ is the {\it distribution function} of $f$ defined by
\[
\mu_f (t) = | \{ x \in \Om : |f(x)|>t \}| \quad (t \ge 0).
\]
Using this quai-norm, we obtain basic inequalities for weak $L^p$-spaces: if $0 < r<p<\infty$, then
$$
\| f\|_{L^{p,\I}(\Om)} \le \|f\|_{L^p (\Om )}
$$
and
\begin{equation}\label{weak-basic-ineq}
 \left( \int_E |f|^r \, dx \right)^{1/r} \le \left( \frac{p}{p-r} \right)^{1/r} |E|^{1/r -1/p } \|f\|_{L^{p,\I}(\Om)}
\end{equation}
for all $E \subset \Omega$ with finite measure $|E|$.
A well-known application of weak $L^p$-spaces is the Marcinkiewicz interpolation theorem in the following simple form (see \cite[Theorem 9.8]{GilTru} or \cite[Theorem 1.3.1]{BL} e.g.):

\begin{lemma}[Marcinkiewicz interpolation]\label{marcinkiewicz-inter}
Let $T$ be a linear operator from $L^{p_0} (\Om) + L^{p_1} (\Om )$ into $L^{p_0,\I} (\Om) + L^{p_1,\I} (\Om )$, where $1 \le p_0 < p_1 < \infty$. Suppose that there are constants $M_0$ and $M_1$ such that
\[
\|T f \|_{L^{p_0 ,\I}(\Om)} \le M_0 \|f\|_{L^{p_0}(\Om )} \quad\mbox{and}\quad \|T f \|_{L^{p_1 ,\I}(\Om)} \le M_1 \|f\|_{L^{p_1}(\Om )}
\]
for all $f \in L^{p_0}(\Om ) \cap L^{p_1}(\Om )$. Then for any $p$ with $p_0 < p<p_1$, we have
\[
\|T f \|_{L^{p}(\Om)} \le C M_0^{1-\theta} M_1^{\theta} \|f\|_{L^{p}(\Om )} \quad\mbox{for all} \,\, f \in L^{p}(\Om ) ,
\]
where $1/p =(1-\theta)/p_0 + \theta/p_1$ and $C = C(p_0 , p_1 , p)$.
\end{lemma}

Another important fact for the Lorentz spaces is the following classical theorem from real interpolation theory (see \cite[Theorem 5.3.1]{BL} e.g.):
\begin{equation}\label{inter-Loren}
 L^{p,q}(\Omega) = \left( L^{p_0 }(\Omega), \; L^{p_1 }(\Omega) \right)_{\theta,q}
\end{equation}
whenever $1 \le p_0 < p < p_1 \le \infty$, $1 \le q \le \infty$,
 $0 < \theta < 1$, and $1/p =(1-\theta)/p_0 + \theta/p_1$. Consequently,

\begin{lemma}\label{inter-Loren-1}
Let $T$ be a linear operator from $L^{p_0} (\Om) + L^{p_1} (\Om )$ into itself, where $1 \le p_0 < p_1 \le \infty$. Suppose that there are constants $M_0$ and $M_1$ such that
\[
\|T f \|_{L^{p_0 }(\Om)} \le M_0 \|f\|_{L^{p_0}(\Om )} \quad\mbox{and}\quad \|T f \|_{L^{p_1 }(\Om)} \le M_1 \|f\|_{L^{p_1}(\Om )}
\]
for all $f \in L^{p_0}(\Om ) \cap L^{p_1}(\Om )$. Then for any $(p,q)$ with $p_0 < p<p_1$ and $1 \le q \le \infty$,
there is a constant $M$, depending only on $p_0$, $p_1$, $p$, $q$, $M_0$, and $M_1$, such that
\[
\|T f \|_{L^{p,q}(\Om)} \le M \|f\|_{L^{p,q}(\Om )} \quad\mbox{for all} \,\, f \in L^{p,q}(\Om ) .
\]
\end{lemma}

\medskip

To estimate $u \mathbf{b}$ with $\mathbf{b} \in L^{n,\I}(\Om;\R^n)$, we shall need H\"{o}lder's and Sobolev's inequalities in Lorentz spaces. The following H\"{o}lder inequality in Lorentz spaces was obtained by O'Neil \cite[Theorems 3.4, 3.5]{ON} (see also \cite[Proposition 2.1]{KoYa}).

\begin{lemma}\label{holder inequalies}
Assume that $f \in L^{p_1 , q_1}(\Om )$ and $g \in L^{p_2 , q_2}(\Om )$, where $1 \le p_1, p_2 < \infty$ satisfy $1/p = 1/p_1 + 1/p_2 \le 1$.
Assume further that $1/q_1 +1/q_2 \ge 1$ if $p=1$. Then
$$
f g \in L^{p,q} (\Om ) \quad\mbox{and}\quad \|f g\|_{L^{p,q} (\Om )} \le C(p) \|f\|_{L^{p_1 , q_1}(\Om )}\|g\|_{L^{p_2 , q_2}(\Om )}
$$
for any $q \ge 1$ with $1/ q_1 + 1/ q_2 \ge 1/q $.
\end{lemma}

In terms of Lorentz spaces, the classical Sobolev inequality can be refined as follows (see \cite[Remark 7.29]{AdamsFournier} and \cite{PO}).

 \begin{lemma}\label{refined Sobolev} Let $1 \le p<n$. Then for every $u \in W^{1,p} (\R^n )$, we have
$$
u \in L^{p^* , p}({{\mathbb R}^n} ) \quad\mbox{and}\quad \|u \|_{L^{p^* , p}({{\mathbb R}^n} )} \le C(n,p) \| \nabla u \|_{L^{p}({{\mathbb R}^n} )}.
$$
\end{lemma}

\subsection{Basic estimates}

We first establish the basic bilinear and trilinear estimates, which have been well-known for smooth domains $\Om$ in $ {\mathbb R}^n$ (see~\cite[Lemma 9]{KiKo} e.g.).
\begin{lemma}
\label{th1} Let $\Om$ be a bounded Lipschitz domain in ${\mathbb R}^n$, $n \ge 2$. Assume that ${\mathbf b} \in L^{n,\infty}(\Om ;\R^n )$ and $1 \le p<n$. Then for every $u \in W^{1,p}(\Om)$, we have
$$
u\mathbf{b} \in L^p (\Om; \R^n) \quad\mbox{and}\quad \norm{u \mathbf{b}}_{L^p (\Om) } \le C
\norm{\mathbf{b}}_{L^{n,\I}(\Om) }\norm{u}_{W^{1,p}(\Om)}
$$
for some constant $C=C (n, p,\Om )$. Moreover, for every $v \in W^{1,p'}(\Om )$,
$$
 \int_{\Om} |( u \mathbf{b}) \cdot \nb v| \, dx \le C \norm{\mathbf{b}}_{L^{n,\I}(\Omega ) } \norm{u}_{W^{1,p} (\Omega )} \norm{v}_{W^{1,p'}(\Omega )}
$$
for some constant $C=C (n, p, \Om )$.
\end{lemma}

\begin{proof}
Let $u \in W^{1,p}(\Om)$ be given. Since $\Omega$ is a bounded Lipschitz domain, there exists $\overline{u} \in W^{1,p}({{\mathbb R}^n} )$ such that $\overline{u}=u$ on $\Omega$ and $\|\overline{u} \|_{W^{1,p}({{\mathbb R}^n} )} \le C (n,p,\Omega ) \|u\|_{W^{1,p}(\Omega)}$. Let $\overline{{\mathbf b}}$ be an extension of ${\mathbf b}$ to ${\mathbb R}^n$ defined by $\overline{{\mathbf b}}= {\mathbf 0}$ on ${{\mathbb R}^n} \setminus \Omega$.
By Lemmas \ref{holder inequalies} and \ref{refined Sobolev},
$$
\left\|\overline{u} \, \overline{{\mathbf b}}\right\|_{L^{p}({{\mathbb R}^n} )}
 \le C \|\overline{{\mathbf b}}\|_{L^{n,\infty} ({{\mathbb R}^n} )}\|\overline{u} \|_{L^{p^* , p}({{\mathbb R}^n} )} \le C \|\overline{{\mathbf b}}\|_{L^{n,\infty} ({{\mathbb R}^n} )} \|\nabla \overline{u} \|_{L^{p}({{\mathbb R}^n} )}.
$$
This proves the bilinear estimate. The trilinear estimate follows immediately, by H\"{o}lder's inequality.
\end{proof}

\medskip
The following result however holds for arbitrary domains $\Om$ in $\R^n$.
\begin{lemma}
\label{th1-general} Let $\Om$ be any
domain in ${\mathbb R}^n$, $n \ge 2$. Assume that ${\mathbf b} \in L^{n,\infty}(\Om ;\R^n )$ and $1 \le p<n$. Then for every $u \in \hat{D}_0^{1,p}(\Om)$, we have
$$
u\mathbf{b} \in L^p (\Om; \R^n) \quad\mbox{and}\quad \norm{u \mathbf{b}}_{L^p (\Om) } \le C
\norm{\mathbf{b}}_{L^{n,\I}(\Om) }\norm{\nb u}_{L^{p}(\Om)}
$$
for some constant $C=C (n, p )$. In addition, if $n \ge 3$ and $\div \mathbf{b} \ge 0$ in $\Om$, then
$$
-\int_{\Om} ( u \mathbf{b}) \cdot \nb u \, dx \ge 0 \quad\mbox{for all} \,\, u \in \hat{D}_0^{1,2} (\Om) .
$$
\end{lemma}
\begin{proof}
Let $u \in \hat{D}_0^{1,p}(\Om)$ be given. By the definition of $\hat{D}_0^{1,p}(\Om)$, there exist functions $u_k \in C_c^\infty (\Om )$ such that $u_k \to u$ in $L^{p^*}(\Om )$ and $\nabla u_k \to \nabla u$ in $L^{p}(\Om; \R^n )$.
It follows from Lemma \ref{refined Sobolev} that $\|u_k\|_{L^{p^* , p}({{\mathbb R}^n} )} \le C (n,p) \|\nabla u_k \|_{L^{p}({{\mathbb R}^n} )}$ and $\{u_k\}$ is Cauchy in $L^{p^*, p}(\Omega )$.
Hence $u \in L^{p^*, p}(\Omega )$ and $\|u\|_{L^{p^* , p}(\Omega )} \le C (n,p) \|\nabla u \|_{L^{p}(\Omega )}$. By Lemma \ref{holder inequalies}, we have
$$
\norm{u \mathbf{b}}_{L^p (\Om) }
 \le C \| {\mathbf b} \|_{L^{n,\infty} (\Om ) }\|u \|_{L^{p^* , p}(\Om )} \le C
\norm{\mathbf{b}}_{L^{n,\I}(\Om) }\norm{\nb u}_{L^{p}(\Om)}
$$
for some $C=C(n,p)$. Assume in addition that $\div \mathbf{b} \ge 0$ in $\Om$. Then since $u_k^2 \in C_c^\infty (\Om )$ and $u_k^2 \ge 0$, it follows that
\[
-\int_{\Om} ( u_k \mathbf{b}) \cdot \nb u_k \, dx =-\int_{\Om} \mathbf{b} \cdot \nb \left(\frac 12 u_k^2 \right)\, dx \ge 0.
\]
Hence, if $n \ge 3$ and $p=2$, then since $u_k \bb \to u\bb$ and $\nabla u_k \to \nabla u$ in $L^2 (\Om ;\R^n )$, we have
\[
-\int_{\Om} ( u \mathbf{b}) \cdot \nb u \, dx = - \lim_{k\to\infty}\int_{\Om} ( u_k \mathbf{b}) \cdot \nb u_k \, dx \ge 0 .
\qedhere
\]
\end{proof}

\medskip

To prove our $W^{1,n+\e}$- and $W^{2, n/2+\delta}$-regularity results in Theorem \ref{th4}, we shall make crucial use of the following estimate, which is a special case of the Miranda-Nirenberg interpolation inequalities \cite{Miranda, Nirenberg}. %

\begin{lemma} \label{MNineq}
Let $\Om $ be a bounded Lipschitz domain in $\R^n$. Let $1 \le p<n$ and $0< \alpha < 1$. Then
 for every $u \in W^{2,p}(\Omega )\cap C^\alpha (\overline \Om)$, we have
\begin{equation}\label{eq2.4}
\nb u \in L^r (\Om; \R^n )\quad\mbox{and}\quad \norm{\nb u}_{L^r(\Om)} \le C \left(%
 \norm{ u}_{W^{2,p}(\Om)}+\norm{ u}_{C^\al(\overline \Om)} \right)
\end{equation}
for some $C=C(n, p , \alpha , \Omega )$, where $ r = (2-\al)p /(1-\al )$.
\end{lemma}

This inequality is particularly useful when $1 \le p \le n/2$.
It should be emphasized that if $1 \le p \le  n/2$, then $r>2p \ge  p^*$; hence the Miranda-Nirenberg embedding $ W^{2,p}(\Omega )\cap C^\alpha (\overline \Om) \subset W^{1,r}(\Omega )$ does not follow from the standard Sobolev embedding $W^{2,p}(\Om )\subset W^{1,p^*}(\Om )$.

\subsection{The Calderon-Zygmund estimates}
\label{sec:CZ-est}

The following is the well-known Calderon-Zygmund result for $p$-weak solutions of the Poisson equation; see \cite[Theorem 1.1]{JK} and \cite[Theorem 9.15]{GilTru} for instance.

\begin{lemma}\label{CZ-estimates0}
Let $\Om $ be a bounded $C^1$-domain in $\R^n , n \ge 2$.
Assume that
$1<p<\infty$. Then for every $f \in W^{-1,p}(\Om)$, there exists a unique $u\in W^{1,p}_0 (\Om)$ such that
$$
\int_\Om \nb u \cdot \nb \phi \, dx = \bka{f,\phi} \quad \text{for all } \phi \in C_c^{1} (\Om).
$$
 Moreover, we have
$$
\|u\|_{W^{1,p}(\Omega )} \le C(n, p, \Omega ) \|f\|_{W^{-1,p}(\Omega )}.
$$
Assume in addition that $\Omega$ is a $C^{1,1}$-domain and $f \in L^{p} (\Omega)$.
Then
$$
u \in W^{2, p}(\Om) \quad\mbox{and}\quad \|u\|_{W^{2,p}(\Omega )} \le C(n, p, \Omega ) \|f\|_{L^{p}(\Omega )}.
$$
\end{lemma}

\medskip

 The following result holds for arbitrary bounded domains $\Om$ in $\R^n$, which will be used in the proof of Lemma \ref{th4-boundedness-exterior}, for instance.

\begin{lemma}\label{dual-char}
Let $\Om$ be any bounded domain in ${\mathbb R}^n, n \ge 2$. Assume that $1<p<\infty$. Then for every $f \in W^{-1,p}(\Om)$, there exists ${\mathbf F} \in L^p (\Om ; \R^n )$ such that
$$
f= \div {\mathbf F} \quad\mbox{in}\,\, \Om \quad\mbox{and}\quad \|{\mathbf F}\|_{L^p (\Omega )} \le C\|f\|_{W^{-1,p}(\Om )}
$$
for some constant $C = C ( n,p,  \Om  )$.
\end{lemma}

\begin{proof} Let $f \in W^{-1,p}(\Om)$ be given. Then by the Riesz representation theorem for $W^{-1,p}(\Omega)$ (see \cite[Theorem 10.41, Corollary 10.49]{Le} e.g.), there exist $g \in L^{p}(\Om)$ and ${\mathbf G} \in L^p (\Om ; \R^n )$ such that
$$
f= g +\div {\mathbf G} \quad\mbox{in}\,\,\Om
$$
and
$$
\|g\|_{L^p (\Omega )} + \|{\mathbf G}\|_{L^p (\Omega )} \le C(n,p,  \Om  )\|f\|_{W^{-1,p}(\Om )}.
$$
Choose $x_0 \in \Om$ and $R=2 \,{\rm diam}\, \Om$ so that the ball $B_R= B_R (x_0 )$ contains $\overline{\Om}$. Next we extend $g$ to $B_R$
by defining zero outside $\Om$. Then by Lemma \ref{CZ-estimates0},
there exists $v \in W_0^{1,p}(B_R )\cap W^{2,p}(B_R )$ such that $\De v =g$ in $B_R$ and $\| v \|_{W^{2,p} (B_R )} \le C(n,p,R) \|g\|_{L^p (B_R )}$. Defining $\mathbf{F}= (\nb v )|_{\Om}+\mathbf{G}$, we complete the proof. \end{proof}

\medskip

The following result for the exterior problem seems to be standard nowadays at least for smooth domains. See e.g.~\cite[Theorem 2.10, Remark 2.11]{AmGiGi} for $C^{1,1}$-domains
(and results for other ranges of $p$). %

\begin{lemma}\label{CZ-estimates in exterior domains0}
\begin{enumerate}[(i)]
 \item Let $\Om $ be an exterior $C^1$-domain in $\R^n$, $n \ge 3$. Assume that $n/(n-1)<p<n$. Then for each ${\mathbf F} \in L^{p} (\Omega ; \R^n )$, there exists a unique $u\in \hat{D}^{1,p}_0 (\Om)$ such that
 \begin{equation}
 \label{bvp30}
\int_\Om \nb u \cdot \nb \phi \, dx = - \int_\Om {\mathbf F} \cdot \nb \phi \, dx \quad \text{for all } \phi \in C_c^{1} (\Om).
 \end{equation}
 Moreover, we have
\begin{equation}\label{pq-estimate0}
 \|\nb u \|_{L^{p} (\Om )} + \| u \|_{L^{p^* } (\Om )} \le C(n,p , \Om ) \|{\mathbf F}\|_{L^{p}(\Om )} .
\end{equation}
 \item Let $\Omega$ be an exterior $C^{1,1}$-domain in $\R^n , n \ge 3.$ Assume that $1<q< n/2$. Then for each $f \in L^{q} (\Omega)$, there exists a unique $u\in \hat{D}^{1,q^*}_0 (\Om) \cap D^{2,q} (\Om )$ such that
$$
-\Delta u = f \quad\mbox{a.e. in} \,\,\Omega .
$$
Moreover, we have
\begin{equation}\label{pq-estimate1}
 \|\nb^2 u\|_{ L^{q} (\Om ) } + \|\nb u\|_{ L^{q^* } (\Om ) } + \| u\|_{ L^{(q^*)^* } (\Om ) } \le C \|f\|_{L^{q} (\Omega )}
\end{equation}
for some $C= C(n,q ,\Om )$.
\end{enumerate}
\end{lemma}

\medskip
 For the sake of convenience of readers, we provide a sketch of the proof of Lemma \ref{CZ-estimates in exterior domains0}. We begin with a general uniqueness result.

\begin{lemma}\label{uniqueness in Lorentz spaces for exterior domains} Let $\Om $ be an exterior $C^1$-domain in $\R^n , n \ge 3$. Assume that
\begin{enumerate}[(i)]
 \item $u \in W_{\loc}^{1,r}(\overline{\Om})$ for some $r>1$, $u =0$ on $\partial \Om$;
 \item $u \in L^{p_1 } (\Om ) + L^{p_2 } (\Om )$ for some $1 < p_1 \le p_2 < \infty$; and
 \item $u$ satisfies
$$
\int_\Om \nb u \cdot \nb \phi \, dx = 0 \quad \text{for all } \phi \in C^1_{c}(\Om).
$$
\end{enumerate}

 \noindent
 Then $u =0$ identically on $\Om$.
\end{lemma}

\begin{remark}
It follows from (\ref{Lebesgue-Lorentz inclusion}) and (\ref{inter-Loren}) that the condition (ii) is equivalent to the following condition:

$\qquad $ (ii)$' $ $ u \in L^{q_1 , \I } (\Om ) + L^{q_2 , \I } (\Om )$ for some $1 < q_1 \le q_2 < \infty$.

\end{remark}

\begin{proof}
[Proof of Lemma \ref{uniqueness in Lorentz spaces for exterior domains}]
By the uniqueness of weak solutions of the Laplace equation (see the proof of Lemma \ref{th4-existence} below), it suffices to show that
$$
u \in \hat{D}_0^{1,2}(\Om ).
$$

First of all, it follows from Weyl's lemma that $u \in C^\infty (\Om )$ and $\Delta u =0$ in $\Om$. Moreover, by a bootstrap argument based on Lemma \ref{CZ-estimates0}, we deduce that $u \in W_{\loc}^{1,q}(\overline{\Omega}) $ for any $q< \infty $ (for details see the proof of Lemma \ref{th4-unique-q version} below).

Let us choose any $\eta \in C_c^\infty (\R^n ; [0,1])$ such that $\eta =1$ on $\Om^c$. Then $v=(1-\eta ) u$ satisfies
$$
 - \De v = f:= u \De \eta \,+\, 2 \nabla u \cdot \nabla \eta \quad \text{in }\R^n .
$$
Let $w $ be the Newtonian potential of $f$ in $\R^n$, so that $-\De w =f$ in $\R^n$. Then since $f\in C_c^\infty (\R^n )$, it follows from the Calderon-Zygmund estimate and the Sobolev inequality that  $w \in D^{2,q} (\R^n )\cap D^{1,q^*}(\R^n )\cap L^{(q^{*})^* }(\R^n ) $ for any $1<q<n/2$. Note that $v-w$ is harmonic in $\R^n$ and belongs to $L^{q_1} (\R^n )+ L^{q_2}(\R^n )$ for some $q_1 , q_2$ with $1< q_1 <p_1 \le p_2 < q_2 < \infty$. Hence by the Liouville theorem (see the proof of \cite[Lemma 2.6]{KoYa0} e.g.), we deduce that $v=w$.
Since $u=v$ outside a large ball containing $\Om^c$ and $u=0$ on $\partial\Omega$, it follows that
$$
u \in \hat{D}_0^{1,q^{*}}(\Om )  = \left\{ v \in D^{1,q^*}(\Om )\cap L^{(q^{*})^{*}}(\Om ) \, :\, v =0 \,\,\mbox{on}\,\,\partial \Omega \right\}
$$
for any $1<q<n/2$. In particular, taking $q =2n/(n+2)$, we conclude that $
u\in \hat{D}_0^{1,2}(\Om )$.
This completes the proof. \end{proof}

\medskip

\begin{proof}
[Proof of Lemma \ref{CZ-estimates in exterior domains0}]
Assume that ${\mathbf F} \in C_c^\infty (\Omega ; \R^n ) $. Then by the Lax-Milgram theorem, there exists a unique $u\in \hat{D}^{1,2}_0 (\Om) $ satisfying (\ref{bvp30}).
By the interior regularity theory, we deduce that $u \in C^\infty (\Om )$ and $- \Delta u =\div {\mathbf F}$ in $\Om$.
Moreover, adapting the proof of Lemma \ref{uniqueness in Lorentz spaces for exterior domains}, we also deduce that
$u \in D^{2,q}(\Om \setminus B_R )\cap \hat{D}_{0}^{1,p}(\Om) $ for any $1<q<n/2$ and $1< p< n $, provided that  $R>0$ is so large that $\Om^c \subset B_R =B_R (0)$.

Let $u_1$ be the Newtonian potential of $-\div {\mathbf F}$ in $\R^n$. Then it follows from the Calderon-Zygmund theory that
$$
u_1 \in D^{1,p}(\R^n) \cap L^{p^*}(\R^n ) \quad\mbox{and}\quad \|\nb u_1 \|_{L^{p} (\R^n )} + \| u_1 \|_{L^{p^* } (\R^n )} \le C \|{\mathbf F}\|_{L^{p}(\R^n )}
$$
for any $1 < p< n $.

We now derive the crucial estimate (\ref{pq-estimate0}).
Let $2 \le p< n$ be given. Choose any $R>0$ and $\eta \in C_c^\infty (B_{2R} ; [0,1])$ such that $\Om^c \subset B_R$ and $\eta =1$ on $B_R$. Then since $u_1 \in W^{1,p}(\Om_{R}) $ and $\| u_1 \|_{W^{1, p}(\Om_{R})} \le C \|{\mathbf F}\|_{L^{p}(\Om )}$, it follows from the classical trace theorem that there exists $u_2 \in W^{1,p}(\Om )$ with support in $\overline{\Om}_{R}$ such that $u_2 =u_1$ on $\partial \Om$ and
$\| u_2 \|_{W^{1,p}(\Om )} \le C (n,p,R, \Om) \|{\mathbf F}\|_{L^{p}(\Om )}$.
Define $\overline{u}=u + u_1 -u_2$. Then
$$
\overline{u} \in \hat{D}_{0}^{1,p}(\Om) \quad\mbox{and}\quad - \Delta \overline{u} =\div \overline{\mathbf F} \quad\mbox{in}\,\,\Om ,
$$
where $\overline{\mathbf F} = \nabla u_2 $.
Moreover, since $\overline{u}\in D^{1,2}_0 (\Om) \cap L^{2^* }(\Om )$,
we have
$$
 \|\nb \overline{u} \|_{L^{2} (\Om )} + \| \overline{u} \|_{L^{2^* } (\Om )} \le C(n, \Om ) \|\overline{\mathbf F}\|_{L^{2}(\Om )} \le C (n,p,R, \Om) \|{\mathbf F}\|_{L^{p}(\Om )}.
$$
We also define $v= \eta \overline{u}$ and $w= (1-\eta ) \overline{u}$. Then since $\overline{\mathbf F}$ has support in $\overline{\Om}_{R}$, we have
$$
-\De v = f_1 := \eta \, \div \overline{\mathbf F} -\overline{u} \De \eta -2 \nabla \overline{u} \cdot \nabla \eta \quad\mbox{in } \Om_{2R}
$$
and
$$
 - \De w = f_2 := \overline{u} \De \eta + 2 \nabla \overline{u} \cdot \nabla \eta \quad \text{in }\R^n .
$$

Suppose that $p \le 2^*$. Then by the Sobolev inequality,
$$
\| f_1 \|_{W^{-1,p}(\Om_{2R})} \le C \left( \|\overline{\mathbf F}\|_{L^{p}(\Om )} + \|\nabla \overline{u} \|_{L^2 (\Om_{2R} )} \right) \le C (n,p,R, \Om) \|{\mathbf F}\|_{L^{p}(\Om )}.
$$
Hence by Lemma \ref{CZ-estimates0}, we obtain
$$
\| \overline{u} \|_{W^{1,p}(\Om_{R})} \le \| v\|_{W^{1,p}(\Om_{2R})} \le C (n,p,R, \Om) \|{\mathbf F}\|_{L^{p}(\Om )} .
$$
It follows from the definition of $\overline{u}$ that if $p \le 2^*$, then
$$
\| u \|_{W^{1,p}(\Om_{R})} \le C (n,p,R, \Om) \|{\mathbf F}\|_{L^{p}(\Om )}
$$
for all large $R>0$. This estimate can also be proved for $ p > 2^*$, by a standard bootstrap argument.
Note   that
$$
f_2 \in L^{np/(n+p)}(\R^n ) \quad\mbox{and}\quad \| f_2 \|_{L^{np/(n+p)}(\R^n )} \le C \|\nabla \overline{u}\|_{L^{p}(\Om_{2R} )} .
$$
Hence by the Calderon-Zygmund estimate again, we have
$$
 \| \nabla w\|_{L^{p}(\R^n ) } + \| w\|_{L^{p^*}(\R^n ) } \le C (n,p,R, \Om) \|{\mathbf F}\|_{L^{p}(\Om )} .
$$
We have derived (\ref{pq-estimate0}) for $2 \le p<n$. The estimate (\ref{pq-estimate0}) also holds for $n/(n-1)< p <2$ by a duality argument. Then by a standard density argument, we deduce that for each ${\mathbf F} \in L^{p} (\Omega ; \R^n )$, there exists a function $u$ in $\hat{D}^{1,p}_0 (\Om) $ satisfying (\ref{bvp30}) and (\ref{pq-estimate0}). The uniqueness of such a function $u$ immediately follows from Lemma \ref{uniqueness in Lorentz spaces for exterior domains}. This completes the first assertion of the lemma. The proof of the second one is similar and omitted.
\end{proof}

\subsection{The Calderon-Zygmund estimates in Lorentz spaces}

Particularly, the proof of Theorem \ref{th4-q version} is based on the Calderon-Zygmund estimates in Lorentz spaces for elliptic equations.

Let $\Om $ be any domain in $\R^n$. For $m \in \N$, $1 <  p<\infty$, and $1 \le q \le \infty$, let $W^{m,p,q}(\Om )$ be the space of all functions $u$ on $\Om$ such that
$D^\alpha u$ exists and belongs to $L^{p,q}(\Om )$ for all indices $\alpha$ with $|\alpha|\le m$.
The Sobolev-Lorentz space $W^{m,p,q}(\Om )$ is a Banach space equipped with the natural norm.
Its homogeneous version $D^{m,p,q}(\Om )$ is the space of all $u \in L_{\loc}^{p,q}(\Om )$ such that
$D^\alpha u \in L^{p,q}(\Om )$ for $|\alpha| = m$.

Assume first that $\Om $ is a bounded Lipschitz domain in $\R^n$. Then since $ W^{1,p,q}(\Om ) \subset W^{1,r}(\Om)$ for any $1<r<p$, every $u \in W^{1,p,q}(\Om )$ has a well-defined trace on $\partial\Omega$.
Let $W^{1,p,q}_0 (\Om)$ be a subspace of $ W^{1,p,q}(\Om ) $ consisting of all $u \in W^{1,p,q}(\Om )$ with $u=0$ on $\partial\Om$. It is remarked that if $\Om$ is sufficiently smooth, then the spaces $W^{m,p,q}(\Om )$ and $W^{1,p,q}_0 (\Om)$ have the same real interpolation property as (\ref{inter-Loren}); see \cite{AdFo} and \cite{KKP} for more details.

Assume next that $\Om$ is an exterior Lipschitz domain in $\R^n$. Then since $D^{m,p,q}(\Om ) \subset W^{m,p,q}( \Om_0 )$ for every bounded $\Omega_0 \subset \Om$, every $u \in D^{1,p,q}(\Om )$ has a well-defined trace on $\partial\Omega$. Denote by $ D^{1,p,q}_0 (\Om )$
the space of all $u \in D^{1, p , q }(\Om )$ with $u =0$ on $\partial \Om$. It can be shown (see \cite[Corollary 3.5]{KKP} e.g.) that if $1 <  p<n$ and $ u \in D^{1,p,q}_0 (\Om )$,  then there is a unique constant $c$ such that $u -c\in L^{p^* , q }(\Om )$ and $\|u-c \|_{L^{p^* ,q}(\Om )} \le C(n,p,q ) \|\nb u \|_{L^{p,q}(\Om )}$.

\medskip

The classical Calderon-Zygmund result can be easily interpolated by real method to obtain the following result (see \cite[Lemmas 3.3, 3.4]{KKP}).
\begin{lemma}%
\label{sol-Laplace in Rn}
\mbox{}
\begin{enumerate}[(i)]
 \item Let $1 < p < n$ and $1 \le q \le \infty$. Then for each ${\mathbf F}\in L^{p,q}(\R^n ; \R^n)$, there exists a unique $u \in D^{1,p,q}(\R^n) \cap L^{p^*,q}(\R^n)$ such that
$ -\Delta u = {\rm div} \,{\mathbf F}$ in $\R^n$. Moreover, we have
$$
\| u \|_{L^{p^* ,q}(\R^n )} + \| \nabla u \|_{L^{p ,q}(\R^n )} \le C \| {\mathbf F} \|_{L^{p ,q}(\R^n )}
$$
for some constant $C = C(n,p,q)$.

 \item Let $n \ge 3, 1 < p <n/2 $ and $1 \le q \le \infty$. Then for each $f\in L^{p,q}(\R^n) $, there exists a unique
$ u \in D^{2,p,q}(\R^n ) \cap D^{1,p^* ,q}(\R^n ) \cap L^{(p^{*})^* ,q}(\R^n ) $ such that
$-\Delta u = f$ in $\R^ n$.
 Moreover, we have
$$
\| u\|_{L^{(p^*)^* ,q}(\R^n )} + \| \nabla u \|_{L^{p^* ,q}(\R^n )} + \|\nabla^2 u\|_{L^{p ,q}(\R^n )} \le C \| f\|_{L^{p ,q}(\R^n )}
$$
for some constant $C=C(n,p,q) $.
\end{enumerate}
\end{lemma}

\medskip
The following two results can be also deduced by real interpolation from the Calderon-Zygmund estimates in Lemmas \ref{CZ-estimates0} and \ref{CZ-estimates in exterior domains0}.

\begin{proposition}\label{CZ-estimates} Let $\Om $ be a bounded $C^1$-domain in $\R^n$.
Assume that
$1<p<\infty$ and $1 \le q \le \infty$. Then for each ${\mathbf F} \in L^{p,q} (\Omega ; \R^n )$, there exists a unique $u\in W^{1,p,q}_0 (\Om)$ such that
\begin{equation}
 \label{bvp2}
\int_\Om \nb u \cdot \nb \phi \, dx = - \int_\Om {\mathbf F} \cdot \nb \phi \, dx \quad \text{for all } \phi \in C_c^{1} (\Om).
\end{equation}
 Moreover, we have
\begin{equation}
 \label{pq-estimate2}
\|u\|_{W^{1,p,q}(\Omega )} \le C(n, p, q, \Omega ) \|{\mathbf F}\|_{L^{p,q}(\Omega )}.
\end{equation}
Assume in addition that $\Omega$ is a $C^{1,1}$-domain and $f= \div {\mathbf F} \in L^{p,q} (\Omega)$. Then
$$
u \in W^{2, p,q}(\Om) \quad\mbox{and}\quad \|u\|_{W^{2,p,q}(\Omega )} \le C(n, p, q, \Omega ) \|f\|_{L^{p,q}(\Omega )}.
$$
\end{proposition}

\begin{proposition}%
\label{CZ-estimates in exterior domains}
\mbox{}

\begin{enumerate}[(i)]
 \item Let $\Om $ be an exterior $C^1$-domain in $\R^n$, $n \ge 3$. Assume that $n/(n-1)<p<n$ and $1 \le q \le \infty$. Then for each ${\mathbf F} \in L^{p,q} (\Omega ; \R^n )$, there exists a unique $u\in D^{1,p,q}_0 (\Om) \cap L^{p^* , q }(\Om )$
      such that
 \begin{equation}
 \label{bvp3}
\int_\Om \nb u \cdot \nb \phi \, dx = - \int_\Om {\mathbf F} \cdot \nb \phi \, dx \quad \text{for all } \phi \in C_0^{1} (\Om).
 \end{equation}
 Moreover, we have
\begin{equation}\label{pq-estimate}
 \|\nb u \|_{L^{p,q} (\Om )} + \| u \|_{L^{p^* ,q} (\Om )} \le C(n,p ,q,\Om ) \|{\mathbf F}\|_{L^{p,q}(\Om )} .
\end{equation}
 \item Let $\Omega$ be an exterior $C^{1,1}$-domain in $\R^n , n \ge 3.$ Assume that $1<r< n/2$ and $1 \le s \le \infty$. Then for each $f \in L^{r,s} (\Omega)$, there exists a unique
$$
 u \in D^{2,r,s} (\Om ) \cap D_0^{1,r^* ,s} (\Om )\cap L^{(r^* )^* ,s} (\Om )
$$
such that
$$
-\Delta u = f \quad\mbox{a.e. in} \,\,\Omega .
$$
Moreover, we have
$$
 \|\nb^2 u\|_{ L^{r,s} (\Om ) } + \|\nb u\|_{ L^{r^* ,s} (\Om ) } + \| u\|_{ L^{(r^*)^* ,s} (\Om ) } \le C \|f\|_{L^{r,s} (\Omega )}
$$
for some $C= C(n,r ,s,\Om )$.
\end{enumerate}
\end{proposition}

\begin{proof}
[Proof of Proposition \ref{CZ-estimates in exterior domains}]
 Assume that $n/(n-1)<p<n$ and $1 \le q \le \infty$. We first prove  existence of a function $u\in D^{1,p,q}_0 (\Om) \cap L^{p^* , q }(\Om )$ satisfying (\ref{bvp3}) and (\ref{pq-estimate}) for each $\mathbf{F} \in L^{p,q}(\Om; \R^n )$.
Choose $p_1 , p_2$ such that $n/(n-1) < p_1 < p < p_2 < n$. Then from Lemma \ref{CZ-estimates in exterior domains0}, it follows that for each $\mathbf{F} \in L^{p_i}(\Om; \R^n )$ there exists a unique $u=T_{p_i} ({\mathbf F}) \in \hat{D}^{1,p_i }_0 (\Om) $ satisfying (\ref{bvp3}). Moreover, the solution operator $T_i = T_{p_i}$ maps $L^{p_i}(\Om; \R^n )$ into $L^{p_i^* }(\Om )$ linearly and boundedly. Define $S_i (\mathbf{F})= \nb T_i (\mathbf{F})$ for all $\mathbf{F} \in L^{p_i}(\Om; \R^n )$. Then $S_i$ is a bounded linear operator from $L^{p_i}(\Om; \R^n )$ into $L^{p_i}(\Om; \R^n )$.
The key observation here is that $T_1 = T_2$ and $S_1 =S_2$ on $L^{p_1}(\Om; \R^n ) \cap L^{p_2}(\Om; \R^n )$, which follows from Lemma \ref{uniqueness in Lorentz spaces for exterior domains}. Hence the operators $T_i$ and $S_i$ can be extended uniquely to linear operators $T$ and $S$ from $L^{p_1}(\Om; \R^n ) + L^{p_2}(\Om; \R^n )$ into $ L^{p_1^* }(\Om ) + L^{p_2^* }(\Om )$ and $L^{p_1}(\Om; \R^n ) + L^{p_2}(\Om; \R^n )$, respectively.
By the real interpolation result (\ref{inter-Loren}), we then deduce that $T$ is bounded from $L^{p,q}(\Om; \R^n )$ into $L^{p^* ,q}(\Om )$ and $S$ is bounded from $L^{p,q}(\Om; \R^n )$ into $L^{p,q}(\Om; \R^n )$.

Let $\mathbf{F} \in L^{p,q}(\Om; \R^n )$ be given. Then since
$ T({\mathbf F}) \in D^{1,p_1 }_0 (\Om)+ D^{1,p_2 }_0 (\Om)$ and $\nb T({\mathbf F}) =S({\mathbf F}) \in L^{p,q}(\Om; \R^n )$, it follows that $T({\mathbf F}) \in D^{1,p ,q}_0 (\Om)$.
Hence $u=T({\mathbf F})$ is a function in $D^{1,p,q}_0 (\Om) \cap L^{p^* , q }(\Om )$ satisfying (\ref{bvp3}) and (\ref{pq-estimate}). Uniqueness of such a function follows again from Lemma \ref{uniqueness in Lorentz spaces for exterior domains}.
This completes the proof of Part (i) of the lemma.

To prove Part (ii), assume that $\Omega$ is a $C^{1,1}$-domain, $1<r< n/2$, and $1 \le s \le \infty$.
Choose $r_1 , r_2$ with $1<r_1 < r<r_2 <n/2$. Given $g \in L^{r_i}(\Om ) $, let $\mathbf{G} = -\nb N(g)$, where $N(g)$ is the Newtonian potential of $g$ over $\Om$.
Then it follows from the Calderon-Zygmund theory that $\mathbf{G} \in L^{r_i^*}(\Om )$, $\|\mathbf{G} \|_{L^{r_i^*}(\Om )} \le C \|g\|_{L^{r_i}(\Om )}$, and $g =\div \mathbf{G} $ in $\Om$.
Define $v=T_{r_i^*}({\mathbf G})$. Then since $v \in \hat{D}^{1,r_i^* }_0 (\Om)$
 and $-\De v = g \in L^{r_i}(\Om ) $ in $\Om$, it follows from Lemma \ref{CZ-estimates in exterior domains0} that $v \in D^{2, r_i}(\Om)$ and $ \| \nb^2 v \|_{L^{r_i }(\Om)} \le C \|g\|_{L^{r_i}(\Om )} $.
For each $i=1,2$, let us define $\tilde{T}_i (g)= T_{r_i^*}( -\nb N(g))$ for each $g \in L^{r_i}(\Om ) $.
Then by Lemma \ref{uniqueness in Lorentz spaces for exterior domains}, we deduce that $\tilde{T}_1 = \tilde{T}_2$ on $L^{r_1 }(\Om ) \cap L^{r_2 }(\Om )$. Hence there exists a unique linear operator $\tilde{T}$ on  $ L^{r_1}(\Om ) + L^{r_2}(\Om )$ that extends both $\tilde{T}_1$ and $\tilde{T}_2$. Moreover, since $\tilde{T}_i$ is bounded from $ L^{r_i}(\Om ) $ into $D^{2, r_i}(\Om) \cap \hat{D}^{1,r_i^* }_0 (\Om)$, it follows from the real interpolation theory that $\tilde{T}$ is bounded from $ L^{r,s}(\Om ) $ into $D^{2,r,s} (\Om ) \cap D_0^{1,r^* ,s} (\Om )\cap L^{(r^* )^* ,s} (\Om )$.
It is obvious that if $f \in L^{r,s}(\Om)$, then $  u = \tilde{T}(f)$ is a function satisfying   all the desired properties. Uniqueness of such a function  follows again from Lemma \ref{uniqueness in Lorentz spaces for exterior domains}.
 This completes the proof.
\end{proof}

\begin{proof}
[Proof of Proposition \ref{CZ-estimates}]
The proof is exactly the same as that of Proposition \ref{CZ-estimates in exterior domains} and so omitted.
\end{proof}

\section{Boundeness and higher integrability of weak solutions}
\label{Sec4}

In this section, we establish several existence and regularity results for weak solutions of the Dirichlet problem \eqref{bvp} and its dual problem \eqref{bvp-dual}.

Existence of weak solutions is easily deduced from Lemma \ref{th1-general}, by applying the Lax-Milgram theorem.

\begin{lemma}\label{th4-existence} Let $\Om $ be any domain in $\R^n$, $n \ge 3$. Assume that $\mathbf{b} \in L^{n,\I}(\Om;\R^n)$ and $\div \mathbf{b} \ge 0$ in $\Om$.
\begin{enumerate}[(i)]
 \item For every ${\mathbf F} \in L^{2}(\Om ;\R^n ) $, there exists a unique weak solution $u $ of \eqref{bvp} with $f=\div {\mathbf F}$. Moreover, we have
\begin{equation}
\label{energy-estimate}
\|\nb u \|_{L^2 (\Om )} \le \|{\mathbf F} \|_{L^{2}(\Om )}.
\end{equation}
 \item For every ${\mathbf G} \in L^{2}(\Om ;\R^n ) $, there exists a unique weak solution $v $ of \eqref{bvp-dual} with $g=\div {\mathbf G}$. Moreover, we have
$$
\|\nb v \|_{L^2 (\Om )} \le \|{\mathbf G} \|_{L^{2}(\Om )}.
$$
\end{enumerate}
\end{lemma}

\begin{proof} Let ${\mathbf F} \in L^{2}(\Om ;\R^n ) $ be given. By Lemma \ref{th1-general}, the bilinear form
$$
B(u,v)= \int_\Om \bke{\nb u - u \mathbf{b}} \cdot \nb
v \, dx
$$
is well-defined, bounded, and coercive on $\hat{D}_0^{1,2}(\Om )$. Hence it follows from the Lax-Milgram theorem that there exists a unique $u \in \hat{D}_0^{1,2}(\Om )$ such that
\begin{equation}\label{weak for-LM}
B(u,v) =-\int_\Om {\mathbf F} \cdot \nb v \, dx
\end{equation}
for all $v \in \hat{D}_0^{1,2}(\Om )$. This function $u$ is a weak solution of \eqref{bvp} with $f=\div {\mathbf F}$ in the sense of Definition \ref{def2.1}. Moreover, taking $v =u$ in (\ref{weak for-LM}), we have
\[
\int_\Omega |\nabla u|^2 \, dx \le B(u,u) =-\int_\Om {\mathbf F} \cdot \nb u \, dx,
\]
from which the estimate \eqref{energy-estimate} follows by the Cauchy-Schwartz inequality.

It remains to prove uniqueness of a weak solution.
 Let $w $ be a weak solution of \eqref{bvp} with $f=\div {\mathbf F}$. Given $v \in \hat{D}_0^{1,2}(\Om )$, let $\{v_k\}$ be a sequence in $C_c^\infty (\Om )$ converging to $v$ in $\hat{D}_0^{1,2}(\Om )$. Then by Lemma \ref{th1-general} again,
\begin{align*}
 B(w,v ) & = \lim_{k \to\infty} B\left(w, v_k \right) \\
&= - \lim_{k \to\infty} \int_\Om {\mathbf F} \cdot \nb v_k \, dx= -\int_\Om {\mathbf F} \cdot \nb v \, dx .
\end{align*}
Hence $w$ satisfies (\ref{weak for-LM}) for all $v \in \hat{D}_0^{1,2}(\Om )$. By the uniqueness assertion of the Lax-Milgram theorem, we conclude that $w=u$. This completes the proof of the first part (i). The proof of the second one is exactly the same and so omitted.
\end{proof}

\subsection{Boundedness and higher integrability}

First, by standard iteration techniques, we prove boundedness of weak solutions of the problems  \eqref{bvp} and \eqref{bvp-dual} on bounded domains.

\begin{lemma}\label{th4-boundedness} Let $\Om $ be any bounded domain in $\R^n$, $n \ge 3$. Assume that $\mathbf{b} \in L^{n,\I}(\Om;\R^n)$, $\div \mathbf{b} \ge 0$ in $\Om$, and $n<p< \infty$.
\begin{enumerate}[(i)]
 \item For every ${\mathbf F} \in L^{p}(\Om ;\R^n ) $, there exists a unique weak solution $u $ of \eqref{bvp} with $f=\div {\mathbf F}$. Moreover, we have
\begin{equation}\label{global bound}
u \in L^\infty (\Om ) \quad\mbox{and}\quad \|u\|_{L^\infty (\Om )} \le C(n,p ) ({\rm diam}\, \Om )^{1-n/p} \|{\mathbf F}\|_{L^p (\Omega )} .
\end{equation}
 \item For every ${\mathbf G} \in L^{p}(\Om ;\R^n ) $, there exists a unique weak solution $v $ of \eqref{bvp-dual} with $g=\div {\mathbf G}$. Moreover, we have
$$%
v \in L^\infty (\Om ) \quad\mbox{and}\quad \|v\|_{L^\infty (\Om )} \le C(n,p ) ({\rm diam}\, \Om )^{1-n/p} \|{\mathbf G}\|_{L^p (\Omega )} .
$$%
\end{enumerate}
\end{lemma}

\begin{proof}[Proof of Lemma \ref{th4-boundedness} (i)] By scaling, we may assume that ${\rm diam}\, \Om =1$. Let ${\mathbf F} \in L^{p}(\Om ;\R^n ) $ be given. Then since $\Omega$ is bounded and $ p>2$, it follows from  Lemma \ref{th4-existence} that  there exists a unique weak solution $u$ of \eqref{bvp}, which satisfies
$$
\|\nb u \|_{L^2 (\Om )} \le \|{\mathbf F}\|_{L^2 (\Omega )} \le C( p )K,
$$
where $K = \|{\mathbf F}\|_{L^{p}(\Om )}$. We may assume that $K>0$.
To prove boundedness of $u$, we first consider the positive part of $u$, $u^+ = \max \{u,0 \}$. First, from the proof of Lemma \ref{th4-existence}, we recall that
\begin{equation}\label{weak-form-02}
\int_\Om \nb u \cdot \nb \phi \, dx = \int_\Om \left( u \mathbf{b} - \mathbf{F} \right) \cdot \nb
 \phi \, dx \quad \text{for all } \phi \in W^{1,2}_{0}(\Om) .
\end{equation}

For a fixed number $N>K $, let $w=T_N (u) +K$, where $T_N :\R \to [0,N]$ is defined by $T_N (t)= 0$ for $t<0$, $T_N (t)=t$ for $0 \le t \le N$, and $T_N (t)=N$ for $t>N$. Since $T_N$ is piecewise linear and $T_N (0)=0$, we easily show that $w -K  =T_N (u)\in W_0^{1,2}(\Om )$ and $\nb w = \chi_{\{0< u < N \}} \nb u$, by adapting the proofs of \cite[Lemmas 7.5, 7.6]{GilTru}.
Next, for a number $\beta \ge 1$, let $G$ be a $C^1$-function on $[0,\infty )$ such that $G(t)=t^\beta -K^\beta$ for $0 \le t \le N+K$ and $G'(t)= G'(N+K)$ for $t >N+K$.
Then since $G'$ is bounded on $[0,\infty )$ and $G(K)=0$, it follows that $G(w) \in W_0^{1,2}(\Om )$ and $\nb G(w)= \beta w^{\beta-1} \nb w =\chi_{\{0<u<N\}} \beta w^{\beta-1} \nb u$.
Hence taking $\phi =G(w)$ in (\ref{weak-form-02}), we obtain
\begin{equation}\label{weak-form-test2}
\int_\Om w^{\beta -1} |\nb w|^2 \,dx = \int_\Om w^{\beta -1} (w-K) \mathbf{b} \cdot \nb w \, dx - \int_\Om w^{\beta -1} \mathbf{F} \cdot \nb w \, dx .
\end{equation}
Let us write
$$
w^{\beta -1} (w-K) \nb w = \nb H(w) ,
$$
where
$$
H(t)= \frac 1{\beta +1} \left( t^{\beta +1} -K^{\beta +1} \right) -\frac{K}{\beta}\left( t^{\beta } -K^{\beta } \right)
$$
for $t>0$. Since $H'(t) = (t-K)t^{\beta-1}$, it follows that $H(t)>0$ for all $t \neq K$.
By L'Hospital's rule,%
\begin{align*}
\lim_{t \to K^{+}} \frac{H(t)^{1/2}}{t-K}&= \lim_{t \to K^{+}} \left[H(t)^{1/2}\right]' =\left\{\lim_{t \to K^{+}} \frac{[H'(t)]^2}{4 H(t)}\right\}^{1/2}\\
&= \left[\lim_{t \to K^{+}} \frac 12 H''(t) \right]^{1/2} = \left(\frac 12 K^{\beta -1}\right)^{1/2},
\end{align*}
which shows that $H^{1/2}$ is a $C^1$-function on $ [K,\infty )$. Since $H^{1/2} (K)=0$, $K \le w \le K+N$ in $\Om$, and $w-K \in W_0^{1,2}(\Om )$, it follows that $ H(w)^{1/2} \in W_0^{1,2}(\Om )$. Hence recalling that $\div {\mathbf b} \ge 0$ in $\Omega$, we deduce from (\ref{weak-form-test2}) and Lemma \ref{th1-general}
that
$$
\int_\Om w^{\beta -1} |\nb w|^2 \,dx \le - \int_\Om w^{\beta -1} \mathbf{F} \cdot \nb w \, dx
$$
and so
$$
\int_\Om w^{\beta -1} |\nb w|^2 \,dx \le \int_\Om w^{\beta -1} | \mathbf{F}|^2 \, dx .
$$

Iterating this inequality, which does not involve $\bb$, we will derive the $L^\infty$-estimate (\ref{global bound}).
Note first that $w^{(\beta +1)/2}-K^{(\beta +1)/2} \in W_0^{1,2}(\Om )$. Hence by Sobolev's inequality,
\begin{align*}
 \left\| w^{(\beta +1)/2}-K^{(\beta +1)/2} \right\|_{L^{2^*} (\Om )}^2
& \le 4 C \int_\Om \left|\nb w^{(\beta +1)/2} \right|^2 \, dx \\
& = C(\beta+1)^2 \int_\Om w^{\beta -1} |\nb w|^2 \,dx \\
& \le C(\beta+1)^2 \int_\Om w^{\beta -1} | \mathbf{F}|^2 \, dx.
\end{align*}
Since $K = \|{\mathbf F}\|_{L^{p}(\Om )}\le w$, we have
\begin{align*}
 \left\| w^{(\beta +1)/2}-K^{(\beta +1)/2} \right\|_{L^{2^*} (\Om )}^2 &
 \le C (\beta+1)^2 \int_\Om w^{\beta +1} \left|\frac{\mathbf{F}}{K}\right|^2 \, dx \\
& \le C (\beta+1)^2 \left\| w^{\beta +1} \right\|_{L^{p/{(p-2)}}(\Om )} .
\end{align*}
Recalling again that $K \le w$, we easily obtain
$$
 \left\| w^{(\beta +1)/2} \right\|_{L^{2^*} (\Om )}
 \le C (\beta+1) \left\|w^{\beta +1} \right\|_{L^{p/{(p-2)}}(\Om )}^{1/2}
$$
or equivalently
$$
 \left\| w \right\|_{L^{\gamma \chi q} (\Om )}^{\gamma}
 \le C \gamma \left\| w \right\|_{L^{\gamma q}(\Om )}^{\gamma},
$$
where
$$
\gamma =\frac{\beta +1}2 \ge 1, \quad  q =\frac{2p}{p-2}< 2^* ,\quad \chi = \frac{n(p-2)}{p(n-2)}>1 ,
$$
and   $C>1 $ is a constant depending  only on $n$ and $p$. Now, taking $\gamma =1, \chi, ..., \chi^{m-1} $, we have
$$
\left\| w \right\|_{L^{ \chi^m q} (\Om )} \le \prod_{k=0}^{m-1} \left( C \chi^k \right)^{1/\chi^k} \left\| w \right\|_{L^{ q}(\Om )}\le C^* \left\| w \right\|_{L^{ q}(\Om )}
$$
for each $m \in \N$, where
$$
C^*= \prod_{k=0}^{\infty} \left( C \chi^k \right)^{1/\chi^k} =C^{\sum_{k=0}^\infty \chi^{-k}}\chi^{\sum_{k=0}^\infty k \chi^{-k}} <\infty .
$$
Letting $m \to \infty$, we deduce that
$$
\left\| T_N (u)+ K \right\|_{L^{\infty} (\Om )} = \left\|w \right\|_{L^{\I } (\Om )}\le C^* \left\| T_N (u)+K \right\|_{L^{ q}(\Om )}
$$
Therefore, letting $N\to\infty$, we conclude that
$$
\left\|u^+ \right\|_{L^{\infty} (\Om )} \le C \left( \| u^+ \|_{L^{2^*}(\Om )} + K \right) \le C(n,p  ) \|{\mathbf F}\|_{L^{p}(\Om )}.
$$
Similarly, the negative part $u^-$ satisfies the same estimate. This completes the proof of  Part (i) of  the lemma.
\end{proof}

\begin{proof}[Proof of Lemma \ref{th4-boundedness} (ii)]
By scaling, we may assume that ${\rm diam}\, \Om =1$.
Let ${\mathbf G} \in L^p (\Om ; \R^n )$ be given. Without loss of generality, we may assume by linearity that $\|{\mathbf G}\|_{L^p (\Omega )} \le 1$. Then  by Lemma \ref{th4-existence} and its proof, there exists a unique weak solution $v$ of \eqref{bvp-dual}, which satisfies
\begin{equation}\label{weak-form-02-dual}
\int_\Om \nb v \cdot \nb \psi \, dx = \int_\Om \left( \psi \mathbf{b} \cdot \nabla v - \mathbf{G} \cdot \nabla \psi \right) \, dx \quad \text{for all } \psi \in W^{1,2}_{0}(\Om)
\end{equation}
and
$$%
\|\nb v \|_{L^2 (\Om )} \le \|{\mathbf G}\|_{L^2 (\Omega )} \le C(  p  ).
$$%

The proof of Lemma \ref{th4-boundedness} (i) by Moser's iteration method can not be easily adapted  to prove an $L^\infty$-estimate for $v$. (For example, the integral $\int \psi \bb \cdot \nb v$ with $\psi  =T_N(v)$ cannot be simplified as $\int u \bb \cdot \nb \phi$ with $\phi  =T_N(u)$.)
 Instead, we  utilize  an iteration argument in \cite{KK} based on a lemma due to Stampacchia \cite{Sta}.
For each $l \ge 0$, define
\[
 H_l (t) =\begin{cases}
 \ t -l & \text{if } t \ge l , \\
 \ t + l & \text{if } t \le - l , \\
 \ \,\,\,\,\, 0 & \text{if }  |t| < l.
\end{cases}
\]
Note that $H_l (v) \in W_0^{1,2}(\Omega )$, $H_l (v)=0$ on $A_l^c$, and $\nabla H_l (v) =\nabla v$ in $A_l$, where $A_l = \{x\in \Om : |v(x)| > l\}$. Hence taking $\psi =H_l (v)$ in (\ref{weak-form-02-dual}) and using Lemma \ref{th1-general}, we have
\begin{align*}
\int_\Om |\nabla H_l (v) |^2 \, dx &= \int_\Om H_l (v) \mathbf{b} \cdot \nabla H_l (v) \, dx - \int_\Om \mathbf{G} \cdot \nabla H_l (v) \, dx \\
& \le -\int_\Om \mathbf{G} \cdot \nabla H_l (v) \, dx \\
&\le \|\mathbf{G}\|_{L^p (\Om )} \| \nabla H_l (v) \|_{L^{p'}(\Om)}\\
& \le    \| \nabla H_l (v) \|_{L^{p'}(A_l )}.
\end{align*}
By the H\"{o}lder and Sobolev inequalities, we thus have
$$
\| H_l (v) \|_{L^{2^*}(\Om)} \le C(n   )|A_l |^{1/{p'} -1/2} .
$$
Noting now that if $0 \le l < h$, then $A_h \subset A_l$ and $|H_l (v) |\ge h-l$ on $A_h$, we deduce that
$$
(h-l) |A_h|^{1/2^*} \le C(n  )|A_l |^{1/{p'} -1/2}
$$
for $0 \le l < h$. Iterating this inequality, we can show that $|A_C | =0 $, that is, $\|v\|_{L^\infty (\Om)} \le C$ for some constant $C=C(n, p  )$; see \cite[Subsection 5.2]{KK} for more details. The proof of Lemma \ref{th4-boundedness} (ii) is complete.
\end{proof}
\medskip

Moser-type arguments can be then used to prove higher integrability of weak solutions of the problems \eqref{bvp} and \eqref{bvp-dual}.

\begin{lemma}\label{th4-higher integrability} Let $\Om $ be any bounded domain in $\R^n$, $n \ge 3$.
Assume that $\mathbf{b} \in L^{n ,\I}(\Om;\R^n)$, $\div \mathbf{b} \ge 0$ in $\Om$, and $2 \le p<n$.
\begin{enumerate}[(i)]
 \item For every ${\mathbf F} \in L^{p}(\Om ;\R^n ) $, there exists a unique weak solution $u$ of \eqref{bvp} with $f=\div {\mathbf F}$. Moreover, we have
\begin{equation}\label{Lp*-estimate for u}
 u \in L^{p^*}(\Om ) \quad\mbox{and}\quad \|u\|_{L^{p^*} (\Om )} \le C(n,p ) \|{\mathbf F}\|_{L^{p}(\Om )}.
\end{equation}
 \item For every ${\mathbf G} \in L^{p}(\Om ;\R^n ) $, there exists a unique weak solution $v$ of \eqref{bvp-dual} with $g=\div {\mathbf G}$. Moreover, we have
$$
 v \in L^{p^*}(\Om ) \quad\mbox{and}\quad \|v\|_{L^{p^*} (\Om )} \le C(n,p ) \|{\mathbf G}\|_{L^{p}(\Om )}.
$$
\end{enumerate}
\end{lemma}

\begin{remark}
 The lemma does not assert that $\nb u$ or $\nb v$ is in $L^p (\Omega ;\R^n )$ for $2<p<n$.
\end{remark}

\begin{proof}
By Lemma \ref{th4-existence}, it remains to prove the $L^{p^*}$-estimates for weak solutions.

Assume that ${\mathbf F} \in C_c^{\infty}(\Om ;\R^n )$. Then by Lemma \ref{th4-boundedness},
there exists a unique weak solution $u\in W_0^{1,2}(\Om )\cap L^\infty (\Om )$ of \eqref{bvp} with $f=\div {\mathbf F}$. Let $w=u^+$, the positive part of $u$. Next, for a number $\beta \ge 1$, let $G$ be a $C^1$-function on $[0,\infty )$ such that $G(t)=t^\beta $ for $0 \le t \le N = \|u\|_{L^\infty (\Om )}$ and $G'(t)= G'(N)$ for $t >\|u\|_{L^\infty (\Om )}$.
Then since $G'$ is bounded on $[0,\infty )$ and $G(0)=0$, it follows that $w^\beta = G(w) \in W_0^{1,2}(\Om )$ and $\nb G(w)= \beta w^{\beta-1} \nb w =\chi_{\{u>0 \}} \beta w^{\beta-1} \nb u$.
Hence taking $\phi =w^\beta $ in (\ref{weak-form-02}), we obtain
$$
\int_\Om w^{\beta -1} |\nb w|^2 \,dx = \int_\Om w^{\beta } \mathbf{b} \cdot \nb w \, dx - \int_\Om w^{\beta -1} \mathbf{F} \cdot \nb w \, dx .
$$
Note that
$$
w^{\beta } \nb w = \frac 1{\beta +1} \nb \left( w^{\beta +1} \right) \quad\mbox{and}\quad w^{(\beta +1)/2} \in W_0^{1,2}(\Om ).
$$
Hence by Lemma \ref{th1-general},   we have
$$
\int_\Om w^{\beta -1} |\nb w|^2 \,dx \le - \int_\Om w^{\beta -1} \mathbf{F} \cdot \nb w \, dx
$$
and so
$$
\int_\Om w^{\beta -1} |\nb w|^2 \,dx \le \int_\Om w^{\beta -1} | \mathbf{F}|^2 \, dx .
$$
By Sobolev's and H\"{o}lder's inequalities,
\begin{align*}
 \left\| w^{(\beta +1)/2} \right\|_{L^{2^*} (\Om )}^2
& \le 4 C \int_\Om \left|\nb \left( w^{(\beta +1)/2} \right) \right|^2 \, dx \\
& = C(\beta+1)^2 \int_\Om w^{\beta -1} |\nb w|^2 \,dx \\
& \le C(\beta+1)^2 \int_\Om w^{\beta -1} | \mathbf{F}|^2 \, dx\\
& \le C (\beta+1)^2 \left\| w^{\beta -1} \right\|_{L^{p/{(p-2)}}(\Om )} \left\|| \mathbf{F}|^2 \right\|_{L^{p/2}(\Om )}
\end{align*}
for some $C=C(n)$.
Setting $\gamma =(\beta +1)/2 \ge 1$, we have
$$
\left\| w \right\|_{L^{ 2^* \gamma } (\Om )}^{\gamma} \le C \gamma \left\| w \right\|_{L^{2(\gamma -1)p /(p-2)}(\Om )}^{ \gamma -1 } \left\| \mathbf{F} \right\|_{L^{p}(\Om )} .
$$
 Note that if $\gamma = p^* /2^* = p(n-2)/2(n-p)$, then
$$
 \gamma \ge 1  \quad\mbox{and}\quad 2^* \gamma = \frac{2(\gamma -1)p}{p-2}= p^*.
$$
Hence taking $\gamma = p^* /2^*$, we obtain
$$
 \left\|u^+ \right\|_{L^{p^*} (\Om )} = \left\| w \right\|_{L^{ p^* } (\Om )} \le C (n,p ) \left\| \mathbf{F} \right\|_{L^{p}(\Om )}.
$$
Similarly, the negative part $u^-$ satisfies the same estimate.

We has shown that for each ${\mathbf F} \in C_c^{\infty}(\Om ;\R^n )$, there exists a unique weak solution $u $ of \eqref{bvp} with $f=\div {\mathbf F}$
satisfying the $L^{p^*}$-estimate (\ref{Lp*-estimate for u}). This enables us to complete the proof of Part (i), by a standard density argument. Part (ii) can be proved by the same argument.
\end{proof}

\medskip
By a cut-off technique based on Lemmas  \ref{th4-boundedness} and  \ref{th4-higher integrability}, we can  show boundedness   of weak solutions of the problems  \eqref{bvp} and \eqref{bvp-dual} on exterior domains.
\begin{lemma}\label{th4-boundedness-exterior}
 Let $\Om $ be any exterior  domain in $\R^n$, $n \ge 3$. Assume that $\mathbf{b} \in L^{n ,\I}(\Om;\R^n)$, $\div \mathbf{b} \ge 0$ in $\Om$, and $n<p< \infty$.
\begin{enumerate}[(i)]
 \item For every ${\mathbf F} \in L^{2}(\Om ;\R^n ) \cap L^{p}(\Om ;\R^n ) $, there exists a unique weak solution $u  $ of \eqref{bvp} with $f=\div {\mathbf F}$. Moreover, we have
$$
  u  \in L^\infty  (\Om )   \quad\mbox{and}\quad \|u\|_{L^{\infty} (\Om )} \le C\left(\|{\mathbf F}\|_{L^{2}(\Om )}+  \|{\mathbf F}\|_{L^{p}(\Om )} \right),
$$
where $C= C  ( n,p , \Om , \|\mathbf{b}\|_{L^{n ,\I}(\Om)}   ) $.
 \item For every ${\mathbf G} \in L^{2}(\Om ;\R^n ) \cap L^{p}(\Om ;\R^n ) $, there exists
 a unique  weak solution $v  $ of \eqref{bvp-dual} with $g=\div {\mathbf G}$. Moreover, we have
$$
 v  \in L^\infty  (\Om )\quad\mbox{and}\quad \|v\|_{L^{\infty} (\Om )} \le C \left(\|{\mathbf G}\|_{L^{2}(\Om )}+  \|{\mathbf G}\|_{L^{p}(\Om )} \right),
$$
where $C= C  ( n,p , \Om , \|\mathbf{b}\|_{L^{n ,\I}(\Om)}   ) $.
\end{enumerate}
\end{lemma}

\begin{proof} Let ${\mathbf F} \in L^{2}(\Om ;\R^n ) \cap L^{p}(\Om ;\R^n ) $ be given. Then by  Lemma \ref{th4-existence},   there exists a unique weak solution $u$ of \eqref{bvp}  with $f=\div {\mathbf F}$, which satisfies
\begin{equation}\label{Linfty-bound for u-0}
\frac{1}{C(n)} \| u\|_{L^{2^*}(\Om )} \le \|\nb u \|_{L^2 (\Om )} \le \|{\mathbf F}\|_{L^2 (\Omega )} .
\end{equation}

We shall derive an $L^\infty$-estimate of $u$, by means of  a cut-off technique.
Choose  $R_0 >0$ so large that $\Omega^c \subset B_{R_0 } =B_{R_0}(0)$
and write $\Om_R = \Om \cap B_{R}$ for $R>R_0$.

We first show that $u \in L^{\infty}(\Om_{R})$ for any $R>R_0$.
Given  $R>R_0$, we fix a  cut-off function $\eta \in C_c^\infty (B_{2R}; [0,1])$   with  $\eta =1$ on $B_{R}$ and define $\overline{u}=\eta u$.  Then  by a direct calculation, we deduce that   $\overline{u}$ satisfies
\[
-\De \overline{u} + \div (\overline{u} \bb) =  \div (\eta  {\mathbf F} -2u\nb \eta  )  +u \De\eta + (u\bb - {\mathbf F})\cdot \nb \eta
\]
in $\Om_{2R}$, that is,
\begin{align*}
\int_{\Om_{2R}} \left( \nb \overline{u} - \overline{u}\mathbf{b} \right) \cdot \nb
\phi \, dx & =  \int_{\Om_{2R}} u \bke{2\nb \eta  \cdot \nb \phi + \phi \De \eta  + \phi \mathbf{b} \cdot \nb \eta  } \, dx\\
&\qquad   -  \int_{\Om_{2R}} \left( \eta  {\mathbf F} \cdot \nabla \phi + \phi {\mathbf F} \cdot \nb \eta  \right)   \, dx
\end{align*}
for all $\phi \in C_c^1 (\Om_{2R})$.
By H\"{o}lder's inequality,
\[
\left| \int_{\Om_{2R}} \left( \eta  {\mathbf F} \cdot \nabla \phi + \phi {\mathbf F} \cdot \nb \eta  \right)   \, dx \right| \le C(  \eta ) \|{\mathbf F}\|_{L^{p}(\Om_{2R} )}\| \phi \|_{W^{1,p'} (\Om_{2R} )}
\]
for all $\phi \in W_0^{1,p'}(\Om_{2R} )$.  Suppose now that $u \in L_{loc}^{q} (\overline{\Om})$ for some $   q  \in [2^*  ,  \infty  )  $.
Then since $1 <    q'  < n$, it follows from Lemma \ref{th1-general} that
$$
 \left| \int_{\Om_{2R}} u \bke{ 2\nb \eta  \cdot \nb \phi +\phi \De \eta  + \phi \mathbf{b} \cdot \nb \eta  } \, dx\right| \le C \|u \|_{L^{q} (\Om_{2R} )}\| \phi \|_{W^{1,q'} (\Om_{2R} )}
$$
for all $\phi \in W_0^{1,q'}(\Om_{2R} )$, where $C= C(n,  q,   \eta , \|\mathbf{b}\|_{L^{n ,\I}(\Om)} )$.
Define $\overline{q}= q^*$ if $q<n$,  $\overline{q}=2n $ if $q  = n$, and  $\overline{q}= \infty $ if $q >n$. Then  by Lemmas \ref{dual-char}, \ref{th4-boundedness}, and  \ref{th4-higher integrability}, there exists a unique $ {w}\in W_0^{1,2}(\Om_{2R} ) \cap L^{\overline{q}}(\Om_{2R})$ such that
\[
-\De w + \div (w \bb)=   \div (\eta  {\mathbf F} -2u\nb \eta  )  +u \De\eta + (u\bb - {\mathbf F})\cdot \nb \eta  \quad\mbox{in}\,\,\Om_{2R}
\]
and
$$
\|w\|_{L^{\overline{q}}(\Om_{2R})} \le C\left(\|{\mathbf F}\|_{L^{p}(\Om_{2R} )}+ \|u \|_{L^{q} (\Om_{2R} )} \right) ,
$$
where $C= C(n, p, q, \Om , R, \eta , \|\mathbf{b}\|_{L^{n ,\I}(\Om)} ) $.
Since $ \overline{u} , {w} \in W_0^{1,2 }(\Om_{2R} )$, it follows from the uniqueness assertion of Lemma \ref{th4-existence}
 that $\overline{u} = {w}$ on $\Omega_{2R}$. This proves that
$$
 u \in L^{\overline{q}}(\Omega_{R}) \quad\mbox{and}\quad \|u\|_{L^{\overline{q}}(\Om_{R})} \le C \left(\|{\mathbf F}\|_{L^{p}(\Om_{2R} )}+ \|u \|_{L^{q} (\Om_{2R} )} \right),
$$
where $C= C(n, p,  q, \Om , R, \eta , \|\mathbf{b}\|_{L^{n ,\I}(\Om)} )$.
Since $R>R_0$ can be arbitrarily large, it follows that   $u \in    L_{\loc}^{\overline{q}}(\overline{\Omega})$. This argument can be repeated finitely many times to show that $u$ is locally bounded on $\overline{\Om}$; that is, by a bootstrap argument starting from  $q =2^*$, we can deduce that $u \in L_{\loc}^{\infty}(\overline{\Omega}) $ and
\begin{equation}\label{Linfty-bound for u-1}
   \|u\|_{L^{\infty}(\Om_{R})} \le C\left(\|{\mathbf F}\|_{L^{p}(\Om_{2R} )}+ \|u \|_{L^{2^*} (\Om_{2R} )} \right)
\end{equation}
for any $R>R_0$, where $C=  C(n, p, \Om , R , \|\mathbf{b}\|_{L^{n ,\I}(\Om)} ) $.

We next  show that $u \in L^\infty  (\Om )$.
Let $x_0 \in \Om$ be any point with $|x_0|>R_0 + 2$. Choosing a  cut-off function $\zeta \in C_c^\infty (B_{2}(x_0 ); [0,1])$   with  $\zeta =1$ on $B_{1}(x_0 )$, we define ${\overline{u}} = u \zeta  $.
Then ${\overline{u}}    \in W_0^{1,2}(B_2 (x_0 )) $   satisfies
\[
-\De \overline{u} + \div (\overline{u} \bb) =  \div (\zeta  {\mathbf F} -2u\nb \zeta  )  +u \De\zeta + (u\bb - {\mathbf F})\cdot \nb \zeta
\]
in $B_{2}(x_0 )$.  Hence by the same bootstrap argument as above,  we can deduce that
\begin{equation}\label{Linfty-bound for u-2}
   \|u\|_{L^{\infty}(B_{1}(x_0 ) )}   \le C  \left(\|{\mathbf F}\|_{L^{p}(B_{2}(x_0 ) )}+ \|u \|_{L^{2^*} (B_{2}(x_0 ) )} \right),
\end{equation}
where $C= C(n, p,  \|\mathbf{b}\|_{L^{n ,\I}(\Om)})$ is independent of  any $x_0 \in \Om$ with $|x_0|>R_0 + 2$.
Combining (\ref{Linfty-bound for u-0}), (\ref{Linfty-bound for u-1}), and (\ref{Linfty-bound for u-2}), we complete the proof of the first part of the lemma. The second part can be proved by exactly the same argument.
\end{proof}

\subsection{Existence and uniqueness for more singular $\mathbf b$}

While the $L^\infty$-estimates in Lemma \ref{th4-boundedness-exterior} involve the $L^{n,\infty}$-norm of $\mathbf{b}$, those of  Lemma \ref{th4-boundedness} do not depend on $\mathbf{b}$ at all.
This enables us to prove existence of bounded weak solutions under a weaker condition on $\mathbf{b}$, if the domain $\Om$ is bounded.

\begin{proposition}\label{th4-boundedness-L1 drift} Let $\Om $ be any bounded domain in $\R^n$, $n \ge 3$.
Assume that $\mathbf{b} \in L_{\loc}^{1}(\Om;\R^n)$, $\div \mathbf{b} \ge 0$ in $\Om$, and $n < p< \infty$.
\begin{enumerate}[(i)]
 \item For each ${\mathbf F} \in L^{p}(\Om ;\R^n ) $, there exists at least one weak solution $u \in W_0^{1,2}(\Om )\cap L^\infty (\Om )$ of \eqref{bvp}  with $f=\div {\mathbf F}$ satisfying the estimate
$$
\|\nabla u \|_{L^2 (\Om )} + \|u\|_{L^\infty (\Om )} \le C(n,p, {\rm diam}\, \Om ) \|{\mathbf F}\|_{L^p (\Omega )}.
$$
 \item Assume in addition  that $\mathbf{b} \in L_{\loc}^{2}(\Om;\R^n)$. Then for each ${\mathbf G} \in L^{p}(\Om ;\R^n ) $, there exists at least one weak solution $v\in W_0^{1,2}(\Om )\cap L^\infty (\Om )$ of \eqref{bvp-dual} with $g=\div {\mathbf G}$ satisfying the estimate%
$$%
\|\nabla v \|_{L^2 (\Om )} + \|v\|_{L^\infty (\Om )} \le C(n,p, {\rm diam}\, \Om )\|{\mathbf G}\|_{L^p (\Omega )}.
$$%
\end{enumerate}
\end{proposition}

\begin{remark}
Uniqueness of such weak solutions is not asserted in the proposition, due to the weak regularity of $\mathbf{b}$.
\end{remark}
\begin{remark}
A similar existence result was obtained by Kontovourkis \cite[Lemmas 2.2.9, 2.2.10]{Ko} when $\div \mathbf{b} = 0$ in $\Om$.
\end{remark}

\begin{proof}  Let $\{\Om_k \}$ be a sequence of bounded $C^1$-subdomains of $\Om$ such that $\overline{\Om}_k \subset \Om_{k+1}$ for each $k$ and $\Om =\cup_{k=1}^\infty \Om_k$ (see \cite[Proposition 8.2.1]{Da} e.g.). By mollification, we choose a sequence $\{{\mathbf b}_k \}$ in $L^n(\Om;\R^n )$ such that $\div \mathbf{b}_k \ge 0$ in $\Om_k$ and ${\mathbf b}_k \to {\mathbf b}$ in $L_{\loc}^{1}(\Om;\R^n)$. Then by Lemmas \ref{th4-existence} and \ref{th4-boundedness}, there exists
$u_k \in W_0^{1,2}(\Om_k )\cap L^{\infty}(\Om_k )$ such that
$$
\int_{\Om_k} \bke{\nb u_k - u_k \mathbf{b}_k} \cdot \nb
\phi \, dx = -\int_{\Om_k} {\mathbf F} \cdot \nb \phi \, dx \quad \text{for all } \phi \in C^1_{c}(\Om_k )
$$
and
$$
\|\nb u_k \|_{L^2 (\Om_k )} + \|u_k \|_{L^{\infty} (\Om_k )}\le C(n,p , {\rm diam}\, \Om ) \|{\mathbf F}\|_{L^{p}(\Om )}.
$$
Extend $u_k$ to $\Om$ by defining $u_k=0$ outside $\Om_k$. Then $\{u_k \}$ is a bounded sequence in $W_0^{1,2}(\Om )$ as well as $L^{\infty}(\Om )$. Hence by standard compactness results in Sobolev and Lebesgue spaces, we may assume that $u_k \to u$ weakly in $W_0^{1,2}(\Om )$ and weakly-$*$ in $L^{\infty}(\Om )$.
It is easily checked that the limit $u$ is indeed a weak solution of \eqref{bvp} satisfying the desired estimates. This completes the first part of the proposition. The second part can be proved similarly.
\end{proof}

\medskip
The estimates in Lemmas \ref{th4-existence} and \ref{th4-higher integrability} do not depend on the domain $\Om$ nor on the drift $\mathbf{b}$.
Hence, adapting the proof of Proposition \ref{th4-boundedness-L1 drift}, we easily  obtain the following existence result for any  domain in $\R^n$, whose proof is omitted.

\begin{proposition}\label{th4-existence-L1} Let $\Om $ be any domain in $\R^n$, $n \ge 3$. Assume that $\mathbf{b} \in L_{\loc}^{1}(\Om;\R^n)$, $\div \mathbf{b} \ge 0$ in $\Om$, and $2 \le p<n$.
\begin{enumerate}[(i)]
 \item
Assume in addition that $\mathbf{b} \in L_{\loc}^{np'/(n+p')}(\Om;\R^n)$. Then for every ${\mathbf F} \in L^{2}(\Om ;\R^n ) \cap L^{p}(\Om ;\R^n ) $, there exists at least one weak solution $u \in \hat{D}_0^{1,2}(\Om )\cap L^{p^*}(\Om )$ of \eqref{bvp} with $f=\div {\mathbf F}$ satisfying the estimates
$$
\|\nb u \|_{L^2 (\Om )} \le \|{\mathbf F} \|_{L^{2}(\Om )}\quad\mbox{and}\quad \|u\|_{L^{p^*} (\Om )} \le C(n,p ) \|{\mathbf F}\|_{L^{p}(\Om )}.
$$
 \item Assume in addition that $\mathbf{b} \in L_{\loc}^{2}(\Om;\R^n)$.
 Then for every ${\mathbf G} \in L^{2}(\Om ;\R^n ) \cap L^{p}(\Om ;\R^n ) $, there exists at least one weak solution $v \in \hat{D}_0^{1,2}(\Om )\cap L^{p^*}(\Om )$ of \eqref{bvp-dual} with $g=\div {\mathbf G}$ satisfying the estimates
$$
\|\nb v \|_{L^2 (\Om )} \le \|{\mathbf G} \|_{L^{2}(\Om )}\quad\mbox{and}\quad \|v\|_{L^{p^*} (\Om )} \le C(n,p ) \|{\mathbf G}\|_{L^{p}(\Om )}.
$$
\end{enumerate}
\end{proposition}

Note that Proposition \ref{th4-existence-L1} is valid for any domain because the constants in the $L^{p^*}$-estimates of  Lemma
 \ref{th4-higher integrability} do not depend on $\Om$. In contrast, for Proposition \ref{th4-boundedness-L1 drift}, $\Om$ should be bounded because the constants in the $L^\infty$-estimates of Lemmas
\ref{th4-boundedness} and \ref{th4-boundedness-exterior} depend on $\Om$.

\medskip

As an application of Propositions \ref{th4-boundedness-L1 drift} and \ref{th4-existence-L1}, we obtain the following result, which has been already proved by Zhikov \cite{Zhi} and Kontovourkis \cite{Ko} for the case when $\div \mathbf{b} =0$ (see also \cite[Theorem 1.3]{KK}).

\begin{proposition}\label{th4-boundedness-L2 drift} Let $\Om $ be any bounded domain in $\R^n$, $n \ge 3$. Assume that $\mathbf{b} \in L^{2}(\Om;\R^n)$ and $\div \mathbf{b} \ge 0$ in $\Om$.

\begin{enumerate}[(i)]
 \item For each $g \in W^{-1,2}(\Om )$, there exists a unique weak solution $v $ of \eqref{bvp-dual}.
 \item Assume in addition that ${\rm div}\, \mathbf{b} \in L^{2n /(n+2)}(\Om)$.
 Then for each $f \in W^{-1,2}(\Om )$, there exists a unique weak solution $u$ of \eqref{bvp}.
\end{enumerate}

\end{proposition}

\begin{proof} By Lemma \ref{dual-char} and Proposition \ref{th4-existence-L1}, there exists at least one weak solution of each of the problems  \eqref{bvp} and \eqref{bvp-dual}. It thus remains to prove uniqueness of weak solutions.

Let $v \in W_0^{1,2}(\Om ) $ be a weak solution of \eqref{bvp-dual} with the trivial data $g=0$; hence
\begin{equation}\label{eq1.3-00}
\int_\Om \nb v \cdot \bke{\nb \psi - \psi \mathbf{b}} \, dx = 0 \quad \text{for all } \psi \in C^\infty_{c}(\Om).
\end{equation}
Let $f \in C_c^\infty (\Om )$ be given. Then by
Proposition \ref{th4-boundedness-L1 drift}, there exists $u \in W_0^{1,2}(\Om )\cap L^\infty (\Om )$ such that
\begin{equation}\label{eq1.3-000}
\int_\Om \bke{\nb u - u \mathbf{b}} \cdot \nb
\phi \, dx = \int_\Om f \phi \, dx \quad\mbox{for all}\,\, \phi \in C^\infty_{c}(\Om).
\end{equation}
Moreover, by the proof of Proposition \ref{th4-boundedness-L1 drift}, we may assume that  there is a sequence $\{ u_k \}$ in $C_c^\infty (\Om ) $ that converges to $u$ weakly in $W_0^{1,2}(\Om)$ and weakly-$*$ in $L^\infty (\Om )$.
Let $\{v_k \}$ be any sequence in $C_c^\infty (\Om)$ converging to $v$ in $W_0^{1,2}(\Om)$. Then by (\ref{eq1.3-000}) and (\ref{eq1.3-00}), we have
\begin{align*}
 \int_\Om fv \, dx & =\lim_{k\to\infty} \int_\Om f v_k \, dx \\
 &= \lim_{k\to\infty} \int_\Om \bke{\nb u - u \mathbf{b}} \cdot \nb
v_k \, dx \\
 & = \int_\Om \bke{\nb u - u \mathbf{b}} \cdot \nb
v \, dx \\
& =\lim_{k\to\infty} \int_\Om \bke{\nb u_k - u_k \mathbf{b}} \cdot \nb
v \, dx =0.
\end{align*}
Since $f\in C_c^\infty (\Om )$ is arbitrary, we conclude that $v=0$ in $\Om$. This completes the  proof of Part (i).

The above argument does not work for uniqueness of weak solutions   of the problem \eqref{bvp}, because the last calculation can not be justified even though the dual problem \eqref{bvp-dual} has   weak solutions $v$  in  $W_0^{1,2}(\Om ) \cap L^\infty (\Om )$.  %
To prove Part (ii), we assume in addition that ${\rm div}\, \mathbf{b} \in L^{2n /(n+2)}(\Om)$, so that
\[
 - \int_\Om \bke{   \psi \mathbf{b}} \cdot \nb \phi + ( \phi \mathbf{b} ) \cdot \nb \psi
  \, dx
= \int_\Om    \psi  \phi  \div \bb  \, dx
\]
for all $\phi    \in C_c^\infty (\Om )$ and $\psi \in W_0^{1,2}(\Om )$.  Let   $u \in W_0^{1,2}(\Om ) $ be a weak solution of \eqref{bvp} with the trivial data $f=0$.  Given $g\in C_c^\infty (\Om )$, let $v \in W_0^{1,2}(\Om )\cap L^\infty (\Om )$ be a weak solution of  \eqref{bvp-dual}. Then choosing sequences $\{ u_k \}$ and $\{v_k \}$ in $C_c^\infty (\Om ) $ such that $u_k \to u$ in $W_0^{1,2}(\Om)$,  $v_k \to v$ weakly in $W_0^{1,2}(\Om)$, and  $v_k \to v$ weakly-$*$ in $L^\infty (\Om )$, we have
\begin{align*}
 \int_\Om g u \, dx & =\lim_{k\to\infty} \int_\Om g u_k \, dx \\
&= \lim_{k\to\infty} \int_\Om \nb v  \cdot \bke{\nb u_k  - u_k  \mathbf{b}}  \, dx \\
&= \lim_{k\to\infty} \int_\Om \nb v  \cdot  \nb u_k  +  ( v \mathbf{b} )  \cdot \nabla u_k +   u_k v \div  \bb \,  dx \\
& =\int_\Om \nb v  \cdot  \nb u   +  ( v \mathbf{b} )  \cdot \nabla u  +   u  v \div  \bb \,  dx \\
& =\lim_{k\to\infty} \int_\Om \nb v_k  \cdot  \nb u   +  ( v_k \mathbf{b} )  \cdot \nabla u  +   u  v_k \div  \bb \,  dx \\
 &= \lim_{k\to\infty} \int_\Om \bke{\nb u - u \mathbf{b}} \cdot \nb
v_k \, dx  =0.
\end{align*}
This completes the proof of Part (ii).
\end{proof}

\section{Global H\"{o}lder regularity of weak solutions}
\label{Sec5}

Assume that $\mathbf{b} \in L^{n,\I}(\Om;\R^n)$ and  $\div \mathbf{b} \ge 0$ in $\Om$. Then
by the basic inequality (\ref{weak-basic-ineq}) for weak spaces, we deduce that if $1 \le q<n$, then
\[
\sup_{\rho>0} \,  \rho^{1-n/q} \left(\int_{B_\rho (x_0 ) \cap \Om} |\bb |^q \, dx \right)^{1/q} \le  C(q,n) \|\bb\|_{L^{n,\infty}(\Om )},
\]
that is,   the drift $\bb$ belongs to the Morrey space $M_{q}^{n/q-1}(\Om ;\R^n)$. Hence it follows from interior regularity results due to  Nazarov and Uraltseva \cite{NU} that weak solutions   of \eqref{bvp-dual} are locally  H\"{o}lder  continuous on $\Om$.
 However H\"{o}lder regularity of solutions of  \eqref{bvp} is unclear; See Remark \ref{holder regularity for u} below.

In this  section, we shall show that   weak solutions  of \eqref{bvp-dual} are even   H\"{o}lder  continuous  up to the boundary $\partial\Om$; thus   they are \emph{globally} H\"{o}lder continuous on $\Om$.

\subsection{Weak Harnack inequalities}

We first prove the weak Harnack inequalities for weak solutions of \eqref{bvp-dual}, from which   H\"{o}lder regularity will be deduced.

Throughout this subsection, let $\Om $ be a bounded Lipschitz domain in $\R^n$, $n \ge 3$, and let $x_0 $ be a fixed point in $\overline{\Omega}$. For $0<R<\infty$, we then write
$$
B_R =B_R (x_0 )  \quad\mbox{and}\quad \Omega_R =\Omega \cap B_R .
$$

\begin{lemma}[Weak Harnack inequality]
\label{Harnack inequality}
Assume that $\mathbf{b} \in L^{n,\I}(\Om;\R^n)$, $\div \mathbf{b} \ge 0$ in $\Om$, and ${\mathbf G} \in L^{p}(\Om ; \R^n)$ for some $n<p<\infty$.
Let $v  $ be a  nonnegative function in $W^{1,2}(\Om_{4R} ) \cap L^\infty (\Om_{4R})$ satisfying
$$
-\Delta v - {\mathbf b} \cdot \nabla v \ge  \div {\mathbf G} \quad\mbox{(weakly) in}\,\,\Omega_{4R}
$$
in the sense that
\begin{equation}\label{weak-form-0200-dual}
\int_{\Om_{4R}} \nabla v \cdot \left( \nabla \psi -\psi \bb \right) \, dx  \ge  -\int_{\Om_{4R}}{\mathbf G} \cdot \nabla \psi \, dx
\end{equation}
for all nonnegative $\psi \in W_0^{1,2} (\Om_{4R})$.
Then there are  numbers $0<  p_0 =p_0 (n, \|{\mathbf b}\|_{L^{n,\infty}(\Omega )} )<1$ and $C =C(n,p, \|{\mathbf b}\|_{L^{n,\infty}(\Omega )} )>1$ such that
there hold the following estimates: %
\begin{enumerate}[(i)]
 \item If $B_{4R}(x_0 ) \subset \Omega$, then
$$
\left( R^{- {n}} \int_{B_{2R}} |v |^{p_0} \, dx \right)^{1/p_0} \le C \left(\inf_{B_R } v + R^{ 1-n/p} \|{\mathbf G}\|_{ L^p (\Omega_{4R})  } \right).
$$
 \item If $x_0 \in \partial\Omega$ and $v \ge m$ on $ B_{4R}(x_0 ) \cap \partial\Omega$ for some constant $m \ge 0$, then
$$
\left( R^{- {n} } \int_{B_{2R}} |v_m |^{p_0} \, dx \right)^{1/p_0} \le C \left(\inf_{B_R } v_m + R^{ 1-n/p} \|{\mathbf G}\|_{L^p (\Omega_{4R})} \right),
$$
where
$$
v_m = \left\{ \begin{array}{rl}
 \inf \,\{v , m \} \quad
 \mbox{in}\,\, \Omega \qquad \, \,\,\\[4pt]
 m \qquad \mbox{on}\,\, {{\mathbb R}^n} \setminus \Omega .
\end{array}\right.
$$
\end{enumerate}
\end{lemma}

\begin{proof}%
Set $\delta =1 -n/p$ and choose any $K>0$ with $R^{\delta} \|{\mathbf G}\|_{L^p (\Omega_{4R})} \le K$.
Let us define
$$
 w =\left\{ \begin{array}{rl}
 v +K \quad\,\,\,\,
 \mbox{if}\,\, x_0 \in \Omega \, \,\,\,\, \\[4pt]
 v_m +K \quad \mbox{if}\,\, x_0 \in \partial \Omega\,\,
 \end{array}\right.
 \quad\mbox{and}\quad
 H(w) = \left\{ \begin{array}{rl}
 w^\beta \qquad \,\,\,\,
 \mbox{if}\,\, x_0 \in \Omega \, \,\,\,\, \\[4pt]
 w^\beta -\overline{K}^\beta \quad \mbox{if}\,\, x_0 \in \partial \Omega ,
 \end{array}\right.
$$
where $\overline{K}=m+K$ and $\beta$ is a negative real number.
Next, we fix a cut-off function $ \eta \in C_c^\infty (B_{4R} ;[0,1] )$. Then since $0< K \le w \le \|v\|_{L^\infty (\Om )}+ K$ on $\Omega$ and $H(w)=0$ %
on $\partial\Omega$ if $x_0 \in \partial\Omega$, it follows that $0 \le  \eta^2 H(w) \in W_0^{1,2}(\Omega )$. Moreover, since $w=v-K$ if either $x_0 \in \Omega$ or $v \le m$, and $H(w)=H(\overline{K} )=0$ if $x_0 \in \partial \Omega$ and $v>m$, it follows that $H(w)\nb v = H(w)\nb w$. Hence taking $\psi = \eta^2 H(w) $ in (\ref{weak-form-0200-dual}), we have
$$
\int_{\Omega} \nabla v \cdot \nabla \left( \eta^2 H(w) \right)\, dx \ge  \int_{\Omega} \left( \eta^2 H (w) {\mathbf b}\right) \cdot \nabla w \, dx - \int_\Omega {\mathbf G} \cdot \nabla \left( \eta^2 H(w) \right) \, dx .
$$
Noting then that
$$
\nabla \left( \eta^2 H(w) \right) = \beta \eta^2 w^{\beta-1}\nabla w + 2 \eta H(w) \nabla \eta ,
$$
$$
\nabla v \cdot \nabla \left( \eta^2 H(w) \right) = \beta \eta^2 w^{\beta-1}|\nabla w |^2 + 2 \eta H(w) \nabla w \cdot \nabla \eta ,
$$
$$
\left| 2 \eta H(w) \nabla w \cdot \nabla \eta \right| \le 2 \eta w^\beta |\nabla w| | \nabla \eta |\le \frac{|\beta|}4 \eta^2 w^{\beta-1} |\nabla w|^2 + \frac{4}{|\beta|} w^{\beta +1} |\nabla \eta|^2 ,
$$
and
$$
\left| {\mathbf G} \cdot \nabla \left( \eta^2 H(w) \right)\right| \le \frac{|\beta|}4 \eta^2 w^{\beta-1} |\nabla w|^2 + \frac{2}{|\beta|} w^{\beta +1} |\nabla \eta|^2 + 2|\beta| \eta^2 w^{\beta -1} |{\mathbf G}|^2,
$$
we obtain
$$
\begin{aligned}
\frac{|\beta|}2 \int_{\Omega} \eta^2 w^{\beta-1} |\nabla w|^2 \, dx
&\le - \int_{\Omega} \left( \eta^2 H (w) {\mathbf b}\right) \cdot \nabla w \, dx \\
&\quad + \frac{6}{|\beta|} \int_{\Omega} w^{\beta +1} |\nabla \eta|^2 \, dx + 2|\beta| \int_\Om \eta^2 w^{\beta -1} |{\mathbf G}|^2 \, dx.
\end{aligned}
$$
Since $0< K$ and $R^{\delta} \|{\mathbf G}\|_{L^p (\Omega )} \le K \le w$,  we also  have
\begin{equation}\label{est-first}
\begin{aligned}
\frac{|\beta|}2 \int_{\Omega} \eta^2 w^{\beta-1} |\nabla w|^2 \, dx &\le \frac{6}{|\beta|} \int_{\Omega} w^{\beta +1} |\nabla \eta|^2 \, dx\\
& \quad + 2 R^{-2 \delta}|\beta| \left\| \ \eta^2 w^{\beta +1} \right\|_{L^{p/(p-2)}(\Om )} \\
 &\quad - \int_{\Omega} \left( \eta^2 H (w) {\mathbf b}\right) \cdot \nabla w \, dx .
\end{aligned}
\end{equation}

Suppose that $\beta < -1$. Then defining
$$
F(t) = \left\{ \begin{array}{rl}
 \frac 1{\beta+1} t^{\beta+1} \qquad \qquad \qquad \,\, \,\,\,\,
 \mbox{if}\,\, x_0 \in \Omega \, \, \,\,\, \\[4pt]
 \frac 1{\beta+1} \left(t^{\beta+1} -\overline{K}^{\beta+1} \right) - \overline{K}^\beta \left(t-\overline{K} \right) \quad\mbox{if}\,\, x_0 \in \partial \Omega ,
 \end{array}\right.
$$
 we can write
$$
 \eta^2 H (w) \nabla w
 = \nabla \left( \eta^2 F(w) \right) - 2 \eta F(w) \nabla \eta .
$$
Obviously, if $x_0 \in \Omega$, then
\begin{equation}\label{Fw-est}
0 \le -F(w) \le \frac 1{|\beta+1| } w^{\beta+1} \quad\mbox{and}\quad \eta |F(w)|^{1/2} \in W_0^{1,2}(\Omega ).
\end{equation}
This property also holds when $x_0 \in \partial\Omega$. Indeed, since $\beta < -1$ and $F'(t) = t^\beta -{\overline K}^\beta > 0$ for $0< t < \overline{K}$, it follows that $ \frac 1{\beta+1} t^{\beta+1} < F(t) < F(\overline{K})=0$ for $0< t < \overline{K}$. Moreover, $(-F )^{1/2}$ is a $C^1$-function on $(0, \overline{K}]$ because
\begin{align*}
\lim_{t \to ({\overline K})^{-}} \frac{[-F(t)]^{1/2}}{t-{\overline K}} &= \lim_{t \to ({\overline K})^{-}} \left\{ [-F(t)]^{1/2} \right\} ' = - \left( \lim_{t \to ({\overline K})^{-}} \frac{[F'(t)]^2}{-4 F(t)} \right)^{1/2}\\
& = - \left(- \lim_{t \to ({\overline K})^{-}} \frac 12 F''(t) \right)^{1/2} = - \left(\frac {|\beta|}2 \overline{K}^{\beta -1}\right)^{1/2} .
\end{align*}
Hence, recalling again that ${\rm div}\, {\mathbf b} \ge 0$ in $\Omega$ and $\nabla w =0$ on ${{\mathbb R}^n} \setminus \Omega$ if $x_0 \in \partial\Omega$, we deduce from (\ref{est-first}), \eqref{Fw-est}, and Lemma \ref{th1-general} that
$$
\begin{aligned}
& \frac{|\beta|}2 \int_{B_{4R}} \eta^2 w^{\beta-1} |\nabla w|^2 \, dx \\
&\qquad\le \frac{6}{|\beta|} \int_{B_{4R}} w^{\beta +1} |\nabla \eta|^2 \, dx + 2 R^{-2 \delta}|\beta| \left\| \ \eta^2 w^{\beta +1} \right\|_{L^{p/(p-2)}(B_{4R} )} \\
&\qquad \qquad + \frac 2{|\beta +1|} \int_{B_{4R}} \eta w^{\beta+1} |\nabla \eta | |{\mathbf b}|\, dx ,
\end{aligned}
$$
where ${\mathbf b}$ is extended to $\R^n$ by defining zero outside $\Om$.
Let us now define $\gamma =\beta +1$ and $\overline{w} =w^{ \gamma/ 2} $.
Then since
$$
 |\nabla \overline{w} |^2= \frac {1}4 \gamma^2 w^{\beta-1} | \nabla w|^2
\quad\mbox{and}\quad
|\nabla (\eta \overline{w})|^2 \le 2 \eta^2 |\nabla \overline{w}|^2 + 2 \overline{w}^2 |\nabla \eta |^2 ,
$$
we obtain
$$
\begin{aligned}
& \frac{ |\beta|}{\gamma^2} \int_{B_{4R}} |\nabla (\eta \overline{w})|^2 \, dx \\
& \quad \le \left( \frac{6}{|\beta|} +\frac{2 |\beta|}{\gamma^2} \right) \int_{B_{4R}} \overline{w}^2 |\nabla \eta|^2 \, dx + 2 R^{-2 \delta}|\beta| \left\| \eta^2 \overline{w}^2 \right\|_{L^{p/(p-2)}(B_{4R} )} \\
&\qquad + \frac 2{|\gamma|} \int_{B_{4R}} \eta \overline{w}^2 |\nabla \eta | |{\mathbf b}|\, dx .
\end{aligned}
$$
Noting that $|\gamma|< |\beta|$, we have
$$
\begin{aligned}
\|\nb (\eta \overline{w} ) \|_{L^{2}(B_{4R} )}^2 & \le 8 \| \overline{w} \nabla \eta\|_{L^2 (B_{4R})}^2 +2 R^{-2 \delta}|\gamma|^2 \left\| \eta \overline{w} \right\|_{L^{2p/(p-2)}(B_{4R} )}^2
 \\
& \qquad + 2 \int_{B_{4R}} \eta \overline{w}^2 |\nabla \eta | |{\mathbf b}|\, dx .
\end{aligned}
$$
But by H\"{o}lder's inequality, Sobolev's inequality, and Lemma \ref{th1-general},
\begin{align*}
& 2 R^{-2 \delta}|\gamma|^2 \left\| \eta \overline{w} \right\|_{L^{2p/(p-2)}(B_{4R} )}^2 \\
&\qquad\le 2 R^{-2 (1-n/p)}|\gamma|^2 \left\| \eta \overline{w} \right\|_{L^{2}(B_{4R} )}^{2 (1-n/p)} \left\| \eta \overline{w} \right\|_{L^{2^*}(B_{4R} )}^{2n/p}\\
 & \qquad \le \frac 14 \left\| \nb ( \eta \overline{w} ) \right\|_{L^{2}(B_{4R} )}^2 + C R^{-2 }|\gamma|^{2p/(p-n)} \left\| \eta \overline{w} \right\|_{L^{2}(B_{4R} )}^2
\end{align*}
and
$$
\begin{aligned}
 2 \int_{B_{4R}} \eta \overline{w}^2 |\nabla \eta | |{\mathbf b}|\, dx
 & \le C \|{\mathbf b}\|_{L^{n,\infty}(\Omega )} \|\overline{w} \nabla \eta \|_{L^2 (\Om )} \|\nabla (\eta \overline{w}) \|_{L^2 (\Om )} \\
& \le \frac{1}{4}\|\nabla (\eta \overline{w}) \|_{L^2 (B_{4R} )}^2 + C \|{\mathbf b}\|_{L^{n,\infty}(\Omega )}^2 \|\overline{w} \nabla \eta \|_{L^2 (B_{4R} )}^2 .
\end{aligned}
$$
Hence using Sobolev's inequality again, we obtain
\begin{equation}\label{ineq for w}
\begin{aligned}
\| \eta \overline{w} \|_{L^{2^*}(B_{4R} )}^2 & \le C \| \overline{w} \nabla \eta\|_{L^2 (B_{4R})}^2 + C R^{-2 }|\gamma|^{2p/(p-n)} \left\| \eta \overline{w} \right\|_{L^{2}(B_{4R} )}^2
 \\
& \qquad + C \|{\mathbf b}\|_{L^{n,\infty}(\Omega )}^2 \|\overline{w} \nabla \eta \|_{L^2 (B_{4R} )}^2.
\end{aligned}
\end{equation}

Suppose now that $0<r_1 <r_2 \le 2R$, $\eta =1$ in $B_{r_1}$ and $\eta =0$ in $B_{r_2}^c$.
Then since ${\rm supp}\, \nb \eta \subset B_{r_2}\setminus B_{r_1}$ and $|\nabla \eta| \le C/(r_2 -r_1 )$, it follows from (\ref{ineq for w}) that
$$%
\| \overline{w}\|_{L^{2^*} (B_{r_1} ) } \le C \left( \frac{1+ \|{\mathbf b}\|_{L^{n,\infty}(\Omega )}}{r_2 -r_1 } + \frac{|\gamma|^{p/(p-n)}}{R} \right)
 \| \overline{w} \|_{L^2 (B_{r_2} ) } .
$$%
Recalling that $\overline{w} =w^{\gamma /2}$ and $\gamma =\beta +1 < 0$, we finally deduce that
\begin{equation}\label{local ineq for -}
 \left\|w \right\|_{L^{-|\gamma|} (B_{r_2} ) } \le \left[ C \left( \frac{1+ \|{\mathbf b}\|_{L^{n,\infty}(\Omega )}}{r_2 -r_1 } + \frac{|\gamma|^{\sigma }}R \right) \right]^{\frac 2{|\gamma|}} \| w\|_{ L^{ -|\gamma| \chi }(B_{r_1} ) }
\end{equation}
for every $\gamma <0 $, where $\chi = n/(n-2)$ and $\sigma = p/(p-n)$.

Given any number $0< p_0 <1 $,
let us define
$$
p_k = {p}_0 \chi^{k} \quad \mbox{and}\quad r_{k}=R+2^{-k} R \quad\mbox{for all}\,\, k \in \N \cup \{0\}.
$$
Then taking $\gamma=- p_k $, $r_1 =r_{k+1} $ and $r_2 =r_k$ in (\ref{local ineq for -}), we obtain
$$
\left\|w \right\|_{L^{-p_k}(B_{r_k} ) } \le \left[ \left( \frac {C}R \right) \left(2 \chi^{\sigma }\right)^{ k } \right]^{\frac 2{p_0} \chi^{-k} }\|w\|_{L^{-p_{k+1}}( B_{r_{k+1}} ) }
$$
for each $k \ge 0$, where $C=C( n,p,p_0 , \|{\mathbf b}\|_{L^{n,\infty}(\Omega )} )$. By iteration, we thus obtain
$$
\left\|w \right\|_{L^{-p_0} (B_{2R} ) }\le \prod_{j=0}^k \left[ \left( \frac {C}R \right) \left(2 \chi^{\sigma }\right)^{ j }\right]^{\frac 2{p_0} \chi^{-j} }\|w\|_{L^{-p_{k+1}}(B_{r_{k+1}} ) }
$$
for each $k \ge 0$. Noting that
$$
\begin{aligned}
 \prod_{j=0}^k \left[ \left( \frac {C}R \right) \left(2 \chi^{\sigma }\right)^{ j } \right]^{\frac 2{p_0} \chi^{-j } }
& \quad \le \left( \frac {C}R \right)^{\sum_{j=0}^\infty \frac 2{p_0} \chi^{-j}} \left( 2 \chi^{\sigma } \right)^{\sum_{j=0}^\infty \frac{2}{p_0} j \chi^{-j} } \\
&\quad \le C R^{- n/p_0} ,
\end{aligned}
$$
we deduce that
$$
\left\|w \right\|_{L^{-p_0 }(B_{2R} ) }\le C R^{- n/p_0}\|v\|_{L^{-p_{k+1}}( B_{r_{k+1}} )}
$$
for all $k \ge 1$. Hence letting $k \to \infty$, we have
\begin{equation}\label{inf-bound for v}
R^{ n/p_0} \left\|w \right\|_{L^{- p_0}( B_{2R} ) }\le C \left( \inf_{\Omega_R} w \right)
\end{equation}
for some $C=C(n,p, p_0 , \|{\mathbf b}\|_{L^{n,\infty}(\Om ) })$.

\medskip

Next, taking $\beta =-1$ in (\ref{est-first}), we obtain
$$
\begin{aligned}
\frac{1}2 \int_{B_{4R}} \eta^2 w^{-2} |\nabla w|^2 \, dx &\le 6 \int_{B_{4R}} |\nabla \eta|^2 \, dx+ 2 R^{-2\delta} \left\|\eta^2 \right\|_{L^{p/(p-2)}(B_{4R} )}\\
&\quad - \int_{\Omega} \left( \eta^2 H (w) {\mathbf b}\right) \cdot \nabla w \, dx .
\end{aligned}
$$
 Let us define $\overline{w} =\ln w $, so that $\nabla \overline{w} =w^{-1} \nabla w $.
Noting then that
$$
\begin{aligned}
- \int_{\Omega} \left( \eta^2 H (w) {\mathbf b}\right) \cdot \nabla w \, dx
& \le \int_{\Omega_{4R}} \eta^2 w^{-1} |{\mathbf b}| |\nabla w| \, dx \\
& \le C \|{\mathbf b}\|_{L^{n,\infty}(\Omega )} \| \nabla \eta \|_{L^2 (B_{4R} )} \|\eta \nabla \overline{w} \|_{L^2 (B_{4R} )} ,
\end{aligned}
$$
 we deduce that
\begin{equation*}%
 \| \eta \nabla \overline{w} \|_{L^{2}(B_{4R} )} \le C \left( \| \nabla \eta\|_{L^2 (B_{4R} )} + R^{-\delta} \left\|\eta \right\|_{L^{\frac {2p}{p-2}}(B_{4R} )} \right)
\end{equation*}
for some $C=C(n, \|{\mathbf b}\|_{L^{n,\infty}(\Om ) })$.
Choosing a cut-off function $\eta \in C_c^\infty (B_{4R};[0,1])$ such that $\eta =1$ on $B_{2R}$, we have
\begin{equation}\label{L2-bound of w}
\|\nabla \overline{w} \|_{L^2 (B_{2R} )} \le C \left[ \frac{R^{n/2}}R+ R^{ \frac np -1} \left(R^n \right)^{\frac{p-2}{2p}} \right] = 2 C R^{\frac n2 -1}
\end{equation}
 for some $C=C(n, \|{\mathbf b}\|_{L^{n,\infty}(\Om ) })$ independent of $p$.

Let $B_\rho (y)$ be any ball with $B_\rho (y) \cap B_{2R} \neq \emptyset$.
If $\rho \ge 2R/3$, then by (\ref{L2-bound of w}),
$$
\int_{B_\rho (y)\cap B_{2R}} |\nabla \overline{w} |\, dx \le C R^{\frac n2} \|\nabla \overline{w}\|_{L^2 ( B_{2R} )}\le C \rho^{n-1} .
$$
 If $0< \rho < 2R/3$, then $ \overline{B}_{2\rho} (y) \subset B_{4R}$; hence choosing a cut-off function $\eta \in C_c^\infty ( {B}_{2\rho }(y); [0,1])$ such that $\eta =1$ on $B_{\rho}(y)$, we have
$$
 \|\nabla \overline{w} \|_{L^2 (B_\rho (y) )}
\le C \left[ \frac {\rho^{n/2}}{\rho} + R^{\frac np -1}\left(\rho^n \right)^{\frac{p-2}{2p}} \right] \le C {\rho}^{\frac n2 -1}
$$
and so
$$
\int_{B_\rho (y)\cap B_{2R}} |\nabla \overline{w} |\, dx \le C \rho^{\frac n2} \|\nabla \overline{w}\|_{L^2 (B_\rho (y) )} \le C\rho^{n-1} $$
for some $C=C(n, \|{\mathbf b}\|_{n,\infty})$. Therefore, by the classical John-Nirenberg estimate (see \cite[Theorem 7.21]{GilTru} e.g.), there exist constants $C =C(n)>0$, $w_0 \in \R$, and $ 0< p_0 = p_0 (n, \|{\mathbf b}\|_{L^{n,\infty}(\Om )})<1 $ such that
$$
\int_{B_{2R}} \exp \left( p_0 \left| \overline{w} -w_0 \right| \right) \,d x \le C R^n .
$$
Note that
$$
\begin{aligned}
&\left( \int_{B_{2R}} w^{p_0} \, dx \right) \left( \int_{B_{2R}} w^{-p_0} \, dx \right)\\
& \qquad = \left(\int_{B_{2R}} \exp \left(p_0 \overline{w} \right) \, dx \right) \left( \int_{B_{2R}} \exp \left( -p_0 \overline{w} \right) \, dx \right)\\
& \qquad \le  (C R^n )^2 \exp \left( p_0 w_0 \right)\exp \left( - p_0 w_0 \right) =C^2 R^{2n}
\end{aligned}
$$
and so
\begin{equation}\label{ineq for p0}
\|w\|_{L^{p_0} ( B_{2R} ) } \le C (n,  \norm{\bb}_{L^{n,\infty}(\Om)}  ) R^{{2n}/{p_0}} \|w\|_{L^{-p_0}(B_{2R} ) } .
\end{equation}
Therefore, combining (\ref{ineq for p0}) and (\ref{inf-bound for v}), we complete the proof of Lemma \ref{Harnack inequality}. \end{proof}

\subsection{H\"{o}lder regularity}

By a standard argument based on Lemma \ref{Harnack inequality}, we can prove   H\"{o}lder regularity of weak solutions of \eqref{bvp-dual}.

\begin{proposition}\label{th4-0} Let $\Om $ be a bounded Lipschitz domain in $\R^n$, $n \ge 3$. Assume that $\mathbf{b} \in L^{n, \I}(\Om;\R^n)$,  $\div \mathbf{b} \ge 0$ in $\Om$, and ${\mathbf G} \in L^{p}(\Om ;\R^n )$ for some $n< p< \infty$. Let  $v \in W_0^{1,2}(\Om )$ be  a weak solution of \eqref{bvp-dual} with $g =\div {\mathbf G}$. Then
\begin{equation}\label{holder estimate}
v \in C^\al(\overline{\Om}) \quad\mbox{and}\quad \|v\|_{C^\al(\overline{\Om})} \le
C \|{\mathbf G}\|_{L^{p} (\Omega )}
\end{equation}
for some constants $0< \alpha \le 1-n/p$ and $C>0$, depending only on $n, p, \Omega $, and $ \|{\mathbf b}\|_{L^{n,\I}(\Om )} $.
\end{proposition}

\begin{proof}
By the $L^\infty$-bound in Lemma \ref{th4-boundedness}, it remains to prove the global H\"{o}lder regularity and estimate for $v$. To do this, we need to assume that $\Om$ is a bounded Lipschitz domain in $\R^n$. Then there are constants $R_0 =R_0 (\Om ) >0$ and $C =C(\Om )>0$ such that
\begin{equation}\label{regularity for omega}
\frac{\left|B_{2R}(x_0) \setminus \Omega \right|}{R^n} \ge \frac 1C \quad\mbox{for all}\,\, x_0 \in \partial\Omega , \,\, 0< R \le R_0 /4 .
\end{equation}

We first prove the interior result. Suppose that $B_{R_0}= B_{R_0}(x_0 ) \subset {\Omega}$ and $0<R \le R_0 /4$. For $0< r \le 4R$, we define
$$
M_r= \sup_{B_{r}} v \quad\mbox{and} \quad m_r =\inf_{B_r} v .
$$
Then since $w =M_{4R} -v$ and $w= v-m_{4R}$ are nonnegative bounded weak solutions of
$$
-\Delta w - {\mathbf b} \cdot \nabla w = \mp  \div {\mathbf G} \quad\mbox{in}\,\, B_{4R},
$$
it follows from the inequality (i) in Lemma \ref{Harnack inequality} that there are numbers $0< p_0 = p_0 (n, \|{\mathbf b}\|_{L^{n,\infty}(\Om ) })< 1 $ and $C =C(n,p, \|{\mathbf b}\|_{L^{n,\infty}(\Omega )} )>1$ such that
$$
R^{-n/p_0} \| M_{4R}-v \|_{L^{ p_0}( B_{2R} )} \le C \left( M_{4R}-M_R + R^\delta \|{\mathbf G}\|_{L^{p} (\Omega )}   \right)
$$
and
$$
R^{-n/p_0} \| v - m_{4R} \|_{L^{p_0}( B_{2R} ) }\le C \left( m_R -m_{4R} + R^\delta \|{\mathbf G}\|_{L^{p} (\Omega )}\right) ,
$$
where $\delta = 1-n/p$. Adding two inequalities, we have
$$
M_{4R} -m_{4R} \le C \left( M_{4R}-m_{4R} -M_R +m_R + R^\delta \|{\mathbf G}\|_{L^{p} (\Omega )}\right)
$$
and so
$$
M_R -m_R \le \gamma \left( M_{4R}-m_{4R}\right) + R^\delta \|{\mathbf G}\|_{L^{p} (\Omega )},
$$
where $0< \gamma =1-C^{-1} <1 $. Hence by a standard iteration lemma (see \cite[Lemma 8.23]{GilTru} e.g.), there are constants $0< \alpha =\alpha (n,p,\|{\mathbf b}\|_{L^{n,\infty}(\Omega )} ) \le 1-n/p$ and $C =C(n,p,\Om , \|{\mathbf b}\|_{L^{n,\infty}(\Omega )}  )>1$ such that
$$
{\rm osc}_{B_R} v = M_R -m_R \le C R^\alpha \left( {\rm osc}_{B_{R_0}} v + \|{\mathbf G}\|_{L^{p} (\Omega )} \right)
$$
for all $R\le R_0 /4$.

 To prove the boundary result, we assume that $x_0 \in \partial \Omega$ and $0<R \le R_0 /4$. For $0< r \le 4R$, we define
$$
M_r= \sup_{\Omega \cap B_{r}(x_0 )} v \quad\mbox{and} \quad m_r =\inf_{\Omega \cap B_{r}(x_0 )} v .
$$
Then since $v=0$ on $\partial\Omega$, it follows that $m_r \le 0 \le M_r$ for all $r \le 4R$.
Moreover, since $w =M_{4R} -v$ and $w= v-m_{4R}$ are nonnegative bounded weak solutions of
$$
-\Delta w - {\mathbf b} \cdot \nabla w = \mp  \div {\mathbf G} \quad\mbox{in}\,\, \Omega \cap B_{4R} (x_0 ),
$$
it follows from the inequality (ii) in Lemma \ref{Harnack inequality} that
$$
M_{4R} \left(\frac{\left|B_{2R}(x_0) \setminus \Omega \right|}{R^n} \right)^{1/p_0} \le C \left( M_{4R}-M_R + R^\delta \|{\mathbf G}\|_{L^{p} (\Omega )}\right)
$$
and
$$
-m_{4R} \left(\frac{\left|B_{2R}(x_0) \setminus \Omega \right|}{R^n} \right)^{1/p_0} \le C \left( m_R -m_{4R} + R^\delta\|{\mathbf G}\|_{L^{p} (\Omega )} \right) ,
$$
where $0< p_0 = p_0 (n, \|{\mathbf b}\|_{L^{n,\infty}(\Om ) })< 1 $ and $C =C(n,p, \|{\mathbf b}\|_{L^{n,\infty}(\Omega )} )>1$. Recalling the regularity condition (\ref{regularity for omega}) on $\Omega$ and adding two inequalities, we have
$$
M_{4R} -m_{4R} \le C \left( M_{4R}-m_{4R} -M_R +m_R + R^\delta \|{\mathbf G}\|_{L^{p} (\Omega )} \right)
$$
and so
$$
M_R -m_R \le \gamma \left( M_{4R}-m_{4R}\right) + R^\delta \|{\mathbf G}\|_{L^{p} (\Omega )} ,
$$
where $0< \gamma =1-C^{-1} <1 $. Hence by a standard iteration lemma, there are constants $0< \alpha =\alpha (n,p, \Om , \|{\mathbf b}\|_{L^{n,\infty}(\Omega )} ) \le 1-n/p$ and $C =C(n,p, \Om , \|{\mathbf b}\|_{L^{n,\infty}(\Omega )} )>1$ such that
$$
{\rm osc}_{\Omega \cap B_R (x_0 ) } v = M_R -m_R \le C R^\alpha \left( {\rm osc}_{\Omega \cap B_{R_0}(x_0 )} v + \|{\mathbf G}\|_{L^{p} (\Omega )} \right)
$$
for all $R\le R_0 /4$. Combining the interior and boundary estimates, we complete the proof of
the proposition.
\end{proof}

\begin{remark}\label{holder regularity for u}
 As shown in Proposition \ref{th4-0}, H\"{o}lder regularity of weak solutions $v$ of   \eqref{bvp-dual} is  deduced from the weak Harnack inequalities in Lemma \ref{Harnack inequality}. But this approach seems not to work for  weak solutions $u$ of  \eqref{bvp}. Here we explain one difficulty.

Assume that $\bb \in L^{n, \infty} (\Om; \R^n )$,  $\div \bb \ge 0$ in $\Om$, and $\mathbf{F}\in L^p (\Om ; \R^n )$ for some $p>n$. Then by Lemma \ref{th4-boundedness}, there exists a unique bounded weak solution  $u  $     of  \eqref{bvp} with $f = \div \mathbf{F} $. If  $ B_{4R} = B_{4R} (x_0 )\subset \Om$, let $M_{4R} =\sup_{B_{4R}} u$ and $m_{4R} =\inf_{B_{4R}} u$. Then   the functions $w_1 = u - m_{4R}$  and $w_2 =M_{4R} -u$  are nonnegative and satisfy
\[
-\Delta w_i +\div (w_i \bb )= \div    {\mathbf{F}}_i   \quad\mbox{in}\,\,B_{4R}    ,
\]
where $ {\mathbf{F}}_1  = \mathbf{F}  - m_{4R}  \bb$ and $ {\mathbf{F}}_2  =  M_{4R}   \bb - \mathbf{F}$. Recall again that $\div \bb \ge 0$ in $\Om$. Hence both $w_1$ and $w_2$   satisfy
\begin{equation}\label{u-inequality for weak harnack}
-\Delta w_i +\div (w_i \bb ) \ge \pm  \div    \mathbf{F}\quad\mbox{in}\,\,B_{4R} ,
\end{equation}
only when $m_{4R} \le 0 \le M_{4R} $,   which holds if and only if there exists at least one $x \in B_{4R}(x_0 )$ such that $u(x )=0$. But we do not know how to prove such a property for a weak solution  $u$ of  \eqref{bvp} under the only assumption that $\mathbf{F}\in L^p (\Om ; \R^n )$.
This is one reason why we can not  derive   interior H\"{o}lder estimates for $u$ from the weak Harnack inequalities. Furthermore, we have not successfully obtained the weak Harnack inequalities for functions $w_i$ satisfying  (\ref{u-inequality for weak harnack}) yet, even under the additional assumption that $\div u \in L^{n/2, \infty}(\Om )$. 
\end{remark}

\section{Proofs of Theorems \ref{th4-q version} and \ref{th4-q version-strong}}
\label{Sec6}

In this section we prove Theorems \ref{th4-q version} and \ref{th4-q version-strong} which are two main results in the paper for  existence, uniqueness, and regularity of $p$-weak solutions of the problems \eqref{bvp} and \eqref{bvp-dual}. The domain $\Om$  can be either a bounded or exterior domain in $\R^n$, $n \ge 3$. It is assumed that $\Om$ is of  the class $C^{1}$ to start with and then  $C^{1,1}$ for second derivative estimates.

\subsection{Proofs of Theorems \ref{th4-q version} and \ref{th4-q version-strong} for bounded domains}

To prove Theorem \ref{th4-q version} for the case when $\Omega$ is bounded and $2 < p<n$, it suffices to prove the following result.

\begin{proposition}\label{existence-qweak} Let $\Om $ be a bounded $C^1$-domain in $\R^n$, $n \ge 3$. Assume that $\mathbf{b} \in L^{n,\I}(\Om;\R^n)$, $\div \mathbf{b} \ge 0$ in $\Om$, and $2 \le p<n$.
\begin{enumerate}[(i)]
\item For each ${\mathbf F} \in L^{p}(\Om; \R^n )$, there exists a unique $p$-weak solution $u $ of \eqref{bvp} with $f= \div {\mathbf F}$. Moreover, we have
$$
 \|\nb u \|_{L^p (\Om )} \le C (n,p ,\Om ) \left( 1+ \|{\mathbf b}\|_{L^{n,\I} (\Om )} \right) \|{\mathbf F}\|_{L^{p}(\Om )}.
$$
\item For each ${\mathbf G} \in L^{p'}(\Om; \R^n )$, there exists a unique $p'$-weak solution $v $ of \eqref{bvp-dual} with $g= \div {\mathbf G}$. Moreover, we have
$$
 \|\nb v \|_{L^{p'} (\Om )} \le C (n,p ,\Om ) \left( 1+ \|{\mathbf b}\|_{L^{n,\I} (\Om )} \right) \|{\mathbf G}\|_{L^{p'}(\Om )}.
$$
\end{enumerate}
\end{proposition}

\begin{proof}
 For the case $p=2$, the proposition was already proved by Lemma \ref{th4-existence}. Moreover, it follows from Lemma \ref{th4-existence} that
for each $\mathbf{F} \in L^{2}(\Om; \R^n )$ there exists a unique weak solution $u=T({\mathbf F})$ of \eqref{bvp} with $f =\div \mathbf{F}$. Define $S(\mathbf{F})= \nb T(\mathbf{F})$ for all $\mathbf{F} \in L^{2}(\Om; \R^n )$. Then by Lemma \ref{th4-existence}, $S$ is a bounded linear operator from $L^{2}(\Om; \R^n )$ into $L^{2}(\Om; \R^n )$.

Suppose that $2 < p<n$, $\mathbf{F} \in L^{p}(\Om; \R^n )$, and $u=T({\mathbf F})$. Then by Lemma \ref{th4-higher integrability}, we deduce that $u \in L^{p^*}(\Om )$ and $\|u \|_{L^{p^*}(\Om)}\le C \|\mathbf{F}\|_{L^p (\Om )}$.
Moreover, since $\mathbf{b} \in L^{n,\I}(\Om;\R^n)$, it follows from Lemma \ref{holder inequalies} that $u \mathbf{b} \in L^{p,\I}(\Om;\R^n)$ and $\|u \mathbf{b}\|_{L^{p,\I}(\Om )} \le C \|u \|_{L^{p^*}(\Om)} \|{\mathbf b}\|_{L^{n,\I} (\Om )}$.
 Note that $-\De u = \div \left(\mathbf{F}- u \mathbf{b} \right)$ in $\Om$. Hence by Proposition \ref{CZ-estimates}, we deduce that
$S(\mathbf{F})= \nb u \in L^{p,\I}(\Om;\R^n)$ and $ \|S(\mathbf{F}) \|_{L^{p,\I}(\Om )} \le C \left( 1+ \|{\mathbf b}\|_{L^{n,\I} (\Om )} \right) \|{\mathbf F}\|_{L^{p}(\Om )}$.
This implies, in particular, that $S$ is bounded from $L^{p}(\Om; \R^n )$ into $L^{p,\I}(\Om; \R^n )$.

We have shown that $S$ is a bounded linear operator from $L^{p}(\Om; \R^n )$ into $L^{p,\I}(\Om; \R^n )$ and its operator norm is bounded above by $C \left( 1+ \|{\mathbf b}\|_{L^{n,\I} (\Om )} \right)$ for every $2<p<n$. Therefore, by Lemma \ref{marcinkiewicz-inter} (the Marcinkiewicz interpolation theorem), we conclude that for every $2<p<n$, $S$ is bounded from $L^{p}(\Om; \R^n )$ into $L^{p}(\Om; \R^n )$ and its operator norm is bounded above by $C \left( 1+ \|{\mathbf b}\|_{L^{n,\I} (\Om )} \right)$.
This proves Part (i) of the proposition.

To prove Part (ii) of the proposition, it suffices to derive some a priori estimate for solutions.
 Suppose that $\mathbf{G} \in L^{p'} (\Om; \R^n )$ and $v$ is a $p'$-weak solution of \eqref{bvp-dual} with $g =\div \mathbf{G}$. Let $\mathbf{F} \in L^{p} (\Om; \R^n )$ be given. Then by Part (i) of the proposition,
there exists a unique $p$-weak solution $u $ of \eqref{bvp} with $f= \div {\mathbf F}$ satisfying
$ \|\nb u \|_{L^{p} (\Om )} \le C (n,p ,\Om ) \left( 1+ \|{\mathbf b}\|_{L^{n,\I} (\Om )} \right) \|{\mathbf F}\|_{L^{p}(\Om )}.$
Using Lemma \ref{th1} (or Lemma \ref{th1-general}), we easily deduce that
$$
\int_\Om \left( \nb u - u \mathbf{b} \right) \cdot\nb \phi \, dx = - \int_\Om \mathbf{F} \cdot \nabla \phi \, dx \quad\mbox{for all}\,\, \phi \in \hat{D}^{1,p'}_{0}(\Om).
$$
Taking $\phi =v$, we thus have
\begin{align*}
- \int_\Om \mathbf{F} \cdot \nabla v \, dx &= \int_\Om \nb v \cdot \left( \nb u - u \mathbf{b} \right) \, dx \\
&= - \int_\Om \mathbf{G} \cdot \nabla u \, dx \\
&\le \| \mathbf{G} \|_{L^{p'} (\Om )} \| \nabla u \|_{L^{p}(\Om )}\\
&\le C (n,p ,\Om ) \left( 1+ \|{\mathbf b}\|_{L^{n,\I} (\Om )} \right) \| \mathbf{F} \|_{L^p (\Om )} \|{\mathbf G}\|_{L^{p'}(\Om )}.
\end{align*}
Since $\mathbf{F} \in L^{p} (\Om; \R^n )$ is arbitrary, it follows from the Riesz representation theorem that
$$
\|\nabla v \|_{L^{p'} (\Om )} \le C (n,p ,\Om ) \left( 1+ \|{\mathbf b}\|_{L^{n,\I} (\Om )} \right) \| \mathbf{G} \|_{L^{p'} (\Om )}.
$$
Uniqueness of a $p'$-weak solution of \eqref{bvp-dual}
immediately follows from this a priori estimate. Moreover, by a density argument, we easily deduce existence of a $p'$-weak solution of \eqref{bvp-dual}. This completes the proof.
\end{proof}

\medskip
We are now ready to prove Theorems \ref{th4-q version} and \ref{th4-q version-strong} for bounded domains

\begin{proof}
[Proof of Theorem \ref{th4-q version} for bounded domains]
For the case $ 2\le p < n$, the theorem immediately follows from Lemma \ref{th4-existence} and Proposition \ref{existence-qweak}.
In particular, it follows from Lemma \ref{th4-existence} that
for each $\mathbf{G} \in L^{2}(\Om; \R^n )$ there exists a unique weak solution $v=T({\mathbf G})$ of \eqref{bvp-dual} with $g =\div \mathbf{G}$. Define $S(\mathbf{G})= \nb T(\mathbf{G})$ for all $\mathbf{G} \in L^{2}(\Om; \R^n )$. Then by Lemma \ref{th4-existence}, $S$ is a bounded linear operator from $L^{2}(\Om; \R^n )$ into $L^{2}(\Om; \R^n )$.

Suppose next that $ n/(n-1) < p< 2 $ and $q=p'$. To derive some a priori estimates, let us suppose that $\mathbf{G} \in L^{q}(\Om; \R^n )$ and $v=T({\mathbf G})$. Then by Lemma \ref{th4-higher integrability}, we deduce that $v \in L^{q^*}(\Om )$ and $\|v \|_{L^{q^*}(\Om)}\le C \|\mathbf{G}\|_{L^{q} (\Om )}$.
Moreover, since $\mathbf{b} \in L^{n,\I}(\Om;\R^n)$ and $ {\rm div}\, \mathbf{b} \in L^{n/2,\I}(\Om)$, it follows from Lemma \ref{holder inequalies} that
$$
\|v \mathbf{b}\|_{L^{q,\I}(\Om )} \le C \|v \|_{L^{q^*}(\Om)} \|{\mathbf b}\|_{L^{n,\I} (\Om )}
$$
and
$$ \|v \, {\rm div}\, \mathbf{b} \|_{L^{nq/(n+q),\I}(\Om )} \le C \|v \|_{L^{q^*}(\Om)} \|{\rm div}\, {\mathbf b}\|_{L^{n/2,\I} (\Om )} .
$$
Note that
\[
-\De v = \div \mathbf{G}+\mathbf{b} \cdot \nabla v = \div (\mathbf{G}+v \mathbf{b} ) - v \, {\rm div}\, \mathbf{b}
\]
in $\Om$. But since $1< nq/(n+q)<n/2$, it follows from Lemma \ref{sol-Laplace in Rn} (ii) that there exists ${\mathbf G}_0 \in L^{q,\I}(\Om ;\R^n )$ such that
$$
- v \, {\rm div}\, \mathbf{b} = {\rm \div} \, \mathbf{G}_0 \quad\mbox{and}\quad \| \mathbf{G}_0 \|_{L^{q,\I}(\Om )} \le C \|v \|_{L^{q^*}(\Om)} \|{\rm div}\, {\mathbf b}\|_{L^{n/2,\I} (\Om )} .
$$
Hence by Proposition \ref{CZ-estimates}, we deduce that
$$
S(\mathbf{G})= \nb v \in L^{q,\I}(\Om;\R^n)
$$
and
$$
\|S(\mathbf{G}) \|_{L^{q,\I}(\Om )} \le C(n,p,\Omega ) M_{\mathbf b} \|{\mathbf G}\|_{L^{q}(\Om )},
$$
 where $M_{\mathbf b} = 1+ \|{\mathbf b}\|_{L^{n,\I} (\Om )}+ \|{\rm div}\, {\mathbf b}\|_{L^{n/2,\I} (\Om )} $.
This implies, in particular, that $S$ is bounded from $L^{q}(\Om; \R^n )$ into $L^{q,\I}(\Om; \R^n )$.

We have shown that $S$ is a bounded linear operator from $L^{q}(\Om; \R^n )$ into $L^{q,\I}(\Om; \R^n )$ and its operator norm is bounded above by $C(n,q',\Omega )M_{\mathbf b}$ for every $2<q<n$. Therefore, by the Marcinkiewicz interpolation theorem  (Lemma \ref{marcinkiewicz-inter}), we conclude that for every $2<q<n$, $S$ is bounded from $L^{q}(\Om; \R^n )$ into $L^{q}(\Om; \R^n )$ and its operator norm is bounded above by $C(n,q',\Omega )M_{\mathbf b}$.
This completes the proof of Part (ii) of the theorem for the case $ n/(n-1) < p< 2 $. Part (i) can be then deduced from Part (ii) by a duality argument as in the proof of Proposition \ref{existence-qweak}.
\end{proof}

\medskip

\begin{proof}
[Proof of Theorem \ref{th4-q version-strong} for bounded domains]
Suppose that $1<q<n/2$. Recall from the Calderon-Zygmund and Sobolev inequalities that if
 ${\mathbf F} = - \nabla N(f)$, where $N(f)$  is the Newtonian potential of $f$ over $\Om$,
then $f= {\rm div} \,{\mathbf F} $ in $\Om$ and
$\|{\mathbf F}\|_{L^{q^*} (\Omega )} \le C \|f\|_{L^q (\Om )}$. Moreover, since $n/(n-1)<q^*< n$, it follows from Proposition \ref{existence-qweak} that for each $f \in L^{q}(\Om ) $, there exists a unique $q^*$-weak solution $u= T(f)$ of \eqref{bvp}, satisfying $\|u \|_{W^{1,q^*} (\Om )} \le C(n,q, \Omega ) M_{\mathbf b}\|f\|_{L^{q}(\Om )}$.
Note that $-\De u = f- \mathbf{b} \cdot \nb u - \left( {\rm div}\, {\mathbf b} \right) u $ in $\Om$. Hence by Lemma \ref{holder inequalies} and Proposition \ref{CZ-estimates}, we deduce
 that $u \in W^{2, q, \I}(\Om;\R^n)$ and
$ \|u \|_{W^{2,q,\I}(\Om)} \le C(n,q ,\Omega )M_{\mathbf b}^2 \|f\|_{L^{q}(\Om )} $.
 We have shown that the linear mapping $T $ is bounded from $L^{q}(\Om )$ into
$W^{2,q,\I}(\Om )$ and its norm is bounded above by $C(n,q , \Omega )M_{\mathbf b}^2$ for every $1<q<n/2$. Therefore, by the Marcinkiewicz interpolation theorem  (Lemma \ref{marcinkiewicz-inter}),
we conclude that $T $ is a bounded linear operator from $L^{q}(\Om )$ into $W^{2,q}(\Om )$ with norm bounded above by $C(n,q , \Omega )M_{\mathbf b}^2$ for every $1<q<n/2$. This completes the proof of Part (i) of the theorem. Part (ii) can be proved by exactly the same argument.
 \end{proof}

\subsection{Uniqueness results for exterior problems}

Theorem \ref{th4-q version} has been completely proved for the case when $\Om$ is a bounded domain. For exterior domains, we first prove its uniqueness assertion. In fact, making essential use of the results of Theorem \ref{th4-q version} for bounded domains, we   prove the following uniqueness result for weak solutions of the problem \eqref{bvp} on exterior domains.
 We will also prove a uniqueness result for weak solutions of the problem \eqref{bvp-dual} on exterior domains in Lemma \ref{th4-unique-q version-dual}.

\begin{lemma}\label{th4-unique-q version} Let $\Om $ be an exterior $C^1$-domain in $\R^n$, $n \ge 3$. Assume that $\mathbf{b} \in L^{n,\I}(\Om;\R^n)$ and $\div \mathbf{b} \ge 0$ in $\Om$. Assume in addition that
\begin{enumerate}[(i)]
\item $u \in W_{\loc}^{1,r}(\overline{\Om})$ for some $r> n'  $, $u =0$ on $\partial \Om$;
\item $u \in L^{p_1 } (\Om ) + L^{p_2 } (\Om )$ for some $p_1 , p_2$ satisfying
 $ n'  <  p_1 \le p_2 <\infty$;
\item ${\rm div}\, {\mathbf b} \in L^{n/2,\I} (\Om )$ if $r < 2$ or $p_1 < 2$; and
\item $u$ satisfies
\begin{equation}\label{eq1.3-012}
\int_\Om \bke{\nb u - u \mathbf{b}} \cdot \nb
\phi \, dx = 0 \quad \text{for all } \phi \in C^1_{c}(\Om).
\end{equation}
\end{enumerate}
 Then $u =0$ identically on $\Om$.
\end{lemma}

\begin{proof} Choose  $R_0 >0$ so large that $\Omega^c \subset B_{R_0 } =B_{R_0}(0)$ and define $\Om_{R} = \Om \cap B_{R}$ for $R>R_0$.

We first show that $u \in L_{\loc}^{\infty}(\overline{\Omega}) \cap W_{\loc}^{1,2}(\overline{\Omega}) $. Given  $R>R_0$, we fix a  cut-off function $\eta \in C_c^\infty (B_{2R}; [0,1])$   with  $\eta =1$ on $B_{R}$ and define $\overline{u}=\eta u$.  Then  $\overline{u} $ belongs to $W_0^{1,r}(\Om_{2R} )$ and satisfies
\[
-\De \overline{u} + \div (\overline{u} \bb) =  - \div ( 2u\nb \eta  )  +u \De\eta + (u\bb   )\cdot \nb \eta  \quad\mbox{in}\,\, \Om_{2R};
\]
that is,
$$
\int_{\Om_{2R}} \left( \nb \overline{u} - \overline{u}\mathbf{b} \right) \cdot \nb
\phi \, dx = \int_{\Om_{2R}} u \bke{ 2\nb \eta  \cdot \nb \phi +\phi \De \eta  + \phi \mathbf{b} \cdot \nb \eta  } \, dx
$$
for all $\phi \in C_c^1 (\Om_{2R})$.
   Suppose now that $u \in L_{loc}^{q} (\overline{\Om})$ for some $   q  \in [p_1 ,  \infty  )  $.
Then since $1 <    q'  \le p_1 ' < n$, it follows from Lemma \ref{th1-general} that
$$
 \left| \int_{\Om_{2R}} u \bke{ 2\nb \eta  \cdot \nb \phi +\phi \De \eta  + \phi \mathbf{b} \cdot \nb \eta  } \, dx\right| \le C \|u \|_{L^{q} (\Om_{2R} )}\| \phi \|_{W^{1,q'} (\Om_{2R} )}
$$
for all $\phi \in W_0^{1,q'}(\Om_{2R} )$, where $C= C(n,  q,  \eta ,  \mathbf{b}  )$.
Define $\overline{q}= q^*$ if $q<n$,  $\overline{q}=2n $ if $q  = n$, and  $\overline{q}= \infty $ if $q >n$. Then  by Lemma \ref{dual-char}, Theorem \ref{th4-q version} (for bounded domains, when $n'< q\le n$), and Lemma \ref{th4-boundedness} (when $q>n$),
there exists a unique $ {w}\in W_0^{1,q_1}(\Om_{2R} ) \cap L^{\overline{q}}(\Om_{2R})$, where $q_1 =\min (2,q)$, such that
\[
-\De w + \div (w \bb)=   - \div ( 2u\nb \eta  )  +u \De\eta + (u\bb )\cdot \nb \eta  \quad\mbox{in}\,\,\Om_{2R}
\]
and
$$
\|w\|_{L^{\overline{q}}(\Om_{2R})} \le C(n,  q,  \Om , R ,  \eta ,  \mathbf{b}  )  \|u \|_{L^{q} (\Om_{2R} )} .
$$
Since $q_0 := \min (r,q_1 )> n' $ and $ \overline{u} , {w} \in W_0^{1,q_0 }(\Om_{2R} )$, it follows from the uniqueness assertion of Theorem \ref{th4-q version} (for bounded domains)
 that $\overline{u} = {w}$ on $\Omega_{2R}$. This proves that
$$
 u \in L^{\overline{q}}(\Omega_{R}) \quad\mbox{and}\quad \|u\|_{L^{\overline{q}}(\Om_{R})} \le C(n,  q,  \Om , R ,  \eta ,  \mathbf{b}  )  \|u \|_{L^{q} (\Om_{2R} )} .
$$
Since $R>R_0$ can be arbitrarily large, it follows that   $u \in    L_{\loc}^{\overline{q}}(\overline{\Omega})$.
This argument can be repeated finitely many times to show that $u \in L_{loc}^\infty (\overline{\Om})$; that is,  by a bootstrap argument starting from  $q =p_1$, we can deduce that
$$%
u \in L^{\infty}(\Om_{R}) \quad\mbox{and}\quad    \|u\|_{L^{\infty}(\Om_{R})} \le C(n,   p_1 , \Om ,  R ,  \mathbf{b}  )  \|u \|_{L^{p_1} (\Om_{2R} )}
$$%
for all $R>R_0$. Moreover, since $u \in L_{loc}^2 (\overline{\Om})$, it follows from Theorem \ref{th4-q version} (for bounded domains) that   $u \in W_{\loc}^{1,2}(\overline{\Omega}) $.

We next show that $u \in L^p (\Om )$ for any $p \in [p_2 , \infty ]$.
Let $x_0 \in \Om$ be any point with $|x_0|>R_0 + 2$. Choosing a  cut-off function $\zeta \in C_c^\infty (B_{2}(x_0 ); [0,1])$   with  $\zeta =1$ on $B_{1}(x_0 )$, we define ${\overline{u}} = u \zeta  $.
Then ${\overline{u}}    \in W_0^{1,2}(B_2 (x_0 )) $   satisfies
\[
-\De \overline{u} + \div (\overline{u} \bb) =   - \div ( 2u\nb \zeta  )  +u \De\zeta + (u\bb   )\cdot \nb \zeta \quad\mbox{in}\,\, B_{2}(x_0 ).
\]
Hence by the same bootstrap  argument as above,  we can deduce that
$$
   \|u\|_{L^{\infty}(B_{1}(x_0 ) )}   \le C(n,   p_1 ,  \mathbf{b}   )     \|u \|_{L^{p_1} (B_{2}(x_0 ) )}  .
$$
Since $u \in L^{p_1}(\Om)+L^{p_2}(\Om)$ and $p_1 \le p_2$, we have
\[
\norm{u}_{L^\I(B_1(x_0))}  \le C \norm{u}_{L^{p_1}(B_2(x_0))}
 \le C\norm{u}_{L^{p_1}(\Om)+L^{p_2}(\Om)}<\infty,
\]
where $C = C(n, p_1, p_2 , \mathbf{b} ) $ is independent of any $x_0 \in \Om$ with $|x_0|>R_0 +2$. Hence it follows  that  $u \in L^\infty(\Om)$. To show that $u \in L^{p_2}(\Om )$, we decompose $u=u_1+u_2$, where $u_1 \in L^{p_1}(\Om)$ and $u_2 \in L^{p_2}(\Om)$.
If   $S= \{ x \in \Om : |u_1(x)|\ge 1\}$, then
\EQN{
\int_{\Om \setminus S} |u_1|^{p_2} \, dx  \le \int_{\Om \setminus S} |u_1|^{p_1} \, dx \le \int_{\Om} |u_1|^{p_1} \, dx < \infty.
}
Moreover,  since
\[
|S| \le \int_\Om | u_1|^{p_1} \, dx
\]
by Chebyshev's inequality,   we have
\begin{align*}
\|u_1 \|_{L^{p_2}(S )}  & \le \|u_2\|_{L^{p_2}( S)} + \|u \|_{L^{p_2}(S)}\\
 &\le \|u_2\|_{L^{p_2}( S)} + \|u \|_{L^{\infty}(S)}|S|^{1/p_2} <\infty .
\end{align*}
This  shows that  $u_1 \in L^{p_2}(\Om)$. Hence it follows that  $u = u_1 +u_2 \in L^{p_2}(\Om)$.
By H\"older's inequality,
we conclude  that  $u\in L^{p}(\Om)$ for any $p \in [p_2,\I]$.

We finally show that $u =0$ on $\Om$.
For a fixed number $\beta \ge 1$, let $w= u^+$ and $G(w) = w^{\beta}$. Since $u \in W_{\loc}^{1,2}(\overline{\Omega}) \cap L^{\infty}(\Omega)$, we easily show that $G(w) \in W_{\loc}^{1,2}(\overline{\Omega}) $.
 Let $k \in \N $ be so large that $\Om^c \subset B_k$. Fix a  cut-off function $\eta \in C_c^\infty (B_{2}; [0,1])$   with  $\eta =1$ on $B_{1}$ and define $\eta_k (x) = \eta (x/k)$.
Then since $\eta_k G(w) \in W_0^{1,2}(\Om_{2k})$, there is a sequence $\{\phi_j \}$ in $C_c^\infty (\Om_{2k})$ such that $\phi_j \to \eta_k G( w)$ in $W^{1,2}(\Om_{2k})$ as $j \to \infty$. Taking $\phi = \eta_k \phi_j$ in (\ref{eq1.3-012}) and letting $j\to\infty$, we have
$$
\int_{\Om} \nb u \cdot \nb \left(\eta_k^2 w^{\beta} \right) \, dx = \int_{\Om} ( u \mathbf{b} ) \cdot \nb \left(\eta_k^2 w^{\beta} \right) \, dx .
$$
Since $\nb w = \chi_{\{u> 0 \}} \nb u$ and $\div {\mathbf b} \ge 0$ in $\Om$, it follows that
\begin{align*}
 \nabla u \cdot \nb \left(\eta_k^2 w^{\beta} \right) & = \beta \eta_k^2 w^{\beta-1} |\nabla w|^2 + 2 \eta_k w^{\beta} \nabla \eta_k \cdot \nabla w \\
 &\ge \frac{\beta}2 \eta_k^2 w^{\beta-1} |\nabla w|^2 - \frac{2}{\beta}  w^{\beta+1} |\nabla \eta_k |^2
\end{align*}
and
\begin{align*}
\int_{\Om} u \mathbf{b} \cdot \nb
\left(\eta_k^2 w^\beta \right) \, dx &= \int_{\Om} w \mathbf{b} \cdot \nb
\left(\eta_k^2 w^\beta \right) \, dx \\
&= \frac{1}{\beta +1} \int_{\Om} \mathbf{b} \cdot \left[ \beta \nb \left( \eta_k^2 w^{\beta +1} \right) + 2 \eta_k w^{\beta +1} \nabla \eta_k \right] \, dx\\
 & \le \frac{2 }{\beta +1} \int_{\Om} \eta_k w^{\beta +1} \nabla \eta_k \cdot \mathbf{b} \, dx.
\end{align*}
Hence, letting $\overline{w}=w^{(\beta +1)/2}$, we have
$$
\frac{2 \beta}{(\beta+1)^2} \int_\Omega \eta_k^2 |\nabla \overline{w} |^2 \, dx \le \frac{2 }{\beta +1} \int_{\Om} \eta_k \overline{w}^2 | \nabla \eta_k | | \mathbf{b} |\, dx + \frac{2}{\beta} \int_{\Om} \overline{w}^2 | \nabla \eta_k |^2 \, dx .
$$
By Lemma \ref{th1-general},
\begin{align*}
\int_{\Om} \eta_k \overline{w}^2 | \nabla \eta_k | | \mathbf{b} |\, dx &\le C(n) \| \mathbf{b} \|_{L^{n,\infty}(\Om)} \|\nabla (\eta_k \overline{w})\|_{L^2 (\Omega )} \| \overline{w} \nb \eta_k \|_{L^2 (\Omega)} \\
&\le \frac{\beta}{2(\beta+1)} \|\nabla (\eta_k \overline{w})\|_{L^2 (\Omega )}^2 + C(n , \beta , \mathbf{b} ) \| \overline{w} \nb \eta_k \|_{L^2 (\Omega)}^2 .
\end{align*}
Therefore, using the Sobolev inequality, we obtain
\begin{equation}\label{local L2}
 \|\eta_k \overline{w} \|_{L^{2^*}(\Om )}^2 + \|\nabla (\eta_k \overline{w})\|_{L^2 (\Omega )}^2 \le C \left( n, \beta , {\mathbf b} \right) \| \overline{w} \nb \eta_k \|_{L^2 (\Omega)}^2 .
\end{equation}
Recall now that $w =u^+ \in L^p (\Om )$ for any $p \in [p_2 , \infty ]$. Take   any $\beta \ge  1$ with $\beta +1 \ge   p_2$.
Then by (\ref{local L2}) and  H\"{o}lder's inequality, we have
\begin{align*}
\|\eta_k \overline{w} \|_{L^{2^*}(\Om )} & \le C \| \overline{w} \nb \eta_k \|_{L^2 (\Omega)} \\
& \le \frac{C}{k} \| \overline{w} \|_{L^2 (B_{2k}\setminus B_k )}   =    \frac{C}{k}   \| w  \|_{L^{\beta+1} (B_{2k}\setminus B_k )}^{(\beta+1)/2}  .
\end{align*}
Letting $k\to \infty$, we deduce that   $\overline{w} =0$ on $\Om$.
It follows from the definitions of $\overline{w}$ and $w$ that $u^+ = 0$ on $\Om$. By a similar argument, we can show that $u^- =0$ on $\Omega$. This completes the proof.
\end{proof}

\medskip

Using Proposition \ref{existence-qweak}   and  Lemma \ref{th4-higher integrability} instead of Theorem \ref{th4-q version} (for bounded domains),  we can also obtain the following uniqueness result for the dual problem \eqref{bvp-dual}, under a weaker assumption on $\bb$ than Lemma \ref{th4-unique-q version} (we do not assume that
$\div \bb  \in L^{n/2,\I} (\Om )$ when $r < 2$ or $p_1 < 2$).

\begin{lemma}\label{th4-unique-q version-dual} Let $\Om $ be an exterior $C^1$-domain in $\R^n$, $n \ge 3$. Assume that $\mathbf{b} \in L^{n,\I}(\Om;\R^n)$ %
 and $\div \mathbf{b} \ge 0$ in $\Om$. Assume in addition that
\begin{enumerate}[(i)]
\item $v \in W_{\loc}^{1,r}(\overline{\Om})$ for some $r> n' $, $v =0$ on $\partial \Om$;
\item $v \in L^{p_1 } (\Om ) + L^{p_2 } (\Om )$ for some $p_1 , p_2$ satisfying
 $ n' <  p_1 \le p_2 <\infty$; and
\item $v$ satisfies
\begin{equation}\label{eq1.3-012-dual}
\int_\Om \nb v \cdot \bke{\nb \psi - \psi \mathbf{b}} \, dx = 0 \quad \text{for all } \psi \in C^1_{c}(\Om).
\end{equation}
\end{enumerate}
 Then $v =0$ identically on $\Om$.
\end{lemma}

\begin{proof} Choose  $R_0 >0$ so large that $\Omega^c \subset B_{R_0 } =B_{R_0}(0)$ and define $\Om_{R} = \Om \cap B_{R}$ for $R>R_0$.

We first show that $v \in L_{\loc}^{\infty}(\overline{\Omega}) \cap W_{\loc}^{1,2}(\overline{\Omega}) $. If $\eta \in C_c^\infty (B_{2R}; [0,1])$ is a cut-off function with  $\eta =1$ on $B_{R}$, where   $R>R_0$, then  $\overline{v}=\eta v$ belongs to $W_0^{1,r}(\Om_{2R} )$ and satisfies
\[
-\De \overline{v} -  \bb \cdot \nabla \overline{v} =  - \div ( 2v\nb \eta  )  +v \De\eta - (v\bb   )\cdot \nb \eta  \quad\mbox{in}\,\, \Om_{2R}.
\]
Suppose   that $v \in L_{loc}^{q} (\overline{\Om})$ for some $   q  \in [p_1 ,  \infty  )  $.
Then since $1 <    q'    < n$, it follows from Lemma \ref{th1-general} that
$$
 \left| \int_{\Om_{2R}} v \bke{ 2\nb \eta  \cdot \nb \psi +\psi \De \eta  - \psi \mathbf{b} \cdot \nb \eta  } \, dx\right| \le C \|v \|_{L^{q} (\Om_{2R} )}\| \psi \|_{W^{1,q'} (\Om_{2R} )}
$$
for all $\phi \in W_0^{1,q'}(\Om_{2R} )$, where $C= C(n,  q,  \eta ,  \mathbf{b}  )$.
Define $\overline{q}= q^*$ if $q<n$,  $\overline{q}=2n $ if $q  = n$, and  $\overline{q}= \infty $ if $q >n$. Then  by Lemma \ref{dual-char}, Proposition \ref{existence-qweak} (when $n'<q <2$), Lemma \ref{th4-higher integrability} (when $2 \le q \le n$), and Lemma \ref{th4-boundedness} (when $q>n$),   we can  deduce   that
$$
 v \in L^{\overline{q}}(\Omega_{R}) \quad\mbox{and}\quad \|v\|_{L^{\overline{q}}(\Om_{R})} \le C(n,  q,  \Om , R ,  \eta ,  \mathbf{b}  )  \|v \|_{L^{q} (\Om_{2R} )} .
$$
Hence  by a bootstrap argument starting from  $q =p_1$, we can show that
$$%
v \in L^{\infty}(\Om_{R}) \quad\mbox{and}\quad    \|v\|_{L^{\infty}(\Om_{R})} \le C(n,   p_1 , \Om ,  R ,  \mathbf{b}  )  \|v \|_{L^{p_1} (\Om_{2R} )}
$$%
for all $R>R_0$. Since $v \in L_{loc}^2 (\overline{\Om})$, it   follows from Proposition \ref{existence-qweak} that   $v \in W_{\loc}^{1,2}(\overline{\Omega}) $.
Adapting the proof of Lemma \ref{th4-unique-q version}, we can also deduce that
$v \in L^p (\Om )$ for any $p \in [p_2 , \infty ]$.
Moreover, it can be shown that if $w =v^{+}$ and $\beta \ge 1$, then
$$
\int_{\Om} \nb v \cdot \nb \left(\eta_k^2 w^{\beta} \right) \, dx = \int_{\Om} \left(\eta_k^2 w^{\beta} \mathbf{b} \right)  \cdot \nb v \, dx ,
$$
where  $\eta_k$ is  the   cut-off function   in the proof of Lemma \ref{th4-unique-q version}.
Note that
\begin{align*}
\int_{\Om} \left(\eta_k^2 w^{\beta} \mathbf{b} \right)  \cdot \nb v  \, dx &= \int_{\Om} \left(\eta_k^2 w^{\beta} \mathbf{b} \right)  \cdot \nb w  \, dx \\
&= \frac{1}{\beta +1} \int_{\Om} \mathbf{b} \cdot \left[   \nb \left( \eta_k^2 w^{\beta +1} \right)  - 2 \eta_k w^{\beta +1} \nabla \eta_k \right] \, dx\\
 & \le -\frac{2 }{\beta +1} \int_{\Om} \eta_k w^{\beta +1} \nabla \eta_k \cdot \mathbf{b} \, dx.
\end{align*}
Hence following exactly   the same argument as in the proof of    Lemma \ref{th4-unique-q version}, we can complete the proof.
\end{proof}

\subsection{Proofs of Theorems \ref{th4-q version} and \ref{th4-q version-strong} for exterior domains}

To prove Theorem \ref{th4-q version} for the case when $\Omega$ is exterior and $2 \le p<n$, it suffices to prove the following result.

\begin{proposition}\label{existence-qweak-exterior} Let $\Om $ be an exterior $C^1$-domain in $\R^n$, $n \ge 3$. Assume that $\mathbf{b} \in L^{n,\I}(\Om;\R^n)$, $\div \mathbf{b} \ge 0$ in $\Om$, and $2 \le p<n$.
\begin{enumerate}[(i)]
\item For each ${\mathbf F} \in L^{p}(\Om; \R^n )$, there exists a unique $p$-weak solution $u $ of \eqref{bvp} with $f= \div {\mathbf F}$. Moreover, we have
$$
 \|\nb u \|_{L^p (\Om )} \le C (n,p ,\Om ) \left( 1+ \|{\mathbf b}\|_{L^{n,\I} (\Om )} \right) \|{\mathbf F}\|_{L^{p}(\Om )}.
$$
\item For each ${\mathbf G} \in L^{p'}(\Om; \R^n )$, there exists a unique $p'$-weak solution $v $ of \eqref{bvp-dual} with $g= \div {\mathbf G}$. Moreover, we have
$$
 \|\nb v \|_{L^{p'} (\Om )} \le C (n,p ,\Om ) \left( 1+ \|{\mathbf b}\|_{L^{n,\I} (\Om )} \right) \|{\mathbf G}\|_{L^{p'}(\Om )}.
$$
\end{enumerate}
\end{proposition}

\begin{proof}
 For the case $p=2$, the theorem was already proved in  Lemma \ref{th4-existence}.

Assume that $2 < p < n$ and ${\mathbf F} \in L^{2}(\Om; \R^n ) \cap L^{p}(\Om; \R^n )$. Then by Lemma \ref{th4-existence} and Proposition \ref{th4-existence-L1}, there exists a unique weak solution $u $ of \eqref{bvp} with $f=\div {\mathbf F}$ satisfying
$\|u\|_{L^{p^*} (\Om )} \le C(n,p ) \|{\mathbf F}\|_{L^{p}(\Om )}.$
Since $-\De u = \div \left(\mathbf{F}- u \mathbf{b} \right)$ in $\Om$, it follows from Lemmas \ref{holder inequalies}, \ref{uniqueness in Lorentz spaces for exterior domains}, and Proposition \ref{CZ-estimates in exterior domains} that $\nb u \in L^{p,\I}(\Om;\R^n)$ and
\begin{equation}\label{existence-qweak-exterior-1}
 \|\nabla u \|_{L^{p,\I}(\Om )} \le C(n,p,\Omega ) \left( 1+ \|{\mathbf b}\|_{L^{n,\I} (\Om )} \right) \|{\mathbf F}\|_{L^{p}(\Om )}.
\end{equation}
Hence by a standard density argument, we deduce that for each ${\mathbf F} \in L^{p}(\Om; \R^n )$, there exists a unique solution $u = T_p ({\mathbf F} ) \in D_0^{1,p,\infty}(\Om ) \cap L^{p^*} (\Om ) $ of \eqref{bvp} with $f=\div {\mathbf F}$ satisfying \eqref{existence-qweak-exterior-1}.
Uniqueness of such a solution follows from Lemma \ref{th4-unique-q version}.

Choose any $p_1 , p_2$ with  $2 < p_1 < p < p_2 < n$. For each $i=1,2$, we  define $S_i (\mathbf{F})= \nb T_{p_i} (\mathbf{F})$ for all  $\mathbf{F} \in L^{p_i}(\Om; \R^n )$. Then $S_i$ is a bounded linear operator from $L^{p_i}(\Om; \R^n )$ into $L^{p_i, \infty }(\Om; \R^n )$ and its operator norm is bounded above by $C(n,p_i ,\Omega ) \left( 1+ \|{\mathbf b}\|_{L^{n,\I} (\Om )} \right) $.
Moreover,  it follows from Lemma \ref{th4-unique-q version} that $S_1 =S_2$ on $L^{p_1}(\Om; \R^n ) \cap L^{p_2}(\Om; \R^n )$. Hence the operators $S_i$ can be extended uniquely to a linear operator $S$ from $L^{p_1}(\Om; \R^n ) + L^{p_2}(\Om; \R^n )$ into $L^{p_1, \I}(\Om; \R^n ) + L^{p_2 , \I}(\Om; \R^n )$.
Therefore, by the Marcinkiewicz interpolation theorem (Lemma \ref{marcinkiewicz-inter}),  we conclude that $S $ is a bounded linear operator from $L^{p}(\Om; \R^n ) $ into $L^{p}(\Om; \R^n ) $ and its operator  norm is bounded above by $C(n,p,\Omega ) \left( 1+ \|{\mathbf b}\|_{L^{n,\I} (\Om )} \right)$. This proves Part (i) of the proposition.
Part (ii) can be proved by using the same duality argument as in the proof of Proposition \ref{existence-qweak}.
\end{proof}

\begin{proof}[Proof of Theorem \ref{th4-q version} for exterior domains]
For the case $2 \le p <n$, Part (ii) of the theorem immediately follows from Proposition \ref{existence-qweak-exterior}. Moreover, for the remaining case $ n'<p<2$, Part (ii)  can be proved by adapting the proof of Proposition \ref{existence-qweak-exterior} but using Lemma \ref{th4-unique-q version-dual} instead of Lemma \ref{th4-unique-q version}. This completes the proof of Part (ii) of the theorem. Part (i) then follows from Part (ii) and Lemmas \ref{th4-unique-q version} and \ref{th4-unique-q version-dual} by a duality argument.
\end{proof}

\medskip

\begin{proof}
[Proof of Theorem \ref{th4-q version-strong} for exterior domains]
Choose any $q_1 , q_2$ such that $1<q_1 < q<q_2 <n/2$.
Let $i=1,2$ be fixed. Then since $n/(n-1)< q_i^* < n$, it follows from Theorem \ref{th4-q version} that for each $\mathbf{F} \in L^{q_i^*}(\Om ; \R^n )$ there exists a unique $q_i^*$-weak solution $u =T_{q_i^*}({\mathbf F})$ of \eqref{bvp} with $f = \div \mathbf{F} $, satisfying $\|\nabla u \|_{L^{q_i^*}(\Om )} \le C(n,q_i , \Om )\|\mathbf{F} \|_{L^{q_i^*}(\Om )}$. Given $f \in L^{q_i}(\Om ) $, let $\mathbf{F} =-\nb N(f)$ and $u = T_{q_i^*}({\mathbf F})$, where $N(f)$ is the Newtonian potential of $f$ over $\Om$. Then it follows from the Calderon-Zygmund theory that $\mathbf{F} \in L^{q_i^*}(\Om )$, $\|\mathbf{F} \|_{L^{q_i^*}(\Om )} \le C \|f\|_{L^{q_i}(\Om )}$, and $f =\div \mathbf{F} $ in $\Om$. Moreover, since $-\De u = f- \mathbf{b} \cdot \nb u - \left( {\rm div}\, {\mathbf b} \right) u $ in $\Om$,
it follows from Lemma \ref{holder inequalies} and Proposition \ref{CZ-estimates in exterior domains} %
 that $\nabla^2 u \in L^{q_i , \I}(\Om;\R^{n^2})$ and
$ \|\nabla^2 u \|_{L^{q_i ,\I}(\Om)} \le C(n,q_i , \Om )  M_{\mathbf b}^2 \|f\|_{L^{q_i }(\Om )} $.

Let $\alpha$ be a multi-index with $|\alpha |=2$. For each $i=1,2$, we now define $\tilde{T}_i (f)= D^\alpha T_{q_i^*}( \nb N(f))$ for all $f \in L^{q_i}(\Om ) $.
Then by the uniqueness assertion of Theorem \ref{th4-q version}, we deduce that $\tilde{T}_1 = \tilde{T}_2$ on $L^{q_1 }(\Om ) \cap L^{q_2 }(\Om )$. Hence there exists a unique linear operator $\tilde{T}$ from $ L^{q_1}(\Om ) + L^{q_2}(\Om )$ into $ L^{q_1,\I}(\Om ) + L^{q_2,\I}(\Om )$
that extends both $\tilde{T}_1$ and $\tilde{T}_2$. Moreover, since $\tilde{T}_i$ is bounded from $ L^{q_i}(\Om ) $ into $L^{q_i , \I}(\Om) $ with norm bounded above by $C(n,q_i , \Om ) M_{\mathbf b}^2 $, it follows from the Marcinkiewicz interpolation theorem (Lemma \ref{marcinkiewicz-inter}) that $\tilde{T}$ is bounded from $ L^{q}(\Om ) $ into $L^{q} (\Om ) $ and its norm is bounded above by $C(n,q , \Om ) M_{\mathbf b}^2$.
This proves Part (i) of the theorem. Part (ii) can be proved by the same argument.
 \end{proof}

\section{Proofs of Theorems \ref{th4}, \ref{th5}, and \ref{th6}}
\label{Sec7}

 In this section, we shall first prove Theorem \ref{th4} on the  further regularity  of weak solutions of the dual problem \eqref{bvp-dual}, by using the H\"{o}lder regularity result in Proposition \ref{th4-0}. Then uniqueness of very weak solutions of \eqref{bvp-dual} will be established by a duality argument (Theorems \ref{th5} and \ref{th6}).
The domain $\Om$  is a bounded or exterior domain in $\R^n , n \ge 3$, of class $C^{1,1}$.

\begin{proof}[Proof of Theorem \ref{th4}]
We first prove the $W^{1,n+\e}$-regularity result (i). Assume that $g \in W^{-1,p}(\Om)$ and $n<p<\infty$.
Then it follows from Proposition \ref{th4-0} that
\begin{equation}\label{holder estimate-2}
v \in C^\al(\overline{\Om}) \quad\mbox{and}\quad \|v\|_{C^\al(\overline{\Om})} \le C \|g\|_{W^{-1,p} (\Omega )}
\end{equation}
for some constants $0< \alpha \le 1-n/p$ and $C>0$, depending only on $n, p, \Omega $ and $ \|{\mathbf b}\|_{L^{n,\I}(\Om )} $. Fix any $r > 1$ such that
$$
\max\left\{ \frac{(1-\al)n}{2-\alpha}, \frac n{2(n-1)}\right\} < r< \frac n2 .
$$
Then since $g \in W^{-1,p}(\Omega ) \subset W^{-1,2r}(\Omega )$ and $n/(n-1)< 2r < n$,
it follows from Theorem \ref{th4-q version} that
$$
v \in W_0^{1,2r}(\Omega )\quad\mbox{and}\quad
\norm{\nabla v}_{L^{2r}(\Omega )} \le C \|g\|_{W^{-1,p} (\Omega )}.
$$
Moreover, since $ |\mathbf{b}| \in L^{n,\infty}(\Omega ) \subset L^{2r} (\Omega )$, we have
$$
 \| \mathbf{b} \cdot \nabla v \|_{L^{r} (\Omega ) } \le \|\mathbf{b}\|_{L^{2r}(\Omega ) } \|\nabla v\|_{L^{2r} (\Omega )} \le C \|\mathbf{b}\|_{L^{n,\infty}(\Omega ) } \|g\|_{W^{-1,p} (\Omega )}.
$$
Hence by Lemma \ref{CZ-estimates0},
there exist $v_1 \in W_0^{1,p}(\Omega )$ and $v_2 \in W_0^{1,r}(\Om ) \cap W^{2, r }(\Om)$ such that
$$
v=v_1 +v_2 , \quad -\De v_1 =g , \quad -\De v_2 = \mathbf{b} \cdot \nabla v,
$$
and
$$%
 \| v_1 \|_{W^{1,p}(\Om )} + \| v_2 \|_{W^{2,r }(\Om)} \le C \|g\|_{W^{-1,p} (\Omega )}.
$$%
Since $\alpha \le 1-n/p$, it follows from (\ref{holder estimate-2}) and the Morrey embedding theorem that
$$
\| v_2 \| _{C^\al (\overline{\Omega})} \le \| v \| _{C^\al (\overline{\Omega})} + C \| v_1 \|_{W^{1,p}(\Om )} \le C \|g\|_{W^{-1,p} (\Omega )}.
$$
Hence by the Miranda-Nirenberg inequality (Lemma \ref{MNineq}),
$$
\norm{\nb v_2}_{L^s(\Omega )} \le C \left( \| v_2 \|_{W^{2,r }(\Om)}+ \norm{v_2
}_{C^\al (\overline{\Omega})} \right)\le C \norm{g}_{W^{-1,p} (\Omega ) } ,
$$
where $s =\frac {(2-\al)r}{1-\al}$. Note that $s>n$.
Therefore, taking
$$
\e = \min \{ p ,s \} -n >0 ,
$$
 we conclude that
\begin{equation}\label{gradient Lr-estimate}
 \nabla v \in L^{n+\e}(\Om ; \R^n ) \quad\mbox{and}\quad \|\nb v \|_{L^{n+\e}(\Om)} \le C \norm{g}_{W^{-1,q} (\Omega ) }.
\end{equation}

To prove the $W^{2,n/2+\delta}$-regularity result (ii), we assume that $g \in L^q (\Omega )$ and $n/2 < q<n$. Then it follows from the Sobolev embedding theorem that $g \in W^{-1, q^*}(\Om )$ and $n<q^*< \infty$. It was already shown that $v \in W^{1,n+\e}(\Om)$.
Since $ |\mathbf{b}| \in L^{n,\infty}(\Omega ) $, it follows that
$$
\mathbf{b}\cdot \nabla v \in L^s (\Omega ) \quad\mbox{for any}\,\, s< \frac{n(n+\e )}{2n+\e}.
$$
Hence recalling again that $-\De v =g + \mathbf{b}\cdot \nabla v$ in $\Om$, we deduce from Lemma \ref{CZ-estimates} that
$$
v \in W^{2,s}(\Omega ) \quad\mbox{for any}\,\, s< \min \left\{\frac{n(n+\e )}{2n+\e} , q \right\}.
$$
The proof is complete.
\end{proof}

\begin{proof}[Proof of Theorem \ref{th5}] Let $r =(n/2+\delta)'$ be the H\"{o}lder conjugate of $ n/2+\delta $, where $\delta$ is the constant in Theorem \ref{th4} (ii) corresponding to $q=n/2 +1$.
Since $0< \delta \le 1$, it follows that $n/2 < r' < n$ and $n ' < r < (n/2)'$.

Suppose that $u \in L^{r} (\Omega )$ satisfies (\ref{eq1.3-0}). To show that $u=0$ on $\Om$, let $g \in C_c^\infty (\Om )$ be given. Then by Theorem \ref{th4},
there exists $\phi \in W^{1,r'}_0(\Om)\cap W^{2,r'}(\Om)$ such that
$-\De \phi - \bb \cdot \nb \phi = g$ in $\Om $.
Choose a sequence $\{ \phi_k \}$ in $ C^2 (\overline{\Om})$ such that $\phi_k =0$ on $\partial\Om$ and $\phi_k \to \phi$ in $W^{2,r'}(\Om)$. Then by Lemma \ref{th1}, we have
\begin{align*}
\int_\Omega u g \, dx & = - \int_\Om u \bke{ \De \phi + \mathbf{b} \cdot \nb
\phi} \, dx \\
 & =- \lim_{k \to \I} \int_\Om u \bke{ \De \phi_k + \mathbf{b} \cdot \nb
\phi_k } \, dx = 0.
\end{align*}
 Since $g \in C_c^\infty (\Om )$ is
arbitrary, it follows that $u=0$ in $\Omega$. This completes the proof of Part (i) of the theorem.

To prove Part (ii), let $f \in W^{-1, n'-}(\Om)$ be given. %
Choose any $p$ such that $1<p< n'$ and $r < p^* < (n')^*  = (n/2)'$. Then since $n/2 < (p^* )' < n$, it follows from Theorem \ref{th4} that there exists $n/2 < s < (p^* )' $ such that for every $g \in L^{(p^* )'}(\Om )$ there exists a unique weak solution $v =L g$ in $W_0^{1,s}(\Om ) \cap W^{2,s}(\Om )$ of \eqref{bvp-dual}. By Theorem \ref{th4} and the Sobolev embedding theorem, we also deduce that $L$ is a bounded linear operator from $L^{(p^* )'}(\Om )$ into $W_0^{1,s^*}(\Om ) $.
Moreover, since $  (s^* )' < n'$, it follows that
$f \in W^{-1,  (s^* )'}(\Om)$. Hence  the mapping $g \mapsto \langle f , Lg \rangle$ is a bounded linear functional on $L^{(p^* )'}(\Om )$. Therefore, by the Riesz representation theorem, there exists a unique $u \in L^{p^* }(\Om )$ such that
\[
\int_\Omega u g \, dx = \langle f , Lg \rangle \quad\mbox{for all} \,\, g \in L^{(p^* )'}(\Om ).
\]

Now, given $\phi \in C^2 (\overline{\Om})$ with $\phi |_{\partial \Om}=0$, we take $g = -\De \phi - b \cdot \nb \phi$. Then since $\phi =Lg $, we have
\[
- \int_\Om u \bke{ \De \phi + \mathbf{b} \cdot \nb
\phi} \, dx = \langle f , \phi \rangle .
\]
Hence it follows from Part (i) that $u$ is a unique very weak solution in $L^{p^* }(\Om )$ of \eqref{bvp}.
On the other hand, by Lemma \ref{holder inequalies} and Proposition \ref{CZ-estimates}, there exists a unique $\overline{u} \in W_0^{1,p, \infty}(\Om )$ such that
\[
- \int_\Om \overline{u} \De \phi
 \, dx = \langle f , \phi \rangle + \int_\Om ( u \mathbf{b} ) \cdot \nb
\phi \, dx
\]
for all $\phi \in C_c^1 (\Om )$. Note that $w=u-\overline{u}$ is a very weak solution in $L^r (\Om )$ of the Laplace equation in $\Om$ with trivial data. Hence by duality, it follows from Lemma \ref{CZ-estimates0} that $u =\overline{u} \in W_0^{1,p, \infty}(\Om )$. This completes the proof of Part (ii), because $p$ can be arbitrarily close to $n'$.
\end{proof}

\medskip

\begin{proof}[Proof of Theorem \ref{th6}] By Lemma \ref{th4-unique-q version}, it suffices to show that $u \in W_{\loc}^{1,r}(\overline{\Omega})$ for some $r>n' $ and $u=0$ on $\partial \Om$.
Choose $\eta \in C_c^\infty (B_{2}; [0,1])$ such that $\eta =1$ on $B_{1}$. For $R>0$, we define $\eta_R (x) = \eta (x/R)$. Fix $R>0$ so large that $\Omega^c \subset B_{R } =B_R(0)$, and define $\overline{u}=\eta_R u$. Then it follows from the condition (i) that $ \overline{u}\in L^{n/(n-2), \I}(\Omega_{2R})
$. Moreover, by a direct calculation,
$$
\int_{\Om_{2R}} \overline{u} \bke{ \De \phi + \mathbf{b} \cdot \nb
\phi} \, dx = - \int_{\Om_{2R}} u \bke{ \phi \De \eta_R + 2\nb \eta_R \cdot \nb \phi + \phi \mathbf{b} \cdot \nb \eta_R } \, dx
$$
for all $\phi \in C^2 (\overline{\Om}_{2R})$ with $\phi|_{\partial\Omega_{2R}}=0 $. Choose any number $r$ such that $ n'  <r< (n/2)' =(n')^*$. Then by Lemma \ref{th1-general}, we have
$$
 \int_{\Om_{2R}} \left| u \bke{ \phi \De \eta_R + 2\nb \eta_R \cdot \nb \phi + \phi \mathbf{b} \cdot \nb \eta_R } \right|\, dx \le C \|u \|_{L^r (\Om_{2R} )}\| \phi \|_{W^{1,r'} (\Om_{2R} )}
$$
for all $\phi \in W_0^{1,r'}(\Om_{2R} )$, where $C= C(r, \eta_R , \mathbf{b})$. Hence by Theorem \ref{th4-q version}, there exists a unique $w\in W_0^{1,r}(\Om_{2R} )$ such that
$$
\int_{\Om_{2R}} w \bke{ \De \phi + \mathbf{b} \cdot \nb
\phi} \, dx = - \int_{\Om_{2R}} u \bke{ \phi \De \eta_R + 2\nb \eta_R \cdot \nb \phi + \phi \mathbf{b} \cdot \nb \eta_R} \, dx
$$
for all $\phi \in C^2 (\overline{\Om}_{2R})$ with $\phi|_{\partial\Omega_{2R}}=0 $. Since
$W_0^{1,r}(\Om_{2R} ) \subset L^{r^*}(\Omega_{2R})$ and $r^*> (n')^*= n/(n-2)$, it follows that $w \in L^{n/(n-2), \I}(\Omega_{2R})$. Therefore, by Theorem \ref{th5}, we deduce that $\overline{u}=w$ on $\Omega_{2R}$. This proves that $u \in W^{1,r}(\Omega_{2R})$ and $u=0$ on $\partial \Om$.
Since $R>0$ is arbitrarily large, it follows that $u \in W_{\loc}^{1,r}(\overline{\Omega}) $.
\end{proof}

\section*{Acknowledgments}
Part of this work was done when Tsai visited the Center of Advanced Study in Theoretical Sciences (CASTS) of National Taiwan University, and the Center of Mathematical Sciences and Applications (CMSA) of Harvard University.
Their kind hospitality and supports are gratefully acknowledged.

\end{document}